%% file: main.tex
\documentclass[11pt,reqno]{amsart}
\pdfoutput=1 % needed for arXiv
\usepackage{preprintStyle}
\usepackage{makecell}

\input{def}

\title{KLAP: KYP lemma based low-rank approximation for $\mathcal{H}_2$-optimal passivation}
\author[J.~Nicodemus \and M.~Voigt \and S.~Gugercin \and  B.~Unger]
{Jonas Nicodemus${}^\star$ \and Matthias Voigt$^{\ddagger}$ \and Serkan Gugercin$^{\dagger}$ \and Benjamin Unger${}^\mathsection$}

\address{${}^{\star}$ Stuttgart Center for Simulation Science (SC SimTech), University of Stuttgart, Universit\"{a}tsstr.~32, 70569 Stuttgart, Germany}
\email{jonas.nicodemus@simtech.uni-stuttgart.de}

\address{${}^{\ddagger}$ Faculty of Mathematics and Computer Science, UniDistance Suisse, Schinerstr. 18, 3900 Brig, Switzerland}
\email{matthias.voigt@fernuni.ch}

\address{${}^{\dagger}$ Department of Mathematics and Division of Computational Modeling and Data Analytics, Academy of Data Science, Virginia Tech, Blacksburg, VA 24061, United States}
\email{gugercin@vt.edu}

\address{${}^{\mathsection}$ Institute for Applied and Numerical Mathematics, Karlsruhe Institute of Technology, 76131
Karlsruhe, Germany}
\email{benjamin.unger@kit.edu}
\date{\today}
\keywords{LTI systems, passivity, passivation, KYP lemma, $\calH_2$-optimization, gradient-based optimization}

\begin{document}

\begin{abstract}
We present a novel passivity enforcement (passivation) method, called KLAP, for linear time-invariant systems based on the Kalman-Yakubovich-Popov (KYP) lemma and the closely related Lur'e equations. 
The passivation problem in our framework corresponds to finding a perturbation to a given non-passive system that renders the system passive while minimizing the $\mathcal{H}_2$ or frequency-weighted $\mathcal{H}_2$ distance between the original non-passive and the resulting passive system.
We show that this problem can be formulated as an unconstrained optimization problem whose objective function can be differentiated efficiently even in large-scale settings. 
We show that any minimizer of the unconstrained problem yields the same passive system. Furthermore, we prove that, in the absence of a feedthrough term, every local minimizer is also a global minimizer. For cases involving a non-trivial feedthrough term, we analyze global minimizers in relation to the extremal solutions of the Lur'e equations, which can serve as tools for identifying local minima.
To solve the resulting numerical optimization problem efficiently, we propose an initialization strategy based on modifying the feedthrough term and a restart strategy when it is likely that the optimization has converged to a non-global local minimum. 
Numerical examples illustrate the effectiveness of the proposed method.
\end{abstract}

\maketitle
{\footnotesize \textsc{Keywords:} LTI systems, passivity, passivation, KYP lemma, $\calH_2$-optimization, gradient-based optimization}

{\footnotesize \textsc{AMS subject classification:} 37J06, 37M99, 65P10, 93A30, 93B30, 93C05}
%
%30E05 = Moment problems and interpolation problems in the complex plane
%37J06 = General theory of finite-dimensional Hamiltonian and Lagrangian systems, Hamiltonian and Lagrangian structures, symmetries, invariants
%37M99 = Approximation methods and numerical treatment of dynamical systems
%65P10 = Hamiltonian systems including symplectic integrators (numerical problems in dynamics systems)
%93A15 = Large-scale systems
%93A30 = Mathematical modeling (models of systems, model-matching, etc.)
%93B15 = Realizations from input-output data
%93B30 = System identification
%93B99 = Systems theory
%93C05 = Linear systems (Control systems)
%-----------------------------------------------------------------------------%

\section{Introduction}
\label{sec:intro}
Numerous prominent examples from various fields, including electrical circuits \cite{Fre00, Fre11}, power systems \cite{GriG16}, mechanical systems \cite{GPBS12}, and poroelasticity \cite{AMU21} can effectively be modeled as passive systems.
Moreover, energy-based modeling techniques \cite{Sch07, MU23} result in port-Hamiltonian systems, which are automatically passive. Having a passive system is often crucial, for instance, for obtaining physically meaningful simulation results. Even further, passive systems are useful as building blocks of larger network models since power-preserving interconnections of passive systems result in an overall passive system. This allows the structure-preserving coupling of models from different physical domains and of varying scales, and enables the use of passivity-based control methodologies~\cite{Sch17}. 
However, even if a physical process is known to be passive, models are often obtained from potentially unstructured model reduction methods or by data-driven system identification techniques. Even though there exist passivity-preserving model reduction \cite{BU22, SV23, DP84, Ant05, GPBS12} and structured system identification methods \cite{MNU23, BGV24, BGV20, ReiGG23, GJT24}, unstructured techniques can offer several benefits, including simplicity, the availability of well-established numerical algorithms and software implementations, or the existence of simple error bounds. Unfortunately, such methods can result in non-passive models. Therefore, a post-processing step is desired to restore passivity by perturbing the model. Such a perturbation is usually acceptable if the perturbation error is small, e.g., in the order of the model reduction error.

In particular, the problem setting is the following: Consider a linear time-invariant dynamical system of the form
\begin{equation}
	\label{eqn:LTI}
	\system:\quad \begin{cases}
		\dot{\state}(t) = A\state(t) + B\inpVar(t),\\
		\outVar(t) = C\state(t) + D\inpVar(t)
	\end{cases}
\end{equation}
and assume that $\Sigma$ is asymptotically stable (i.e., $A$ is Hurwitz, meaning all eigenvalues of $A$ have negative real part) with 
$A\in\R^{\stateDim\times\stateDim}$, $B\in\R^{\stateDim\times\inpVarDim}$, $C\in\R^{\inpVarDim\times\stateDim}$, and $D\in\R^{\inpVarDim\times\inpVarDim}$. Moreover, our standing assumption is that~\eqref{eqn:LTI} is not passive. We are interested in the \emph{passivation} of~\eqref{eqn:LTI}, i.e., our goal is to find a modified system 
\begin{equation}
	\label{eqn:LTI:modifiedC}
	\systemPas\big(\reduce{C}\big): 
	\quad \begin{cases}
		\dot{\state}(t) = A\state(t) + B\inpVar(t),\\
		\outVar(t) = \reduce{C}\state(t) + D\inpVar(t)
	\end{cases}
\end{equation}
with $\reduce{C}\in\R^{\inpVarDim\times\stateDim}$  
such that $\systemPas(\reduce{C})$ 
is passive and minimizes the distance to~\eqref{eqn:LTI}\footnote{As we explain in \Cref{sec:passivation},  our choice for the distance requires retaining the same feedthrough term $D$, in the passivated model  $\systemPas\big(\reduce{C}\big)$.}, in other words, we want to solve the constrained optimization problem
\begin{equation}
	\label{eqn:passivationProblem}
	\min_{\reduce{C}\in\R^{\inpVarDim\times\stateDim}} \big\|\system-\systemPas\big(\reduce{C}\big)\big\| \qquad\text{such that}\qquad \systemPas(\reduce{C}) \text{ is passive.}
\end{equation}
We will refer to this problem as the \emph{passivation problem.}
For the norm, we rely on the Hardy spaces, i.e., we compare the difference of the transfer functions of the systems $\system$ and $\systemPas(\reduce{C})$, given by
\begin{align}
	\label{eqn:transferFunction}
	\transferFunction(s) &= C(s I_{\stateDim}-A)^{-1}B + D, &
	\transferFunctionPas\big(s;\reduce{C}\big) &= \reduce{C}(s I_{\stateDim} - A)^{-1}B + D.
\end{align}

We emphasize that the constraint in the optimization problem~\eqref{eqn:passivationProblem} can be written as a linear matrix inequality with an additional $\stateDim^2$ decision variables, yielding a total of $\stateDim^2 + \inpVarDim\stateDim$ decision variables; see \cite[Cha.~5.5.1]{GriS21}. Hence, directly aiming for an efficient numerical method to solve~\eqref{eqn:passivationProblem} for $\stateDim\geq 100$ is challenging with standard methods; see \cite{GU08}. Our main idea is to apply the \emph{Kalman-Yakubovich-Popov} (\KYP) lemma in low-rank factorized form to obtain an explicit parameterization of $\reduce{C}$ that renders $\systemPas\big(\reduce{C}\big)$ passive.
We call the resulting proposed method \emph{\ourmethodfull}.

Our main contributions are the following:
\begin{enumerate}
	\item Exploiting the existence of rank-minimizing solutions of the \KYP inequality, we obtain an explicit parameterization of any passive system with the $\stateDim\inpVarDim$ decision variables, which is precisely the number of unknowns in $\reduce{C}$ in \Cref{thm:passParameterization}.
	\item We then focus on the Hardy $\calH_2$-norm as the distance measure in \eqref{eqn:passivationProblem}, use the parameterization from \Cref{thm:passParameterization} to reformulate the convex constrained optimization problem~\eqref{eqn:passivationProblem:H2:constrained} as a non-convex unconstrained optimization problem~\eqref{eqn:passivationProblem:H2}, and establish solvability, uniqueness, and gradient computations; see \Cref{thm:H2passivation} and \Cref{thm:H2solvability}. Moreover, for the case of a skew-symmetric feedthrough term, i.e., $D + D^\T= 0$ in~\eqref{eqn:passivationProblem}, we show in \Cref{prop:locGlobMin} that any local minimizer is also a global minimizer.
	\item If $D+D^\T \neq 0$, in general, local minimizers might not yield the global minimum. We thus provide a novel criterion to check whether a local minimizer is a global minimizer using the extremal solutions of the \KYP inequality in \Cref{thm:locmin}. We demonstrate that the criterion can be checked without computing the extremal solutions and present a restart strategy~\eqref{eqn:gradientDescent}. To avoid getting trapped in a non-global local minimum in the first place, we propose an initialization strategy in \Cref{sec:passivation:init}, which relies on a perturbation of the feedthrough term.
\end{enumerate} 
We demonstrate \ourmethod on three numerical examples in \Cref{sec:numerics}, including a challenging benchmark system of dimension $\stateDim = 800$ with four inputs and outputs obtained from system identification of a high-speed smartphone interconnect link taken from \cite{GriS21}.

\subsection{Literature review and state-of-the-art}
\label{sec:intro:literature}
% KYP lemma
Key foundational work on dissipative systems can be traced back to Willems\cite{Wil72, Wil72a}, who introduced a formal framework for passive systems and their relation to Lyapunov stability. Willems defined dissipative systems in terms of a supply rate and a storage function.
A comprehensive treatment of passivity and $L_2$ gain methods is given in~\cite{Sch17}, establishing links between dissipativity and control design. A recent overview of dissipativity theory is given in~\cite{BLME20}, including a detailed chapter about the \KYP lemma and the Lur'e equations~\cite[Cha.~3]{BLME20}. 
% passivation
A survey of passivation methods is given in \cite{GriS21, GriG16}; see also~\cite{GU08} for a comparative study of the methods. Passivation methods can be divided into three categories:
Methods in the first category such as \cite[Sec.~10.7]{GriG16} or \cite{FaePCMR21} are based on the \KYP lemma and result in a constrained optimization problem, where, usually, the goal is to find the smallest perturbation on the output matrix in a weighted norm such that a solution to the \KYP inequality exists. The problem can be formulated as a standard \LMI and is known to be convex. Unfortunately, the computational cost scales poorly with the system size due to the presence of a Lyapunov matrix (in our manuscript $X$ in \eqref{eqn:KYP}), which is instrumental for the \KYP lemma. The trace parameterization methods \cite{Dum02,CoePS04} try to overcome this issue by eliminating the Lyapunov matrix from the optimization problem. 
% However, this approach requires the initial solution of many Lyapunov equations and storage of the solutions. \ourmethod is similar to the trace parametrization methods but bypasses the need for initially solving Lyapunov equations. 
However, this approach first requires solving many Lyapunov equations and storing those solutions. \ourmethod is similar to the trace parametrization methods but bypasses the need for this initial stage of solving  many Lyapunov equations. 
Instead, only two Lyapunov equations are solved in each iteration of the optimization procedure.
Methods of the second category are based on the characterization of passivity in terms of the spectrum of a Hamiltonian matrix, see ~\cite{Gri04,GriU06}. The idea consists of iteratively perturbing the Hamiltonian matrix until the perturbed Hamiltonian matrix has no imaginary eigenvalues, indicating passivity of the perturbed model. 
However, this method has no convergence guarantees. Moreover, imaginary eigenvalues may be reintroduced during the iteration and thus, it may take many iterations to remove all imaginary eigenvalues, making the method potentially inefficient. Various improvements of this approach were considered, e.g., the use of structure-preserving methods for the Hamiltonian eigenvalue problem~\cite{SchS07} or the extension to wider problem classes such a descriptor systems with general supply rates \cite{BruS13}.
A related approach to finding the closest passive system is devised in~\cite{FazGL21}. Therein, a low-rank \textsf{ODE} is proposed to find an optimal low-rank perturbation in the sense of a weighted Frobenius norm of the difference of the output matrices.
Finally, the methods of the third category are based on passivity enforcement at discrete frequencies via linear or quadratic programming~\cite{SAN05, SAN07, GS01}.
Consequently, methods from this category only guarantee passivity in a certain frequency range.
Instead of finding the nearest passive system, one may also look for the nearest port-Hamiltonian system. This problem is studied in~\cite{GS18}. There, the authors consider the Frobenius norm of the difference of the system matrices as a perturbation measure. However, this measure is not invariant under state-space transformations and depends on the particular realization under consideration. 

\subsection{Organization of the manuscript}
After this introduction, we recall fundamentals on passivity, positive realness, and the \KYP lemma in \Cref{sec:prelims}, which concludes with a parameterization of passive systems with~\Cref{thm:passParameterization}. Based on this parameterization, we establish an $\Htwo$-based optimization problem, derive the gradient of the objective functional, and discuss well-posedness in the first subsections of~\Cref{sec:passivation}. We also present strategies to overcome potential pitfalls in the optimization process in the latter part of~\Cref{sec:passivation}.
Finally, we demonstrate the effectiveness of \ourmethod on three numerical examples in~\Cref{sec:numerics}. Conclusions are given in~\Cref{sec:conclusions}.

%-----------------------------------------------------------------------------%
\subsection{Notation}
We use the symbols $\N$, $\R$, $\C$, $\C^+$, $\C^-$, $\R^n$, and $\R^{n\times m}$ to denote, respectively, the positive integers, the real numbers, the complex numbers, the complex numbers with positive real part, the complex numbers with negative real part, the set of column vectors with $n\in\N$ real entries, and the set of real $n\times m$ matrices. Furthermore, $A \succ 0$ and $A \succeq 0$ indicate that $A$ is, respectively, symmetric positive definite and symmetric positive semi-definite. We denote by $\Spd{n}$, and $\Spsd{n}$ the sets of real $n \times n$ symmetric positive definite and symmetric positive semi-definite  matrices, respectively. For a matrix $A \in \R^{n \times m}$, $\sigma(A)$ denotes the spectrum, $\tr(A)$ the trace, and $\|A\|_{\mathrm{F}}$ the Frobenius norm. Additionally, for some $W\in\Spd{m}$, we define the weighted inner product
\begin{equation}
	\label{eqn:weightedInnerProduct}
	\langle \cdot,\cdot \rangle_{W} \colon \R^{n\times m}\times \R^{n \times m} \to \R,\qquad (A,B) \mapsto \tr\big(B W A^\T\big)
\end{equation}and the associated induced norm ${\|A\|}_{W} = \sqrt{\tr\big(A W A^\T\big)}$.

\section{Preliminaries and parameterization of passivity}
\label{sec:prelims}

\subsection{Passive systems and positive real transfer functions}
\label{sec:prelims:passivity}
In this section, we review the theory on dissipative dynamical systems, which will be fundamental for our passivation method. The \emph{modified Popov function} of the system \eqref{eqn:LTI} is defined as
\begin{equation}
	\label{eqn:popovFunction}
	\Phi\colon\C\setminus \sigma(A) \to\C^{\inpVarDim\times \inpVarDim},\qquad 
	s\mapsto \transferFunction(s) + \transferFunction(s)\herm.
\end{equation}

\begin{definition}[Positive realness, passivity, storage function]\,
	\begin{enumerate}
		\item The transfer function $G$ in \eqref{eqn:transferFunction} is \emph{positive real} if $G$ has no poles in $\C^+$ and the modified Popov function $\Phi$ is Hermitian positive semi-definite in the open right half-plane, i.e.,
		\begin{equation*}
			 \Phi(\lambda) \succeq 0 \qquad \text{for all } \lambda \in \C^+.
		\end{equation*}
		\item The system~$\system$ in \eqref{eqn:LTI} is called \emph{passive}, if there exists a \emph{storage function} $\calH: \R^n \to [0,\infty)$ such that
		\begin{equation*}
			 \calH(x(t_2)) \le \calH(x(t_1)) + \int_{t_1}^{t_2} u(t)^\T y(t)\,\dt
		\end{equation*}
		for all solution trajectories $(u,x,y)$ of \eqref{eqn:LTI} and all $t_1 \le t_2$.
	\end{enumerate}
\end{definition}
If the system~\eqref{eqn:LTI} is minimal, then positive realness and passivity are equivalent; see~\cite{BLME20}. Positive realness can also be checked in terms of a condition on the imaginary axis. More precisely, the transfer function $G$ in \eqref{eqn:transferFunction} is positive real if and only if $G$ has no poles in $\C^+$,
\begin{equation*}
 \Phi(\i\omega) \succeq 0 \quad \text{for all } \i \omega \in \i\R \setminus \sigma(A),
\end{equation*}
and all the purely imaginary poles of $G$ are simple and the corresponding residue matrices are Hermitian and positive semi-definite \cite[Thm.~2.7.2]{AndV73}.
The \KYP operator associated with the system $\system$ in~\eqref{eqn:LTI} is given by
\begin{equation*}
    \calW_\system\colon \R^{\stateDim\times\stateDim}\to\R^{(\stateDim+\inpVarDim)\times(\stateDim+\inpVarDim)},\qquad X \mapsto 
    \begin{bmatrix}
        -A^\T X - XA & C^\T - XB\\
        C-B^\T X & D+D^\T
    \end{bmatrix}.
\end{equation*}
The associated \KYP inequality reads
\begin{equation}
    \label{eqn:KYP}
    \calW_\system(X) \succeq 0.
\end{equation}
For our forthcoming analysis, we draw on several results from the literature~\cite{LanT85, CamIV14, Wil72a} related to the \KYP inequality~\eqref{eqn:KYP}.
\begin{theorem}
    \label{thm:KYP}
    Consider the dynamical system~$\Sigma$ in~\eqref{eqn:LTI} and the associated \KYP inequality~\eqref{eqn:KYP}.
    \begin{enumerate}
        \item\label{thm:KYP:symPosSemi} If the system $\Sigma$ is asymptotically stable, then any solution $X\in\R^{\stateDim\times\stateDim}$ of~\eqref{eqn:KYP} is symmetric positive semi-definite.
        \item\label{thm:KYP:bounded} Assume that $\Sigma$ is controllable and passive. Then, there exist \emph{minimal and maximal solutions} $\Xmin,\,\Xmax \in \R^{n \times n}$ of \eqref{eqn:KYP} such that any solution $X \in \R^{n \times n}$ of \eqref{eqn:KYP} satisfies
        \begin{equation*}
            \Xmin \preceq X \preceq \Xmax.
        \end{equation*}
    \end{enumerate}
\end{theorem}
If $D + D^\T$ is nonsingular, then applying
the Schur complement to~\eqref{eqn:KYP} yields the \emph{algebraic Riccati equation} (\ARE)
\begin{equation}
	\label{eqn:kyp:are}
	A^\T X +XA + \big(C^\T - XB\big){\big(D+D^\T\big)}^{-1}\big(C-B^\T X\big) = 0.
\end{equation}
The connection between solutions of~\eqref{eqn:kyp:are} and the \KYP inequality are studied in great detail in~\cite{Wil71}. Numerical solvers for the \ARE are readily available and can be used to compute both \emph{minimal} and \emph{maximal solutions}, which are also the minimal and maximal solutions of the \KYP inequality from \Cref{thm:KYP}\,\ref{thm:KYP:bounded}. It should be noted that both extremal solutions are also \emph{rank-minimizing} solutions, meaning that $W_\Sigma(\Xmax)$ and $W_\Sigma(\Xmin)$ achieve that the smallest possible rank among all matrices $W_\Sigma(X)$ for solutions $X$ of the \KYP inequality \cite[Rem.~10]{Wil71}. 
It can be shown (for nonsingular $D + D^\T$) that this minimal rank equals $m$, the number of inputs and outputs; see the discussion following \cite[Thm.~10]{Wil72a}).
The extremal solutions have the additional property that $\sigma(Y_{\min}) \subseteq \C^- \cup \i\R$ and $\sigma(Y_{\max}) \subseteq \C^+ \cup \i \R$ with 
\begin{equation*}
	% \label{eqn:kyp:are:Y}
	\begin{aligned}
		Y_{\min} &\vcentcolon= A - B\big(D+D^\T\big)^{-1}\big(C-B^\T X_{\min} \big) \qquad\text{and}\\
		Y_{\max} &\vcentcolon= A - B\big(D+D^\T\big)^{-1}\big(C-B^\T X_{\max} \big).
	\end{aligned}
\end{equation*}
Hence, in the case of $\Xmin = \Xmax$, it holds $Y_{\min}=Y_{\max}$ and all eigenvalues of $Y_{\min}$ are on the imaginary axis.

We note that in case of a singular $D+D^\T$, instead of {\ARE}s, one could consider Lur'e equations and obtain similar statements as the ones above, see, e.g., \cite{Rei11}. We will however not make use of such results  and for this reason stay with the simpler analysis using {\ARE}s.

\subsection{Parameterization of passivity via the KYP lemma}
\label{sec:prelims:parameterization-passivity}
Since there are many variants available in the literature with different assumptions on controllability, we recall here the exact notion that we need in the sequel.
\begin{theorem}[{\KYP} lemma {\cite[Cha.~3.3]{BLME20}}, \cite{CGH24}\footnote{As presented, this result is actually a special version of a more general \KYP lemma that is called \emph{positive real lemma} in the literature. The \KYP lemma in the literature is mostly attributed to the second statement of this theorem.}] % p. 116
	\label{thm:KYPlemma}
    Consider the system $\system$ in \eqref{eqn:LTI} with the transfer function $G$ in \eqref{eqn:transferFunction}.
    \begin{enumerate}
    \item The system $\system$ is passive if and only if the Lur'e equations
	\begin{subequations}
		\label{eqn:Lure}
		\begin{align}
			\label{eqn:Lure:a} A^\T X + X A &= -LL^\T,\\
			\label{eqn:Lure:b} X B - C^\T &= -LM^\T,\\
			\label{eqn:Lure:c} D + D^\T &= M M^\T
		\end{align}
	\end{subequations}
	have a solution $(X,L,M)\in\Spsd{\stateDim}\times\R^{\stateDim\times\inpVarDim}\times\R^{\inpVarDim\times\inpVarDim}$. In particular, $\system$ is passive if and only if the \KYP inequality~\eqref{eqn:KYP} has a solution in $\Spsd{\stateDim}$.
    \item If the \KYP inequality~\eqref{eqn:KYP} has a solution in $\Spsd{\stateDim}$, then the transfer function $G$ is positive real. Conversely, if $G$ is positive real and $\system$ is controllable\footnote{Instead of the controllability, slightly weaker assumptions can be imposed. These conditions are rather technical and for this reason, we refer the reader to \cite{CleALM97}.}, then the \KYP inequality~\eqref{eqn:KYP} has a solution in $\Spsd{\stateDim}$.
    \end{enumerate}
\end{theorem}

Clearly, the first equation~\eqref{eqn:Lure:a} is a Lyapunov equation and is uniquely solvable (for a fixed $L$) if $A$ is Hurwitz \cite[Cha.~12, Thm.~2]{LanT85}. 
We define the Lyapunov operator
\begin{equation}
	\label{eqn:LyapOperator}
	\calL\colon \R^{\stateDim\times\stateDim}\to\R^{\stateDim\times\stateDim},\qquad X \mapsto A^\T X + XA.
\end{equation}
We emphasize that $A$ being Hurwitz implies that $\calL$ is invertible \cite[Cha.~12, Thm.~2]{LanT85}. 
Conversely, we can use the right-hand sides of the Lur'e equations~\eqref{eqn:Lure} to parameterize any passive system $\systemPas\big(\reduce{C}\big)$, as we detail in the next result. 

\begin{proposition}
	\label{thm:passParameterization}
	The system $\system$ in \eqref{eqn:LTI} is passive if and only if there exist matrices $L\in\R^{\stateDim\times\inpVarDim}$ and $M\in\R^{\inpVarDim\times\inpVarDim}$ such that
	\begin{subequations}
		\begin{align}
			\label{eqn:passParam:C} 
			C &= B^\T \calL^{-1}\big(-LL^\T\big) + ML^\T,\\
			\label{eqn:passParam:D}
			D + D^\T &= M M^\T.
		\end{align}
	\end{subequations}
\end{proposition}

\begin{proof}
	The result is a direct application of \Cref{thm:KYPlemma}. In particular, note that if $\system$ is passive, then the Lur'e equations~\eqref{eqn:Lure} have a solution $(X,L,M)$ and~\eqref{eqn:passParam:C} is obtained from~\eqref{eqn:Lure:b} by eliminating the matrix $X$ via~\eqref{eqn:Lure:a}. Conversely, if~\eqref{eqn:passParam:C} is satisfied, then we immediately obtain a solution of the Lur'e equations~\eqref{eqn:Lure} by setting $X = \calL^{-1}\big(-LL^\T\big)$ and hence, $\system$ is passive by virtue of \Cref{thm:KYPlemma}.
\end{proof}

\Cref{thm:passParameterization} provides a parameterization of passive systems in terms of the matrices $L$ and $M$, i.e., a mapping from $L$ and $M$ to the output matrix $\reduce{C}$ and the feedthrough matrix $D$.
The reverse mapping is not unique. If $D+D^\T$ is nonsingular, one choice is given by
\begin{equation}
	\label{eqn:kyp:are:LM}
	M = (D+D^\T)^{\tfrac{1}{2}}\qquad \text{and}\qquad L = \big(C^\T-XB\big)M^{-1},
\end{equation}
where $X$ is a solution of the \ARE~\eqref{eqn:kyp:are}.

\begin{remark}
	The parameterization of $C$ and $D$ is similar to the so-called trace parameterization presented in \cite{Dum02,CoePS04}. The key difference here is that we explicitly exploit the existence of rank-minimizing solutions and bypass the semi-definiteness constraint by optimizing over the Cholesky factors.
\end{remark}

\section{\texorpdfstring{$\calH_2$}{H2}-optimal passivation}
\label{sec:passivation}
Recall the passivation problem  defined in \eqref{eqn:passivationProblem}. In this section,  we will use the $\calH_2$-norm for the objective functional and introduce an efficient numerical approach to the resulting passivation problem with respect to the $\calH_2$-norm.
For an asymptotically stable dynamical system~\eqref{eqn:LTI} with transfer function~\eqref{eqn:transferFunction}, the $\calH_2$-norm of $\transferFunction$ is defined as 
$$
\left\| \transferFunction \right\|_{\calH_2}^2 \vcentcolon= \frac{1}{2\pi} \int_{-\infty}^\infty \left\| \transferFunction(\i \omega) \right\|_{\mathrm{F}}^2\, \domega.
$$
Note that  ${\left\| \transferFunction \right\|}_{\calH_2}$ is finite if and only if $D = 0$, see, e.g., \cite[Cha.~5]{Ant05a}.

For our passivation problem~\eqref{eqn:passivationProblem} where $A$ is Hurwitz, the $\calH_2$-error (the objective functional) is given by
\begin{align}
	\label{eqn:H2err}
	\left\|\transferFunction - \transferFunctionPas\big(\cdot;\reduce{C}\big)\right\|_{\calH_2}^2 = \frac{1}{2\pi} \int_{-\infty}^\infty \left\|\transferFunction(\i \omega) - \transferFunctionPas\big(\i\omega;\reduce{C}\big)\right\|_{\mathrm{F}}^2\, \domega.
\end{align}
This distance measure implies that the passivated model $\systemPas(\reduce{C})$ has the same feedthrough term $D$ as $\system$. 
Therefore, given the asymptotically stable system $\system$ in~\eqref{eqn:LTI} with the transfer function $\transferFunction$ in~\eqref{eqn:transferFunction}, our goal is to find $\systemPas\big(\reduce{C}\big)$ in~\eqref{eqn:LTI:modifiedC} with the transfer function $\transferFunctionPas\big(\cdot;\reduce{C}\big)$ in~\eqref{eqn:transferFunction} that solves the minimization problem 
\begin{equation}
	\label{eqn:passivationProblemwithH2}
	\min_{\reduce{C}\in\R^{\inpVarDim\times\stateDim}} \big\|\transferFunction-\transferFunctionPas\big(\cdot,\reduce{C}\big)\big\|_{\Htwo} \qquad\text{such that}\qquad \systemPas(\reduce{C}) \text{ is passive.}
\end{equation}
For the forthcoming analysis, we define the feasible set as
\begin{equation}
	\label{eqn:feasibleSet}
	\feasibleSet \vcentcolon= \left\{\reduce{C}\in\R^{\inpVarDim\times\stateDim} \;\Big| \; \systemPas\big(\reduce{C}\big) \text{ is passive}\right\}.
\end{equation}

\begin{remark}
 In the passivation approach we will develop, it is in principle possible to consider the case that the matrix $A$ has some semi-simple eigenvalues on the imaginary axis. If the corresponding residues of the transfer function $G$ are Hermitian and positive semi-definite, then the subsystem corresponding to these
 pole-residue pairs with purely imaginary poles is passive \cite{AndV73}. Then we could truncate this subsystem from the original system, passivate the remaining asymptotically stable subsystem using, e.g., our algorithm, and then add the two passive subsystems to obtain an overall passive system. 
\end{remark}

\subsection{Reformulation of the optimization problem and gradient computation}
\label{sec:passivation:optproblem}

We note, in the view of the \KYP inequality~\eqref{eqn:KYP}, that we need to assume $D + D^\T\in\Spsd{\inpVarDim}$ for $\calH_2$-optimal passivation.
With the previous discussion and \Cref{thm:passParameterization}, let $M\in\R^{\inpVarDim \times \inpVarDim}$ denote the square-root of $D+D^\T\in\Spsd{\inpVarDim}$ and define the system
\begin{equation}
	\label{eqn:LTI:passive:H2}
	\reduce{\system}(L): \quad \begin{cases}
		\dot{\state}(t) = A\state(t) + B\inpVar(t),\\
		\outVar(t) = \reduce{C}(L)\state(t) + D\inpVar(t),
	\end{cases}
\end{equation}
where 
\begin{equation}
	\label{eqn:Cparam}
	\reduce{C}(L) \vcentcolon= B^\T \calL^{-1}\big(-LL^\T\big) + ML^\T
\end{equation}
and $\calL$ is the bijective Lyapunov operator defined in~\eqref{eqn:LyapOperator}. By construction, \eqref{eqn:LTI:passive:H2} is guaranteed to be passive for any $L\in\R^{\stateDim\times\inpVarDim}$; see \Cref{thm:passParameterization}. The associated transfer function is given by
\begin{equation}
	\label{eqn:LTI:passive:H2:transferFunction}
	\transferFunctionPas(s;L) \vcentcolon= \reduce{C}(L)(s I_{\stateDim}-A)^{-1}B + D.
\end{equation}
Note that we employ a slight abuse of notation here, as we denote the passivated system, the passivated transfer function, and the passivating output matrix depending on $L$ using the same symbols as the quantities depending on $\reduce{C}$ in~\eqref{eqn:LTI:modifiedC} and~\eqref{eqn:transferFunction}. However, we will explicitly write the dependence in the following to prevent confusion.

The constrained optimization problem~\eqref{eqn:passivationProblemwithH2} can be equivalently formulated as the unconstrained minimization problem
\begin{equation}
	\label{eqn:passivationProblem:H2}
	\min_{L\in\R^{\stateDim\times\inpVarDim}} \objFunc(L)\qquad\text{with}\qquad \objFunc(L) \vcentcolon= \big\|\transferFunction-\transferFunctionPas(\cdot;L)\big\|_{\calH_2}^2.
\end{equation}
Before establishing existence and uniqueness of the solution of~\eqref{eqn:passivationProblem:H2}, we first present an alternative description of the objective functional and its gradient, which we will later use for our numerical algorithm.

\begin{theorem}
	\label{thm:H2passivation}
	Assume that $\Sigma$ in~\eqref{eqn:LTI} is asymptotically stable and let $\ctrlG\in\Spsd{\stateDim}$ denote the controllability Gramian of~\eqref{eqn:LTI}. Then $\objFunc$ defined in~\eqref{eqn:passivationProblem:H2} is 
	% Fr\'echet 
	differentiable, and
	\begin{subequations}
		\begin{align} 
		\label{eqn:costFunctionalH2}
		\objFunc(L) &= \tr\left(\big(C-\reduce{C}(L)\big)\ctrlG\big(C^\T - \reduce{C}(L)^\T\big)\right),\\
		\label{eqn:costFunctionalH2:gradient}
		\nabla\objFunc(L) &= 2\mathcal{X}L - 2 \ctrlG\big(C^\T - \reduce{C}(L)^\T\big)M,
		\end{align}
	\end{subequations}
	where $\calX \in \R^{\stateDim \times \stateDim}$ is the unique solution of the Lyapunov equation
	\begin{equation}
		\label{eqn:h2:grad:lyap}
		A \calX + \calX A^\T - \ctrlG\big(C^\T - \reduce{C}(L)^\T\big)B^\T - B\big(C - \reduce{C}(L)\big)\ctrlG = 0.
	\end{equation}
\end{theorem}

Before proving \Cref{thm:H2passivation}, we first recall the following auxiliary result to relate the inner products of solutions of Lyapunov equations.
\begin{lemma}[{\!\!\cite[Lem.~A.1]{WL99}}]
	\label{lem:lyapunovTrace}
	Let $A\in\R^{\stateDim\times \stateDim}$ be Hurwitz and $D,\,F\in\R^{\stateDim\times \stateDim}$  be such that the matrices $Y,\,Z\in\R^{\stateDim\times \stateDim}$ solve the Lyapunov equations 
	\begin{align*}
		A Y + YA^\T + D = 0 \qquad\text{and}\qquad
		A^\T Z + ZA + F = 0.
	\end{align*}
	Then, $\tr\big(D^\T Z\big) = \tr\big(F^\T Y\big)$.
\end{lemma}

\begin{proof}[Proof of \Cref{thm:H2passivation}]
	Through direct algebraic manipulation, it can be verified that the controllability Gramian $\Perr$ of the error system $\systemErr(L) = (\Aerr,\Berr,\Cerr(L),0)$  defined via
	\begin{align*}
		\Aerr &\vcentcolon= \begin{bmatrix}
			A & 0\\
			0 & A
		\end{bmatrix}, &
		\Berr &\vcentcolon= \begin{bmatrix}
			B\\
			B
		\end{bmatrix}, &
		\Cerr(L) &\vcentcolon= \begin{bmatrix}
			C & -\reduce{C}(L)
		\end{bmatrix}
	\end{align*}
	is given as $\Perr = \begin{smallbmatrix}\ctrlG & \ctrlG\\\ctrlG & \ctrlG\end{smallbmatrix}$. Hence, using standard results on the $\calH_2$-norm, see, e.g., \cite{Ant05a,VanGA08}, we obtain
	\begin{align*}
		\objFunc(L) &= \tr\left(\Cerr(L) \Perr \Cerr(L)^\T\right) = \tr\left(\left(C-\reduce{C}(L)\right)\calP\left(C^\T - \reduce{C}(L)^\T\right)\right),
	\end{align*}
	which is the first statement. 

	Inserting~\eqref{eqn:passParam:C} and setting $X \vcentcolon= \calL^{-1}\big(-LL^\T\big)$ yields
	\begin{multline*}
			\objFunc(L) =  \tr\big(C\ctrlG C^\T\big) - 2\tr\big(B^\T X \ctrlG C\tran + ML^\T \ctrlG C^\T\big) \\
			+ \tr\big(B^\T X \ctrlG  X B + 2 B^\T X \ctrlG  L M^\T + M L^\T \ctrlG  L M^\T\big).
	\end{multline*}
	We calculate the Fr\'echet derivative of $\objFunc$, which is the usual derivative  for finite-dimensional spaces. To this end we introduce a perturbation $\Delta_L$ for $L$, which invokes the symmetric perturbation $\Delta_X \vcentcolon= \calL^{-1}(-(L+\Delta_L)(L+\Delta_L)^\T) - \calL^{-1}(-LL^\T)$ for~$X$. Using the cyclic property of the trace, we obtain
	\begin{equation*}
		\begin{aligned}  
			\objFunc(L+\Delta_L) - \objFunc(L) &= - 2\tr\big(\ctrlG C^\T B^\T \Delta_X\big) - 2 \tr\big(M^\T C \ctrlG  \Delta_L\big)\\
		 	&\qquad + 2\tr\big(BB^\T X \ctrlG  \Delta_X\big) + \mathcal{O}\big({\|\Delta_X\|}_\mathrm{F}^2\big) \\
			&\qquad + 2\tr\big(M^\T B^\T X \ctrlG  \Delta_L\big) + 2 \tr\big(\ctrlG LM^\T B^\T \Delta_X\big) + \mathcal{O}\big({\|\Delta_X\|}_{\mathrm{F}} {\|\Delta_L\|}_\mathrm{F}\big)\\
			&\qquad + 2\tr\big(M^\T M L^\T \ctrlG  \Delta_L\big) + \mathcal{O}\big({\|\Delta_L\|}_\mathrm{F}^2\big).
		\end{aligned}
	\end{equation*}
	The definition of $\Delta_X$ is equivalent to
	\begin{align*}
		A^\T \Delta_X + \Delta_X A + \big(\Delta_L L^\T + L\Delta_L^\T + \Delta_L\Delta_L^\T\big) = 0.
	\end{align*}
	Hence, applying \Cref{lem:lyapunovTrace} together with the additional Lyapunov equations
	\begin{equation*}
		\begin{aligned}
			AY + YA^\T + BC\ctrlG  &= 0, &
			AZ + ZA^\T + \ctrlG XBB^\T &= 0, &
			AW + WA^\T + BML^\T \ctrlG  &= 0
		\end{aligned}
	\end{equation*}
	yields
	\begin{equation*}
		\small
		\begin{aligned}
			\tr\big(\ctrlG C^\T B^\T \Delta_X\big) &= \tr\big(\big(\Delta_L L^\T + L \Delta_L^\T + \Delta_L\Delta_L^\T\big) Y\big) = \tr\big(L^\T \big(Y+Y^\T\big) \Delta_L\big) + \calO\big({\|\Delta_L\|}_{\mathrm{F}}^2\big),\\
			\tr\big(BB^\T X \ctrlG  \Delta_X\big) &= \tr\big(\big(\Delta_L L^\T + L \Delta_L^\T + \Delta_L\Delta_L^\T\big) Z\big)= \tr\big(L^\T \big(Z+Z^\T\big) \Delta_L\big) + \calO\big({\|\Delta_L\|}_{\mathrm{F}}^2\big),\\
			\tr\big(\ctrlG LM^\T B^\T \Delta_X\big) &= \tr\big(\big(\Delta_L L^\T + L \Delta_L^\T + \Delta_L\Delta_L^\T\big) W\big) = \tr\big(L^\T \big(W+W^\T\big) \Delta_L\big) + \calO\big({\|\Delta_L\|}_{\mathrm{F}}^2\big).
		\end{aligned}
	\end{equation*}
	Setting $\calX \vcentcolon= - \big(Y+Y^\T\big) + \big(Z+Z^\T\big) + \big(W+W^\T\big)$ and using~\eqref{eqn:passParam:C}, we obtain 
	\begin{multline*}
		A\calX + \calX A^\T = \big(BC\ctrlG  - \ctrlG XBB^\T - BML^\T \ctrlG \big) + \big(BC\ctrlG  - \ctrlG XBB^\T - BML^\T \ctrlG \big)^\T\\
			= \ctrlG \big(C^\T - \reduce{C}(L)^\T\big)B^\T + B\big(C-\reduce{C}(L)\big)\ctrlG.
	\end{multline*}
	This shows that $\calX$ solves the Lyapunov equation~\eqref{eqn:h2:grad:lyap}.
	Plugging these identities back into the expression for $\objFunc(L+\Delta_L) - \objFunc(L)$ above yields
	\begin{align*}
			\objFunc(L+\Delta_L) - \objFunc(L) &= - 2\tr\big(L^\T \big(Y+Y^\T\big) \Delta_L\big) - 2 \tr\big(M^\T C \ctrlG  \Delta_L\big) + 2\tr\big(L^\T \big(Z+Z^\T\big) \Delta_L\big)\\
			&\qquad + 2\tr\big(M^\T B^\T X \ctrlG  \Delta_L\big) + 2 \tr\big(L^\T \big(W+W^\T\big) \Delta_L\big)\\
			&\qquad + 2\tr\big(M^\T M L^\T \ctrlG  \Delta_L\big)
			 + \mathcal{O}\left(\big({\|\Delta_X\|}_\mathrm{F} + {\|\Delta_L\|}_{\mathrm{F}}\big)^2\right)\\
			&= 2\tr\big(\big(L^\T \mathcal{X} - M^\T C\ctrlG  + M^\T B^\T X \ctrlG  + M^\T M L^\T \ctrlG \big) \Delta_L \big)\\
			&\qquad + \mathcal{O}\left(\big({\|\Delta_X\|}_\mathrm{F} + {\|\Delta_L\|}_{\mathrm{F}}\big)^2\right).
	\end{align*}
	Since $\Delta_L \to 0$ implies $\Delta_X \to 0$, and
	\begin{equation*}
	 \R^{n \times m} \to \R, \qquad \Delta_L \mapsto \tr\big(\big(L^\T \mathcal{X} - M^\T C\ctrlG  + M^\T B^\T X \ctrlG  + M^\T M L^\T \ctrlG \big) \Delta_L \big)
	\end{equation*}
	is a bounded linear operator, we can conclude Fr\'echet differentiability of $\objFunc$.
	Moreover, using again~\eqref{eqn:passParam:C}, we conclude that the gradient of $\objFunc$ is given by~\eqref{eqn:costFunctionalH2:gradient}.
\end{proof}

\begin{theorem}
	\label{thm:H2solvability}
	Consider system~\eqref{eqn:LTI} and assume it is asymptotically stable and controllable. Furthermore, let $D + D^\T\in\Spsd{\inpVarDim}$. Then, the minimization problem~\eqref{eqn:passivationProblem:H2} is solvable and for any minimizers $L_1,L_2\in\R^{\stateDim\times\inpVarDim}$, we have $\reduce{C}(L_1) = \reduce{C}(L_2)$, i.e., the minimizing passive system is unique.
\end{theorem}

\begin{proof}
	To prove the claim, we consider the constrained optimization problem~\eqref{eqn:passivationProblemwithH2} and use~\eqref{eqn:costFunctionalH2}, \eqref{eqn:weightedInnerProduct}, and~\eqref{eqn:feasibleSet} to rewrite it as
	\begin{equation}
		\label{eqn:passivationProblem:H2:constrained}
		\min_{\reduce{C}\in\R^{\inpVarDim\times\stateDim}} {\|C-\reduce{C}\|}_{\ctrlG}\qquad\text{such that } \reduce{C}\in\feasibleSet.
	\end{equation}
	Due to \Cref{thm:KYPlemma} we observe that $\reduce{C}\in\feasibleSet$ if and only if there exists $X\in\Spsd{\stateDimRed}$ such that the \KYP inequality~\eqref{eqn:KYP} is satisfied. As the mapping
	\begin{align*}
		(\reduce{C},X) \mapsto \begin{bmatrix}
			-A^\T X - XA & \reduce{C}^\T - XB\\
        	\reduce{C}-B^\T X & D+D^\T
		\end{bmatrix}
	\end{align*}
	is continuous, we note that $\feasibleSet$ is closed. Moreover, direct verification demonstrates that $\feasibleSet$ is convex and non-empty (since $0\in\feasibleSet$). Hence, the optimization problem~\eqref{eqn:passivationProblem:H2:constrained} is uniquely solvable due to the Hilbert projection theorem~\cite[Thm.~4.10]{Rud87}.
	The result now follows from \Cref{thm:passParameterization}.
\end{proof}

\begin{remark}
	The constrained optimization problem~\eqref{eqn:passivationProblem:H2:constrained} is considered in the literature, albeit in a slightly reformulated way; see for instance \cite[Sec.~5.4.1]{GriS21}. Specifically, the matrix $\reduce{C}$ is typically considered as a perturbed variant of the original matrix $C$, i.e., $\reduce{C} = C + \Delta P^{-1}$ for some $\Delta\in\R^{\inpVarDim\times\stateDim}$, where $P$ is a Cholesky factor of the controllability Gramian, i.e., $\ctrlG = PP^\T$. With this change of coordinates, we obtain $\|C-\reduce{C}\|_{\ctrlG} = \|\Delta\|_{\mathrm{F}}$ 
	and hence simply minimize the (weighted) perturbation in the Frobenius norm. 
\end{remark}

\begin{remark}
	We can define a new objective functional and its associated gradient by replacing the controllability Gramian with any symmetric positive semi-definite matrix $\ctrlGweighted$, for instance, given by the frequency-limited or time-limited controllability Gramian; cf.~\cite{GawJ90,BKS16}. We could also consider another dynamical system 
	$W(s) = C_{\mathrm{w}}(sI-A_{\mathrm{w}})^{-1}B_{\mathrm{w}} + D_{\mathrm{w}}$ as a weight, where $A_{\mathrm{w}}$ is Hurwitz. Then we can define the weighted $\calH_2$ error
	\begin{align*}
		\big\|\transferFunction - \transferFunctionPas(\cdot;L)\big\|_{W}^2 \vcentcolon= \big\|\big(\transferFunction - \transferFunctionPas(\cdot;L)\big)W\big\|_{\Htwo}^{2} = \big\|C-\reduce{C}(L)\big\|_{\ctrlGweighted}^2
	\end{align*}
	with weighted controllability Gramian
	\begin{align*}
		\begin{bmatrix}
			A & BC_{\mathrm{w}}\\
			0 & A_{\mathrm{w}}
		\end{bmatrix}\begin{bmatrix}
			\ctrlGweighted & \ctrlG_1\\
			\ctrlG_1^\T & \ctrlG_2
		\end{bmatrix} + \begin{bmatrix}
			\ctrlGweighted & \ctrlG_1\\
			\ctrlG_1^\T & \ctrlG_2
		\end{bmatrix}\begin{bmatrix}
			A^\T & 0\\
			C_{\mathrm{w}}^\T B^\T & A_{\mathrm{w}}^\T
		\end{bmatrix} + \begin{bmatrix}
			BD_{\mathrm{w}}\\
			B_{\mathrm{w}}
		\end{bmatrix}\begin{bmatrix}
			BD_{\mathrm{w}}\\
			B_{\mathrm{w}}
		\end{bmatrix}^\T = 0.
	\end{align*}
	We refer to \cite[Sec.~10.9]{GriG16} for further details.
\end{remark}

\subsection{Numerical considerations}
\label{sec:accelerations}
Suppose we use a gradient-based optimization strategy to solve the minimization problem~\eqref{eqn:passivationProblem:H2}. In that case, we need to solve two Lyapunov equations in each optimization step: the first one defined by~\eqref{eqn:Lure:a} related to~\eqref{eqn:Cparam}, and the second by~\eqref{eqn:h2:grad:lyap}. Solving these Lyapunov equations is thus the main computational cost. We refer the reader to \cite{Sim16,BS13a} for an overview of state-of-the-art methods for solving Lyapunov equations. Here, for the Lyapunov equation $A\herm X + X A = W$, we distinguish three scenarios\footnote{Here, we write the Lyapunov equation for complex matrices, since the diagonalizing transformation we consider may result in a complex matrix $A$.}.
First, the matrix $A$ is diagonal, say $A = \diag(\lambda_1,\ldots,\lambda_n)$. This can be achieved, e.g., if the model is obtained from data 
 and such a diagonal form is enforced during the data-driven modeling stage, 
or if the model is diagonalizable and the transformation to diagonal form is done before the passivation. In this case, a closed-form solution is given by
\begin{equation*}
	x_{ij} = \frac{w_{ij}}{\overline{\lambda_i} + \lambda_j},
\end{equation*}
where $w_{ij}$ corresponds to the $(i,j)$ entry of the right-hand side $W$, giving a computational cost of $\calO\big(\stateDim^2\big)$ for solving the Lyapunov equation. If $A$ is not in diagonal form or the transformation to diagonal form is not possible, numerically not stable, or computationally infeasible, then we can either use a Krylov subspace method or an ADI method if $\stateDim$ is large and $A$ is sparse, or the Bartels-Stewart algorithm or Hammarling's method if $\stateDim$ is small or moderate and $A$ is dense. The latter two methods require $\calO(\stateDim^3)$ floating point operations, while the former two are iterative methods that are typically much more efficient in the large and sparse setting. A detailed complexity analysis is available, e.g., in \cite{LW02}.
Notably, the Bartels-Stewart algorithm first performs a Schur decomposition of $A\herm$. To avoid recomputing the Schur decomposition, we only compute it once and then transform the system accordingly. Whenever numerically reasonable, we recommend transforming the system to diagonal form, as this reduces the overall computational cost significantly; see \Cref{sec:numerics:smartphone}. If we want to retain real-valued matrices throughout the optimization, it is advisable to transform them into real-diagonal form.

\subsection{Local minimizers and uniqueness}
\label{sec:passivation:localMin}
In the proof of \Cref{thm:H2solvability}, we have seen that the constrained passivation problem~\eqref{eqn:passivationProblem:H2:constrained} is convex, and hence any local minimizer will be a global minimizer. Unfortunately, the transformation from the constrained optimization problem~\eqref{eqn:passivationProblem:H2:constrained} to the unconstrained optimization problem~\eqref{eqn:passivationProblem:H2} results in a non-convex minimization problem. Hence, if we apply gradient-based algorithms to solve the unconstrained minimization problem~\eqref{eqn:passivationProblem:H2}, we must investigate whether we have found the global minimum or are stuck in a non-global local minimum.
Recall from~\eqref{eqn:passParam:D} the definition of $M$: $D + D^\T = M M^\T$. As we will see, the situation is quite different for the cases $M=0$ and $M\neq 0$, so we analyze the two scenarios separately.

\subsubsection{Case $M=0$}
Albeit the resulting passive system is unique, the minimizer of~\eqref{eqn:passivationProblem:H2} may not be unique. This can be easily seen in the case $M = 0$. In this case, if $L^\star\in\R^{\stateDim\times\inpVarDim}$ is a minimizer, so is $L^\star U$ for any orthogonal matrix $U\in\R^{\inpVarDim\times\inpVarDim}$. 

\begin{proposition}
	\label{prop:locGlobMin}
	Assume $D + D^\T = 0$. Then any local minimizer of~\eqref{eqn:passivationProblem:H2} is a global minimizer.
\end{proposition}

\begin{proof}
	We note that for $D + D^\T=0$, the objective functional~\eqref{eqn:costFunctionalH2} and $\reduce{C}$ from~\eqref{eqn:Cparam} only depend on $LL^\T$, not on $L$.
	Let $\check{L}$ be a local minimizer of~\eqref{eqn:passivationProblem:H2}, i.e., there exists $\delta>0$ such that
	\begin{equation*}
		\objFunc\big(\check{L}\big) \leq \objFunc\big(L\big) \qquad\text{for all $LL^\T\in\R^{\stateDim\times\stateDim}$ with $\big\|LL^\T-\check{L}\check{L}^\T\big\| \leq \delta$}.
	\end{equation*}
	Further, assume that $\check{L}$ is not globally optimal and let $L^\star$ be a global minimizer. We define
	\begin{equation}
		L_\theta L_\theta^\T \vcentcolon= (1-\theta) \check{L}\check{L}^\T + \theta L^\star (L^\star)^\T,\qquad \theta\in(0,1).
	\end{equation}
	Let $\theta$ be such that $\big\|L_\theta L_\theta^\T - \check{L}\check{L}^\T\big\|\leq \delta$. 
	Then, using 
	$\reduce{C}\big(L_\theta\big) = (1-\theta)\reduce{C}\big(\check{L}\big) + \theta\reduce{C}\big(L^\star\big)$,
	we obtain
	\begin{align*}
		\big\|C-\reduce{C}(L_\theta)\big\|_{\ctrlG} &\leq (1-\theta)\big\|C-\reduce{C}\big(\check{L}\big) \big\|_{\ctrlG} + \theta \big\|C-\reduce{C}\big(L^\star\big)\big\|_{\ctrlG} \\ &< \big\|C-\reduce{C}\big(\check{L}\big)\big\|_{\ctrlG},
	\end{align*}
	which is a contradiction. Hence, $\check{L}$ has to be a global optimizer.
\end{proof}

\begin{example}
	\label{ex:Meq0}
	Consider the non-passive system
	\begin{equation*}
		A = \begin{bmatrix}
			-1 & \phantom{}4 \\
			-2 & -1
		\end{bmatrix}, \quad 
		B = \begin{bmatrix}
			1 \\
			2
		\end{bmatrix}, \quad
		C = \begin{bmatrix}
			1 & 0
		\end{bmatrix},
		\quad D = 0.
	\end{equation*}
	In~\Cref{fig:Meq0}, we visualize the squared $\calH_2$-error for all $L$ and passivating $\reduce{C}$, respectively. We see in~\Cref{fig:Meq0:L} that we have two global minima in the $L$ parameterization. Indeed, with 
 	$L^{\star}_{1} = [0.96, -0.48]^\T$ and $L^{\star}_{2} := -L^{\star}_{1}$, we obtain  the same 
	$\reduce{C}^{\star} \approx  [0.46, 0.80]$, which yields a global minimum of $\objFunc(L^{\star}_{1,2}) \approx 0.94$.
	\begin{figure}[htb]
		\centering
		\ref{leg:toy0}
		\\
		\begin{subfigure}{.5\linewidth}
			\centering
			\input{plots/lspace_D=0.0.tikz}
			\caption{$L$-space}
			\label{fig:Meq0:L}
		\end{subfigure}\hfill
		\begin{subfigure}{.5\linewidth}
			\centering
			\input{plots/cspace_D=0.0.tikz}
			\caption{$\reduce{C}$-space}
			\label{fig:Meq0:C}
		\end{subfigure}
		\caption{Squared $\mathcal{H}_2$-error for all $L$ and passivating $\reduce{C}$, respectively, for \Cref{ex:Meq0} with $M=0$. The blue and orange dot correspond to the local minimizers, which are global minimizers due to \Cref{prop:locGlobMin}. 
		}
		\label{fig:Meq0}
	\end{figure}
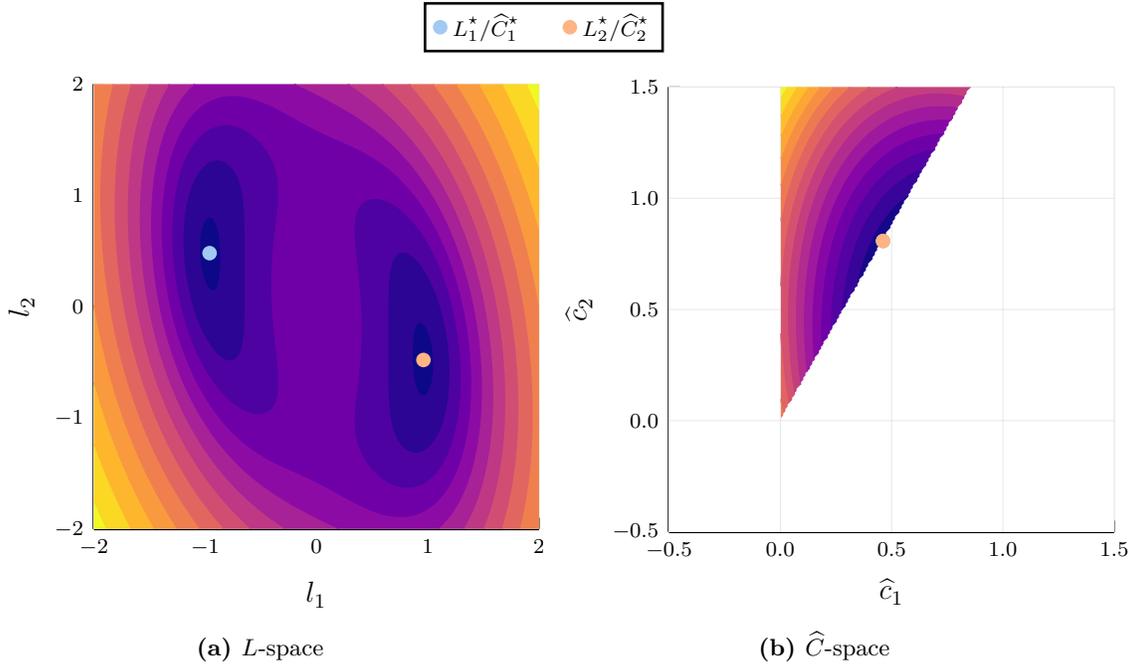
\end{example}

\subsubsection{Case $M\neq 0$}
Unfortunately, \Cref{prop:locGlobMin} cannot be generalized to the case $M\neq 0$, which can be seen from the following counter example.

\begin{example}
	\label{ex:Mneq0}
	Consider the system from~\Cref{ex:Meq0}, but this time with the feedthrough term $D=\tfrac{1}{8}$. We set $M = \tfrac{1}{2}$ and visualize the squared $\calH_2$-error for all $L$ respectively $\reduce{C}$ in~\Cref{fig:Mneq0}.
	This time we observe, that the presence of the feedthrough term leads to a unique global minimum of the squared $\calH_2$-error in the $L$-parametrization; see~\Cref{fig:Mneq0:L}. However, we also observe the presence of a non-global local minimum at $L^{\star}_{\mathrm{loc}}= [-1, 0]^\T$.

	\begin{figure}[htb]
		\centering
		\ref{leg:toy}
		\\
		\begin{subfigure}{.5\linewidth}
			\centering
			\input{plots/lspace_D=0.125.tikz}
			\subcaption{$L$-space}
			\label{fig:Mneq0:L}
		\end{subfigure}\hfill
		\begin{subfigure}{.5\linewidth}
			\centering
			\input{plots/cspace_D=0.125.tikz}
			\subcaption{$\reduce{C}$-space}
			\label{fig:Mneq0:C}
		\end{subfigure}
		\caption{Squared $\mathcal{H}_2$-error for all $L$ respectively passivating $\reduce{C}$ for \Cref{ex:Mneq0} with $M\neq 0$. The blue dot is a non-global local minimizer and the orange dot is the global minimizer.}
		\label{fig:Mneq0}
	\end{figure}
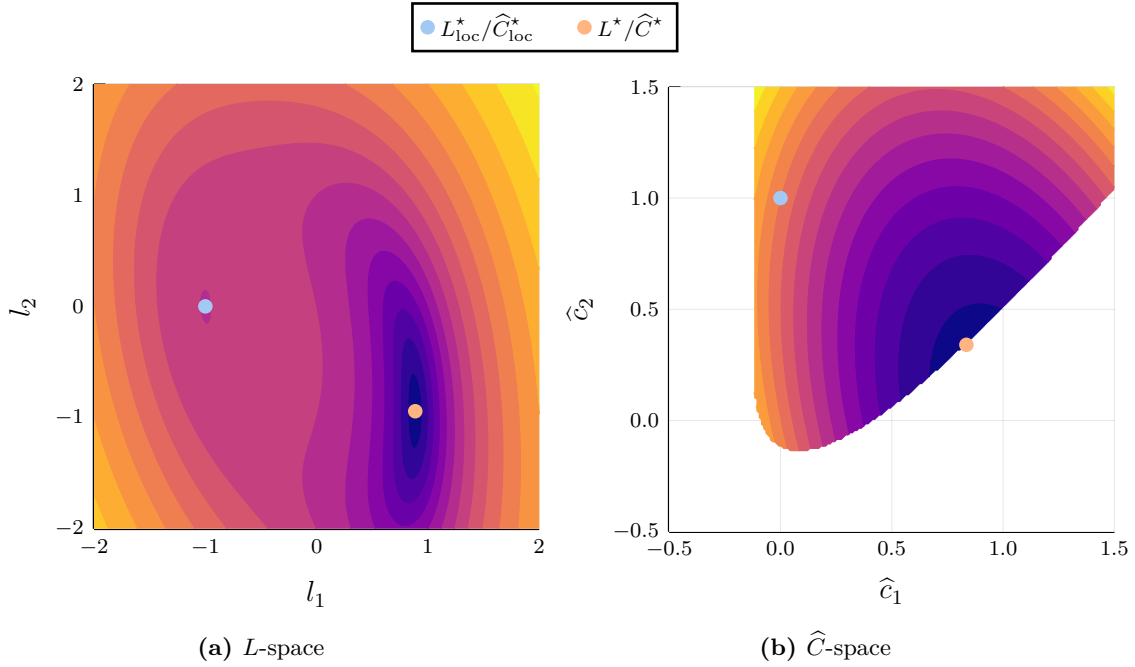
\end{example}
\Cref{ex:Mneq0} illustrates that the presence of the feedthrough term $D$ can lead to (multiple) local minima. 
Thus, once the optimization algorithm has converged, we need to determine whether it corresponds to the global minimum or to some non-global local minimum.
The following theorem provides a simple criterion to check whether a local minimum is a global one. We will prove that if $D+D^\T \succ 0$ and $\reduce{C} \neq 0$ is passivating and if the extremal \ARE solutions satisfy $X_{\min} = X_{\max}$, then $\reduce{C}$ is on the boundary of the feasible set $\feasibleSet$. In particular, if $\reduce{C}$ is locally $\calH_2$-optimal, then it is also globally $\calH_2$-optimal due to the convexity of the optimization problem \eqref{eqn:passivationProblem:H2}.

\begin{theorem}
	\label{thm:locmin}
	Assume $D+D^\T \succ 0$. Let $\systemPas\big(\reduce{C}\big)$ with $\reduce{C}\neq 0$ be passive such that there exists only one solution to the associated \KYP inequality \eqref{eqn:KYP}, i.e., the minimal and maximal solution of the \ARE are identical, i.e., $\Xmin=\Xmax$. Then, $\reduce{C}$ lies on the boundary of the feasible set $\feasibleSet$ defined in~\eqref{eqn:feasibleSet}.
\end{theorem}

\begin{proof}
	To simplify the notation in the proof, we introduce the matrix function
	\begin{align}
		\label{eqn:kyp:C}
		\calW(C,X) \vcentcolon= \begin{bmatrix}
			-A^\T X - XA & C^\T - XB\\
			C - B^\T X & D+D^\T
		\end{bmatrix}.
	\end{align}
	We assume that $\reduce{C}$ does not lie on the boundary. Then there exist $\Delta\in\R^{\inpVarDim\times\stateDim}\setminus\{0\}$ and $\varepsilon>0$ such that $\systemPas\big(\reduce{C} + \alpha \Delta\big)$ is passive for all $\alpha \in (-\varepsilon, \varepsilon)$. We will use this fact to prove the result via contradiction.
	Let $\Delta_{\reduce{C}} \in \R^{\inpVarDim\times\stateDim}\setminus\{ 0\}$ be such that $\systemPas\big(\reduce{C}+\Delta_{\reduce{C}}\big)$ is passive. 
	
	In particular, for $X \vcentcolon= \Xmin = \Xmax$ there exists symmetric $\Delta_X\in\R^{\stateDim\times\stateDim}$ such that
	\begin{align}
		\label{eqn:KYP:alpha}
		0 &\preceq \calW\big(\reduce{C}+\Delta_{\reduce{C}},X+\Delta_X\big)\\
		&= \begin{bmatrix}
			-A^\T (X + \Delta_X) - (X+\Delta_X) A & \big(\reduce{C}+\Delta_{\reduce{C}}\big)^\T - (X+\Delta_X)B\\
			\big(\reduce{C}+\Delta_{\reduce{C}}\big) - B^\T (X+\Delta_X) & D+D^\T
		\end{bmatrix}.
	\end{align}
	We define $\reduce{C}_\alpha \vcentcolon= \reduce{C} + \alpha \Delta_{\reduce{C}}$ and $X_{\alpha} \vcentcolon= X + \alpha \Delta_X$. 
	Due to convexity of~\eqref{eqn:kyp:C}, we immediately conclude $0 \preceq \calW\big(\reduce{C}_\alpha,X_\alpha\big)$ for all $\alpha\in[0,1]$. 
	We now prove that
	\begin{align*}
		0 \not\preceq \calW\big(\reduce{C}_\alpha,X_\alpha\big)\qquad\text{for any $\alpha<0$}.
	\end{align*}
	Since we do not rely on a specific choice of $\Delta_X$,
	we conclude that $\systemPas\big(\reduce{C}_\alpha\big)$ cannot be passive for $\alpha<0$. 
	We investigate two cases:
	\begin{description}
		\item[Case 1] Assume that $\Delta_{\reduce{C}} - B^\T \Delta_X = 0$. Using the Schur complement and the fact that $X$ solves the Riccati equation \eqref{eqn:kyp:are}, we conclude that $0 \preceq \calW(\reduce{C}_\alpha,X_\alpha)$ is satisfied if and only if $0 \preceq -A^\T X_\alpha - X_\alpha A$ and 
			\begin{align*}
				0 \preceq -\alpha (A^\T \Delta_X + \Delta_X A).
			\end{align*}
			Since this holds for $\alpha\in[0,1]$, we immediately conclude that this inequality cannot hold for $\alpha < 0$.
		\item[Case 2] Assume $\Delta_{\reduce{C}} - B^\T \Delta_X \neq 0$. Then there exists $y\in\R^n$ such that 
			\begin{align*}
				y^\T (\Delta_{\reduce{C}}^\T - \Delta_X B)(D+D^\T)^{-1}(\Delta_{\reduce{C}} - B^\T \Delta_X)y = 1.
			\end{align*}
			Set $\beta \vcentcolon= y^\T\left(A^\T \Delta_X + \Delta_X A + 2\big(\Delta_{\reduce{C}}^\T - \Delta_X B)(D+D^\T)^{-1}(\reduce{C} - B^\T X)\right)y$. Using once again the Schur complement and the fact that $X$ solves the Riccati equation \eqref{eqn:kyp:are}, we estimate
            \begin{align*}
			0 &\le y^\T \left( -A^\T X_\alpha - X_\alpha A - \big( C_\alpha^\T - X_\alpha B \big) \big(D+D^\T\big)^{-1}
			\big(C_\alpha - B^\T X_\alpha \big)\right)y \\
              &= y^\T \left( -\alpha A^\T  \Delta_X - \alpha \Delta_X A - 2\alpha\big(\Delta_{\reduce{C}}^\T - \Delta_X B)(D+D^\T)^{-1}(\reduce{C} - B^\T X) \right. \\ &\qquad\left. - \alpha^2\big( \Delta_{\reduce{C}}^\T - \Delta_X B \big) \big(D+D^\T\big)^{-1}
			\big(\Delta_{\reduce{C}} - B^\T\Delta_X \big)\right)y \\
            &= -\alpha \beta - \alpha^2.
            \end{align*}
            We conclude that
			\begin{align}
				\label{eqn:alphaBetaIneq}
				0 \leq -\alpha(\beta + \alpha)
			\end{align}
			is a necessary condition for $0 \preceq \calW(\reduce{C}_\alpha,X_{\alpha})$. 
			By assumption, $0 \preceq \calW(\reduce{C}_\alpha,X_{\alpha})$ for all $\alpha \in[0,1]$, and therefore~\eqref{eqn:alphaBetaIneq} must hold. In particular, evaluating~\eqref{eqn:alphaBetaIneq} at $\alpha = 1$, we obtain $0 \leq -\beta - 1$, which implies $\beta \leq -1$.
			Hence, we conclude that~\eqref{eqn:alphaBetaIneq} cannot hold for $\alpha < 0$.\qedhere
	\end{description}
\end{proof}

\begin{remark} 
    The previous result provides a sufficient condition for a local minimizer being a global one, namely if the extremal solutions of the \ARE \eqref{eqn:kyp:are} coincide. Unfortunately, this characterization is not true if $D+D^\T \not\succ 0$.
    For instance, in the scalar case, i.e., $\stateDim=\inpVarDim=1$ and $D=0$, the Lur'e equations~\eqref{eqn:Lure} reduce to 
	\begin{equation*}
		\label{eqn:Luren1}
		2AX = -L^2, \qquad BX -C = 0. 
	\end{equation*}
	If we assume $A<0$ and $B\neq0$, then~\eqref{eqn:Lure} is only solvable if $\sign(B) = \sign(C)$. Hence, the feasible set equals $\feasibleSet =\bigl\{\reduce{C} \in \R \;\big|\; \reduce{C}/B \geq 0\bigr\}$. For all $\reduce{C}\in\feasibleSet$, including $\reduce{C}\neq 0$, the unique solution is $X = -{\reduce{C}}/{B}$.
	Consequently, in the view of the passivation problem, $\Xmin = \Xmax$ does not imply that we are on the boundary of the feasible set and, therefore, does not imply that we have found the nearest passive system.
\end{remark}
\Cref{thm:locmin} provides a simple way to detect global minimizers. If our optimization algorithm converges to a minimizer $L^\star$, we can check whether the condition of \Cref{thm:locmin} is satisfied. 
In particular, we can compute the eigenvalues of the matrix 
\begin{align*}
	Y^\star \vcentcolon=&\, A - B\big(D+D^\T\big)^{-1}\left( \reduce{C}(L^\star) - B^\T \calL^{-1}\left(-L^\star \big(L^{\star}\big)^\T\right)\right) \\
	=&\, A - B(D+D^\T)^{-1} M\big(L^\star\big)^\T.
\end{align*}
If, up to a prescribed tolerance, the eigenvalues are not on the imaginary axis, then we conclude $\Xmin \neq \Xmax$ and we are potentially stuck in a non-global local minimum. In this case, we suggest performing a small gradient step on $\reduce{C}^{\star} \vcentcolon= \reduce{C}\big(L^\star\big)$ in the direction of $\nabla_{\reduce{C}} \calJ(L^\star)$, i.e., 
\begin{equation}
	\label{eqn:gradientDescent}
	% \reduce{C} \leftarrow \reduce{C} - \alpha \frac{\nabla_C \calJ(\cdot)}{\|\nabla_C \calJ(\cdot)\|_{\mathrm{F}}},
	\reduce{C}^{\star} \leftarrow \reduce{C}^{\star} - \alpha \nabla_{\reduce{C}} \calJ(L^\star),
\end{equation}
with $\nabla_{\reduce{C}} \calJ(L^\star) = 2(\reduce{C}^{\star}-C)\calP$ and $\alpha \ll 1$. We then check whether the resulting $\reduce{C}^{\star}$ still yields a passive realization. If this is the case we simply restart the optimization from this point on, by obtaining the new initial $L_0$ via~\eqref{eqn:kyp:are:LM}. If not, we apply our initialization strategy, which is described in the forthcoming~\Cref{sec:passivation:init} or try a smaller step size $\alpha$.

\begin{example}
	\label{ex:Mneq0:restart}
	Let us showcase the local minimum detection plus restart strategy in the following example. Consider again the system from~\Cref{ex:Mneq0} with feedthrough term $D = \tfrac{1}{8}$. We initialize the optimization at $L_0 = [-2,0]^\T$, which leads to convergence to a non-global local minimum. The result is $L^{\star}_{\mathrm{loc}} = [-1, 0]^\T$ with $\reduce{C}^{\star}_{\mathrm{loc}} = [0, 1]$, yielding
	\begin{equation*}
		Y^{\star}_{\mathrm{loc}} = \begin{bmatrix}
			1 & 4 \\
			2 & -1 
		\end{bmatrix}
		\quad \text{with eigenvalues } \lambda_{1,2} = \pm 3.
	\end{equation*}
	Thus, we conclude $X_{\min} \neq X_{\max}$, indicating that the algorithm is potentially stuck in a non-global local minimum, as demonstrated in this example.
	Applying our restart strategy ensures convergence to the global minimum. The results are visualized in~\Cref{fig:Mneq0:restart}. 
	Indeed, we obtain $L^{\star} \approx [0.89, -0.94]^\T$ with $\reduce{C}^{\star} \approx [0.84, 0.34]$ and 
	\begin{equation*}
		Y^{\star} \approx \begin{bmatrix}
			-2.77 & 5.89 \\
			-5.54 & 2.77
		\end{bmatrix}
		\quad \text{with eigenvalues } \lambda_{1,2} \approx \pm 5\imath.
	\end{equation*}
\end{example}

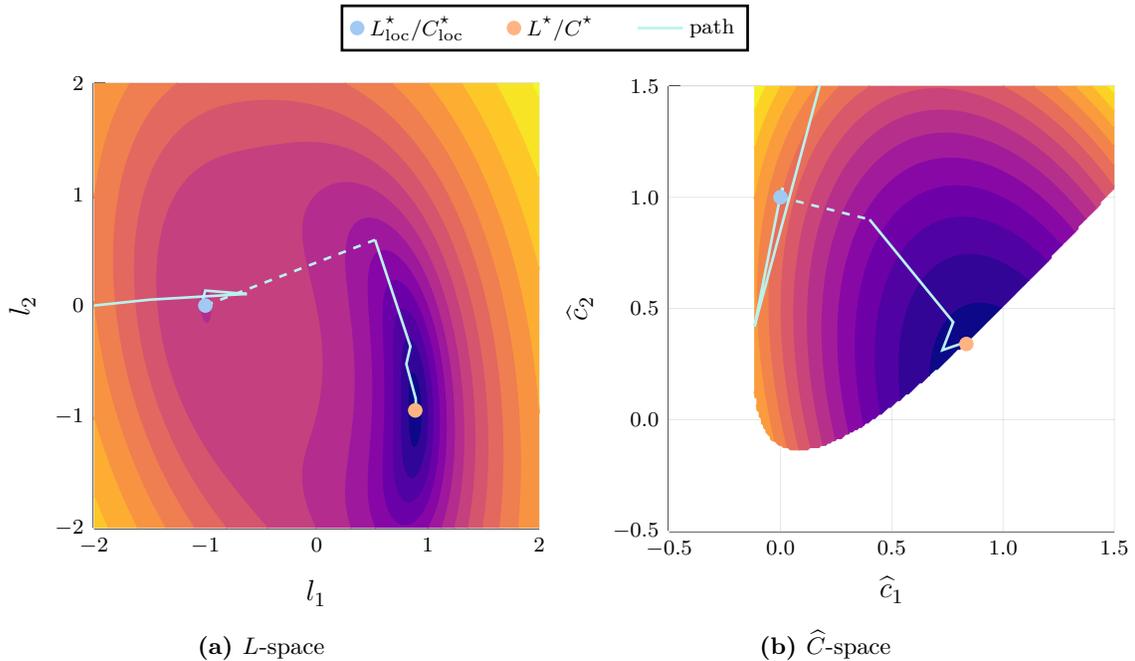
\begin{figure}[htb]
	\centering
	\ref{leg:toy:restart}\\
	\begin{subfigure}{.5\linewidth}
		\centering
		\input{plots/lspace_restart_D=0.125.tikz}
		\subcaption{$L$-space}
		\label{fig:Mneq0:restart:L}
	\end{subfigure}\hfill
	\begin{subfigure}{.5\linewidth}
		\centering
		\input{plots/cspace_restart_D=0.125.tikz}
		\subcaption{$\reduce{C}$-space}
		\label{fig:Mneq0:restart:C}
	\end{subfigure}
	\caption{Visualization of the $\Htwo$-error during the optimization for \Cref{ex:Mneq0} with $M\neq 0$. The path line corresponds to the path taken by the optimizer, the solid parts represent the standard L-BFGS optimization and the dashed part represents the gradient step performed after a local minimum is detected.}
	\label{fig:Mneq0:restart}
\end{figure}

\subsection{Initialization via solutions of the perturbed ARE}
\label{sec:passivation:init}
The \ARE~\eqref{eqn:kyp:are} associated with the \KYP inequality or the Lur'e equations does not have a solution for a non-passive system.
If we have a solution of the \ARE at hand, we can use it as an initialization for the optimization problem via~\eqref{eqn:kyp:are:LM}.
Now the idea is the following: Instead of perturbing $C$ to find a passive realization, we perturb the feedthrough term $D$. Then  we solve the \ARE for the perturbed system and use the solution as an initialization for the optimization problem.
The perturbation can be chosen as $\Delta_D = -\frac{\lambda_{\min}}{2} I_m$, with 
$\lambda_{\min}<0$ being the smallest eigenvalue of the modified Popov function along the imaginary axis $\Phi(\imath \omega)$ for all $\omega \in \R$.
Since we may not obtain the exact value of $\lambda_{\min}$, we can set $\Delta_D = -\left(\frac{\lambda_{\min}}{2} - \varepsilon\right) I_m$ for some small $\varepsilon>0$ to ensure that the perturbed system is passive.
In our algorithms we use a sampling technique over a sufficiently large frequency domain. The pseudocode for this initialization strategy is given in~\Cref{alg:init}. 

\begin{algorithm}[htb]
	\caption{Initialization}
	\label{alg:init}
	\KwIn{non-passive \LTI system~\eqref{eqn:LTI}, sampling frequencies $\Omega = \left\{\omega_1,\ldots,\omega_q \right\}$, tolerance $\varepsilon>0$.}
	\KwOut{Initial $L_0$ for the optimization.}
    Set $\lambda_{\min} = \min_{j \in \left\{1,\ldots,q\right\}} \lambda_{\min}\left(\Phi(\imath \omega_j)\right)$ with the modified Popov function $\Phi$ as in~\eqref{eqn:popovFunction}. \\
	Set $D_{\mathrm{pert}} = D - \left(\frac{\lambda_{\mathrm{min}}}{2} - \varepsilon\right) I_m$.\\
	Solve $A^\T X +XA + \big(C^\T - XB\big){\big(D_{\mathrm{pert}}+D_{\mathrm{pert}}^\T\big)}^{-1}\big(C-B^\T X\big) = 0$ for minimal solution $X_{\mathrm{min}}$.\\
	Set $M = \big(D+D^\T\big)^{1/2}$.\\
	Set $L_0 = \big(C^\T- X_{\mathrm{min}}B\big) M^{-1}$.\\
\Return{$L_0$.}
\end{algorithm}

An alternative would be an iteration over a typically small number of Hamiltonian eigenvalue problems. Such techniques are widely used for computing the $\calH_\infty$-norm of a transfer function \cite{BoyBK89, BruS90} and it is not difficult to adapt them to our setting.

\begin{example}
	\label{ex:init}
	Consider again the system from~\Cref{ex:Mneq0}.
	We first estimate the smallest eigenvalue of the modified Popov function $\Phi$ through sampling and set $\Delta_D =~-\frac{\lambda_{\mathrm{min}}}{2} I_m$ accordingly (see~\Cref{fig:init:popov}).
	With this adjustment, we solve the \ARE~\eqref{eqn:kyp:are} for the system incorporating the perturbed feedthrough term.
	The resulting $\Xmin$ serves as the initialization for the optimization problem formulated in~\eqref{eqn:kyp:are:LM}.
	In~\Cref{fig:init}, we present the squared $\calH_2$-error as a contour surface, similar to the earlier visualizations.
	In~\Cref{fig:init:Cspace}, we also display the $\mathcal{H}_2$-error contour plot for the perturbed system with reduced opacity.
	As anticipated, this contour plot is larger than that of the original system, indicating that the original $C$ already provides a passive realization.
	We also show the resulting initial $L_0$ and $\reduce{C}_0$, respectively and compare the latter with the original $C$ in~\Cref{fig:init:Cspace} and~\Cref{fig:init:Lspace}.
\end{example}

\begin{figure}[htb]
	\input{plots/popov_D=0.125.tikz}
	\caption{Sampled modified Popov function $\Phi(\imath \omega)$ for the system in~\Cref{ex:Mneq0} and the perturbed system.}
	\label{fig:init:popov}
\end{figure}
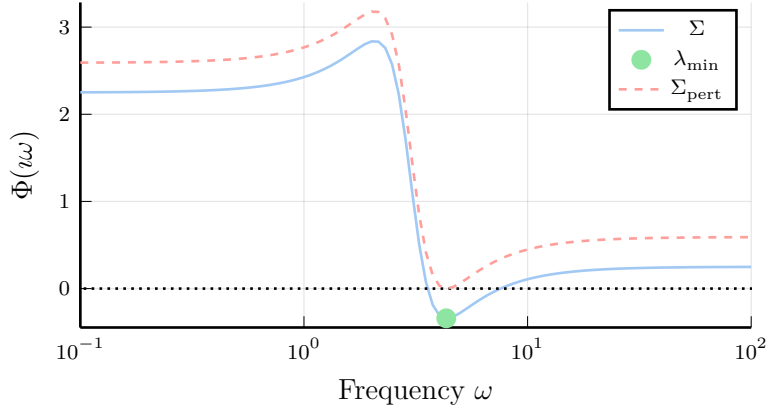 

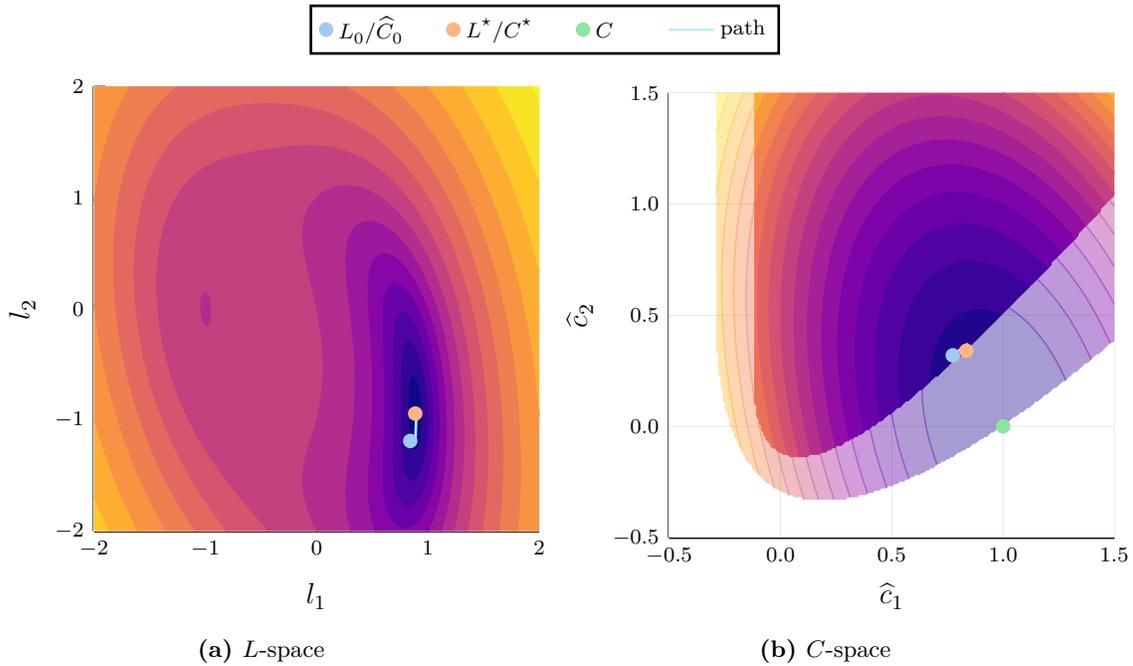
\begin{figure}[htb]
	\centering
	\ref{leg:toy:init}\\
	\begin{subfigure}{.5\linewidth}
		\centering
		\input{plots/lspace_init_D=0.125.tikz}
		\caption{$L$-space}
		\label{fig:init:Cspace}
	\end{subfigure}\hfill
	\begin{subfigure}{.5\linewidth}
		\centering
		\input{plots/cspace_init_D=0.125.tikz}
		\caption{$C$-space}
		\label{fig:init:Lspace}
	\end{subfigure}
	\caption{Visualization of the $\Htwo$-error during the optimization for \Cref{ex:init}. In reduced opacity, we show the $\Htwo$-error for the perturbed system with (enlarged) $D_{\mathrm{pert}}$. The green dot represents the $C$ matrix of the original system, which provides a passive realization for the perturbed system.
	Blue dots indicate the initial guesses for the optimization, while orange dots denote the global minimizer. The path line illustrates the trajectory taken by the optimizer.}
	\label{fig:init}
\end{figure}

\subsection{The proposed passivation algorithm}
\label{sec:passivation:algorithm}
We now have all the ingredients to  present the proposed passivation algorithm~\ourmethodfull. The pseudocode for~\ourmethod is given  in~\Cref{alg:h2opt-passivation}.
\begin{algorithm}[htb]
	\caption{\ourmethodfull}
	\label{alg:h2opt-passivation}
	\KwIn{non-passive \LTI system~\eqref{eqn:LTI}, tolerance $\varepsilon>0$.
	}
	\KwOut{$\Htwo$-optimal passivating $\reduce{C}$.}
	Set $M=\big(D+D^\T\big)^{1/2}$.\\
	Perform~\Cref{alg:init} to obtain an initial guess for $L_0$.\\
	Set $L=L_0$.\\
	\Repeat{break}{
    	Solve~\eqref{sec:passivation:optproblem} for $L^{\star}$ via L-BFGS with initial guess $L$.\\
		Set $Y^{\star} = A - B\big(D+D^\T\big)^{-1}M\big(L^{\star}\big)^\T$.\\
		\eIf{$\max_{\lambda \in \sigma(Y^\star)} \Real(\lambda) > \varepsilon$}{
			Perform gradient step~\eqref{eqn:gradientDescent} to obtain new $\reduce{C}^{\star}$.\\
			\eIf{$\systemPas\big(\reduce{C}^{\star}\big)$ is passive}{
				Obtain $L$ via~\eqref{eqn:kyp:are:LM}.\\
			}
			{
				Set $L=L^{\star}$ and \emph{break}.\\
			}
		}{
			Set $L=L^{\star}$ and \emph{break}.\\
		}
	}
	\Return{$\reduce{C}= B^\T \calL^{-1}\big(-LL^\T\big) + ML^\T$.}
\end{algorithm}
%-----------------------------------------------------------------------------%
\section{Numerical Experiments}
\label{sec:numerics}
In the following sections, we illustrate the effectiveness of~\Cref{alg:h2opt-passivation} on three benchmark systems:
\begin{inparaenum}[i)]
    \item a perturbed version of the ACC benchmark problem~\cite{WB92, KJ97} (small scale),
	\item the swing arm of a CD Player from~\cite{DSB92} (medium scale),
	\item and a high-speed smartphone interconnect link, which is also used in~\cite{GriS21} (medium/large scale).
\end{inparaenum}
A summary of the results is given in~\Cref{tab:runtimes}, we provide a detailed description of the results in the following sections.

Regarding the implementation, the following remarks are in order:
\begin{enumerate}
	\item All Experiments were conducted on an Apple M1 Max chip with 10 CPU cores and 32 GB of RAM.
	\item The implementation is done using the Julia programming language.
	\item The \LMI optimization problems are formulated in the JuMP modeling language~\cite{LDG+23} and solved using the Hypatia solver~\cite{CKV22} with default parameters. 
	\item For the proposed method~\ourmethod, we use the Optim.jl package~\cite{MR18} for the optimization. Specifically utilizing its L-BFGS implementation with the default convergence criterion based on the gradient norm. Moreover, we include an additional convergence criterion based on the relative change in the objective value, setting the tolerance to $\delta_{\mathcal{J}} = \num{1e-6}$ (equals \texttt{f\textunderscore tol} in the Optim.jl documentation). 
	\item For the restart strategy, we use $\alpha = \num{1e-8}$ as step size for the unconstrained gradient step.
\end{enumerate}

\begin{table}
	\caption{Comparative runtimes of the optimization process. The methods \LMI-TP~\cite{Dum02,CoePS04} and \LMI~\cite[Ch. 5.5.1]{GriS21} refer to the standard LMI method with and without trace parameterization.}
	\label{tab:runtimes}
	\begin{tabular}{llcccc}
		\toprule
		model & method & iterations & time (s) & time per iteration (s) & $\mathcal{H}_2$-error\\
		\midrule
		\multirow{3}{*}{ACC}
		& \ourmethod & 12 & \num{2.29e-04} & \num{1.91e-05} & \num{8.71e-01}\\
		& \LMI & 13 & \num{4.61e-03} & \num{3.54e-04} & \num{8.71e-01}\\
		& \LMI-TP & 11 & \num{3.59e-02} & \num{3.26e-03} & \num{8.71e-01}\\
		\midrule
		\multirow{3}{*}{CD Player}
		& \ourmethod & 30 & \num{5.44e-1} & \num{1.81e-2} & \num{1.06e+06}\\
		& \ourmethod\tablefootnote{With smaller relative convergence tolarance ($\delta_{\mathcal{J}} = \num{1e-10}$)} & 13303 & \num{1.58e+02} & \num{1.18e-02} & \num{1.00e+06}\\
		% & LMI & 83 & $4.45e+02$ & $5.37e+00$ & $1.00e+06$\\
		& \LMI-TP & 116 & \num{6.04e+02} & \num{5.21e+00} & \num{1.00e+06}\\
		\midrule
		Smartphone & \ourmethod & 2208 & \num{1.46e2} & \num{6.63e-2} & \num{8.32e05}\\
		\bottomrule
	\end{tabular}
\end{table}

\subsection{ACC benchmark problem}
\label{subsec:ACC}
We consider the system described by the matrices
\begin{equation*}
	A= \left[\begin{array}{rrrr}
	-\tfrac{1}{4} & 1 & 0 & 0 \\
	0 & -\tfrac{1}{4} & 1 & 0 \\
	0 & 0 & -\tfrac{1}{4} & 1 \\
	0 & 0 & -2 & -\tfrac{1}{4} 
	\end{array}\right], \quad B= \begin{bmatrix}
		0 \\
		0 \\
		0 \\
		1 
	\end{bmatrix}, \quad C= \begin{bmatrix}
			1 & 0 & 0 & 0 
	\end{bmatrix}, \quad D=0.
\end{equation*}
This system is motivated by the ACC benchmark problem~\cite{WB92}, which was  also considered in~\cite{KJ97}. 
% Note that this system has a pole at $s=0$ we shift the works for asymptotically stable systems, we perturb the system by adding a small perturbation $\Delta_A=-\tfrac{1}{4}I_{\stateDim}$ to the state matrix $A= A_{\rm o} -\tfrac{1}{4}I_{\stateDim}$. 
Regardless of the choice for the initialization $L_0$, \ourmethod converges to a global minimum, with an $\Htwo$-error of $\big\|\transferFunction - \transferFunctionPas(\cdot;\reduce{C})\big\|_{\Htwo} \approx 1.03$.
However, if we add a feedthrough of $D=\tfrac{1}{8}$ we observe that for randomly initialized $L_0$, the optimization algorithm converges to a non-global local minimum in approximately 40\% of the cases.
This local minimum is characterized by a $\Htwo$-error of $\big\|\transferFunction - \transferFunctionPas(\cdot;\reduce{C})\big\|_{\Htwo} \approx 1.07$, while the global minimum is characterized by a $\Htwo$-error of $\big\|\transferFunction - \transferFunctionPas(\cdot;\reduce{C})\big\|_{\Htwo} \approx 0.87$.
In~\Cref{fig:acc} we show the optimization process for the ACC benchmark problem for $9$ random initializations and additional the \ARE initialization as proposed in~\Cref{alg:init}.  
In the case of the \ARE initialization, we add the artificial feedthrough of $\Delta_D = 1.14$ and initialize $L_0$ with the solution of the \ARE~\eqref{eqn:kyp:are} and \eqref{eqn:kyp:are:LM}. This initialization already achieves an an initial $\Htwo$-error of $\big\|\transferFunction - \transferFunctionPas(\cdot;\reduce{C})\big\|_{\Htwo} \approx 1.25$ and converges to the global minimum after 13 iterations. 
Out of the 9 random initializations, 5 converge to the global minimum and 4 converge to the local (non-global) minimum. However, when we employ our restart strategy, all initializations converge to the global minimum.
\begin{figure}[htb]
	\centering
	\input{plots/acc_optim_restart.tikz}
	\caption{Squared $\Htwo$ error for each iteration during the optimization for the ACC benchmark problem from \Cref{subsec:ACC}. The black thick line corresponds to the initialization strategy using the \ARE; discussed in \Cref{sec:passivation:init}. The colored lines correspond to random initializations. If the lines become dashed, then a non-global local minimum was detected and the restart strategy was applied.} 
	\label{fig:acc}
\end{figure}
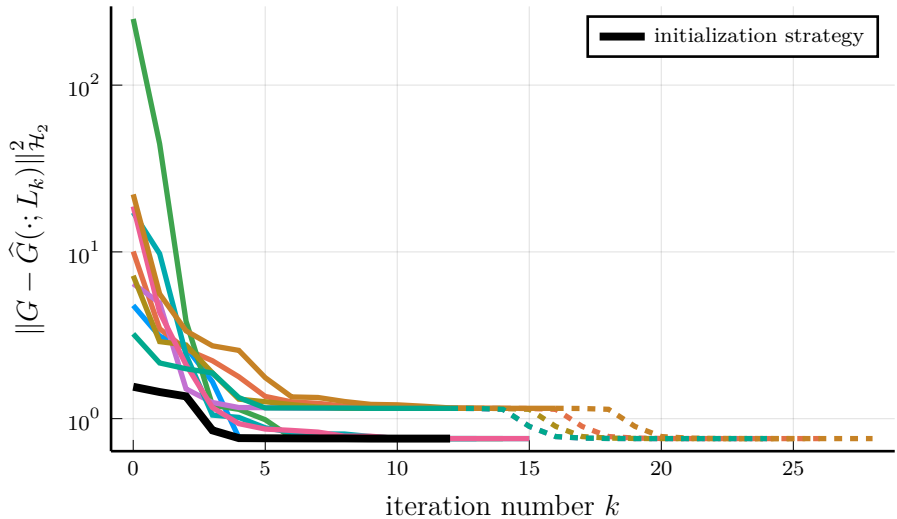

Finally, we compare runtime of the optimization process for \ourmethod (with initialization obtained by the initialization strategy), the standard \KYP lemma based \LMI and the \LMI with trace parametrization in~\Cref{tab:runtimes} (see the rows for the model ACC). We observe that, for this benchmark example, \ourmethod is faster by one order of magnitude compared to the other methods.

\subsection{CD Player}
Our second example is a model of the swing arm of a CD Player holding a lens which can be moved in the horizontal plane. The model is part of the   
SLICOT benchmark collection and was presented in~\cite{DSB92}. The system dimensions are $\stateDim = 120$ and $\inpVarDim = 2$ and the system has no feedthrough, i.e., $D=0$. Although the system is of medium size, the model is a challenging benchmark model, since the distance to passivity is large. This can be seen from sampling the modified Popov function  where we observe passivity violation of magnitude $\num{1e6}$.
In order to apply our initialization strategy, we need to perturb the feedthrough of the system by a perturbation with an order of magnitude of $\num{1e6}$. \ourmethod converges after 30 iterations within approximately half a second. The results are summarized in \Cref{tab:runtimes}.
In this example the standard \LMI approach converges, but the resulting model is numerically not passive and thus not shown in~\Cref{tab:runtimes}. If we use the trace parametrized version, the \LMI is solved after 116 iterations in approximately 600 seconds.
We notice, that the \LMI-TP approach achieves a smaller $\Htwo$-error than \ourmethod. This is due to our convergence criterion. If we set the relative tolerance of the objective to $\delta_{\mathcal{J}} = \num{1e-10}$, we achieve a similar $\Htwo$-error after 13303 iterations. This is still significantly faster than the \LMI-TP approach.

\subsection{High-speed smartphone interconnect link}
\label{sec:numerics:smartphone}
Our final example is a high-speed smartphone interconnect link presented in~\cite{GriS21}. The system dimensions are $\stateDim = 800$ and $\inpVarDim = 4$, the state matrix $A$ is dense, and the system was identified from frequency response  data via vector fitting. For more details we refer to~\cite{GriS21}. We will compare our passivated model obtained by \ourmethod, with the passivated model from~\cite{GriS21} obtained by perturbation of Hamiltonian eigenvalue problems~\cite{Gri04}. We refer to the latter as the reference model. 
Before we apply our optimization, we transform the system to diagonal form, i.e., a similarity transformation with the eigenvectors of $A$.
This significantly reduces the computational effort for the optimization as presented in~\Cref{sec:accelerations}. To give an intuition, solving the Lyapunov equation including the Schur decomposition for a random right-hand side takes around 550 \unit{\milli\second} before the transformation, with the Schur decomposition alone accounting for about 260 \unit{\milli\second}. After the transformation, however, the total time drops to under 4 \unit{\milli\second}.
After applying our optimization, which takes roughly 150 \unit{\second}, we achieve a passivated model that improves the $\Htwo$-error by approximately $31\%$ relative to the reference model in~\cite{GriS21}.
In~\Cref{fig:smartphone} we show the largest and smallest singular values of the transfer function of the original model, our passivated model and the reference model from~\cite{GriS21}. We observe in~\Cref{fig:smartphone:sigma}, that our passivated system captures the dynamics at least as well as the reference model. Indeed, in \Cref{fig:smartphone:sigmaerr} we  see the error of the singular values, and for most frequencies our model achieves a lower error than the reference model.

\begin{figure}[htb]
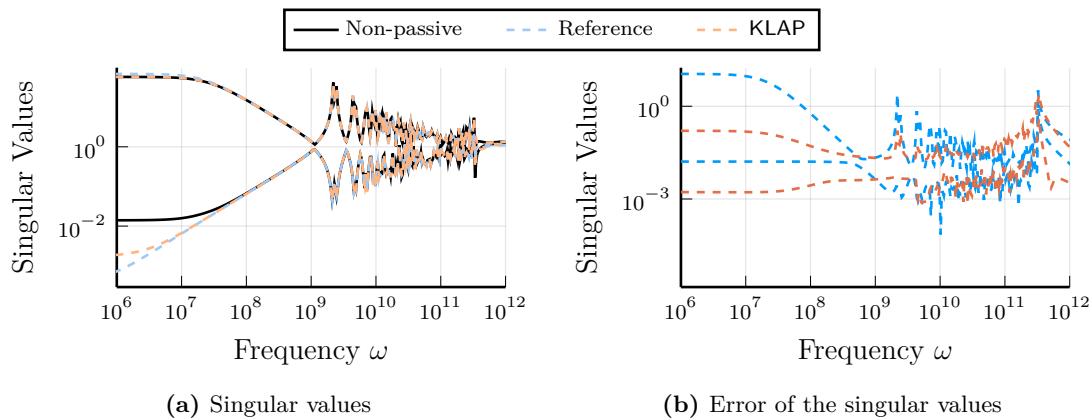

	\centering
	\ref{leg:smartphone}\\
	\begin{subfigure}{.5\linewidth}
		\centering
		\input{plots/smartphone_sigma.tikz}
		\caption{Singular values}
		\label{fig:smartphone:sigma}
	\end{subfigure}\hfill
	\begin{subfigure}{.5\linewidth}
		\centering
		\input{plots/smartphone_sigmaerr.tikz}
		\caption{Error of the singular values}
		\label{fig:smartphone:sigmaerr}
	\end{subfigure}
	\caption{Largest and smallest singular values of the transfer function over a frequency range for the high-speed smartphone interconnect link discussed in \Cref{sec:numerics:smartphone}.}
	\label{fig:smartphone}
\end{figure}

%-----------------------------------------------------------------------------%
\section{Conclusions}
\label{sec:conclusions}
We presented a novel approach, called~\ourmethod, to address the $\mathcal{H}_2$-optimal passivation problem for linear dynamical systems. The core of \ourmethod combines the \KYP or positive real lemma with the existence of rank-minimizing solutions. This combination allows us to explicitly parametrize the output matrix, ensuring that the system is passive by design. Using this parametrization, we formulated the optimization problem and examined its well-posedness. Additionally, we proposed an initialization strategy based on the \ARE related to the \KYP inequality. Since our parametrization may lead to the presence of non-global local minima, we established a criterion for identifying the global minima and introduced a restart strategy in case the algorithm is potentially stuck in a non-global local minimum.
We validated the effectiveness of our approach on three benchmark examples, demonstrating significant speedups over traditional methods based on the positive real lemma while achieving comparable accuracy in $\mathcal{H}_2$-error. Although not covered in this work, we expect that the same approach can be readily applied to the bounded real lemma to find the nearest contractive system.  Furthermore, we foresee the method being extendable to parametric systems and differential-algebraic equations. Also, the application of \ourmethod to the $\mathcal{H}_\infty$-optimal passivation problem is a promising direction for future research.

\subsection*{Code and data availability} 
\phantom{x}\\
\vspace{0.5cm}
\noindent\fbox{%
    \parbox{0.98\textwidth}{%
        The code and data used to generate the numerical results are accessible via
		\begin{center}
			\href{https://doi.org/10.5281/zenodo.14617036}{doi.org/10.5281/zenodo.14617036}
		\end{center}
		under MIT Common License.
    }%
}
\vspace{0.2cm}
%-----------------------------------------------------------------------------%
\subsection*{Acknowledgments}
The idea for this research project started during SG's Argyris visiting professorship in 2022 at the University of Stuttgart and major parts of this manuscript were written while BU was with the University of Stuttgart. JN and BU acknowledge funding from the DFG under Germany's Excellence Strategy -- EXC 2075 -- 390740016. JN, BU, and SG are thankful for support by the Stuttgart Center for Simulation Science (SimTech). SG's work was supported in part by the US National Science Foundation grant AMPS-2318880. BU is partly funded by the Deutsche Forschungsgemeinschaft (DFG, German Research Foundation) -- Project-ID 258734477 -- SFB 1173.
Moreover, the authors want to thank Prof.~Stefano Grivet-Talocia for providing the data for the high-speed smartphone interconnect link example.

\subsection*{CRediT author statement}
\textbf{Jonas Nicodemus:} Conceptualization, Methodology, Software, Investigation, Data Curation, Writing - Original Draft, Writing - Review \& Editing; \textbf{Matthias Voigt:} Writing - Original Draft, Writing - Review \& Editing ; \textbf{Serkan Gugercin:} Writing - Original Draft, Writing - Review \& Editing; \textbf{Benjamin Unger:} Conceptualization, Methodology, Investigation, Writing - Original Draft, 
Writing - Review \& Editing, Supervision, Funding acquisition.

%-----------------------------------------------------------------------------%
\bibliographystyle{plain-doi}
\bibliography{journalabbr, literature, NGUV24}

\end{document}

%% file: def.tex
\newcommand{\N}{\ensuremath\mathbb{N}}

\newcommand{\R}{\ensuremath\mathbb{R}}
\newcommand{\C}{\ensuremath\mathbb{C}}

\newcommand{\T}{\ensuremath\mathsf{T}}
\newcommand*{\tran}{^{\mkern-1.5mu\mathsf{T}}}

\newcommand*{\herm}{^{\mathsf{H}}}

\newcommand{\Spd}[1]{\mathcal{S}^{#1}_{\succ}}
\newcommand{\Spsd}[1]{\mathcal{S}^{#1}_{\succeq}}
\newenvironment{smallbmatrix}{\left[\begin{smallmatrix}}{\end{smallmatrix}\right]}

\newcommand{\dt}{\,\mathrm{d}t}

\newcommand{\domega}{\,\mathrm{d}\omega}

\DeclareMathOperator{\diag}{diag}

\DeclareMathOperator{\tr}{tr}

\DeclareMathOperator{\sign}{sign}

\newcommand{\calH}{\mathcal{H}}

\newcommand{\calJ}{\mathcal{J}}

\newcommand{\calL}{\mathcal{L}}

\newcommand{\calO}{\mathcal{O}}
\newcommand{\calP}{\mathcal{P}}

\newcommand{\calW}{\mathcal{W}}
\newcommand{\calX}{\mathcal{X}}

\newcommand{\Htwo}{\calH_2}

\definecolor{mycolor1}{rgb}{0.00000,0.44700,0.74100}% blue
\definecolor{mycolor2}{rgb}{0.85000,0.32500,0.09800}% red
\definecolor{mycolor3}{rgb}{0.92900,0.69400,0.12500}% orange/yellow
\definecolor{mycolor4}{rgb}{0.46600,0.67400,0.18800}% green
\definecolor{mycolor5}{rgb}{0.49400,0.18400,0.55600}% purple

\definecolor{nicegreen}{rgb}{0.3,0.7,0.4}

\newcommand{\system}{\Sigma}
\newcommand{\state}{x}
\newcommand{\stateDim}{n}
\newcommand{\reduce}[1]{\widehat{#1}}

\newcommand{\stateDimRed}{r}
\newcommand{\transferFunction}{G}

\newcommand{\inpVar}{u}

\newcommand{\inpVarDim}{m}

\newcommand{\outVar}{y}

\newcommand{\feasibleSet}{\reduce{\mathcal{C}}}
\newcommand{\systemPas}{\reduce{\system}}
\newcommand{\transferFunctionPas}{\reduce{\transferFunction}}

\newcommand{\err}[1]{#1_{\mathrm{e}}}
\newcommand{\systemErr}{\err{\system}}
\newcommand{\Aerr}{\err{A}}
\newcommand{\Berr}{\err{B}}
\newcommand{\Cerr}{\err{C}}
\newcommand{\Perr}{\err{\ctrlG}}

\newcommand{\ctrlG}{\calP}
\newcommand{\ctrlGweighted}{\ctrlG_{\mathrm{w}}}

\newcommand{\objFunc}{\mathcal{J}}

\newcommand{\Xmin}{X_{\min}}
\newcommand{\Xmax}{X_{\max}}

\newcommand{\acronym}[1]{\textsf{#1}\xspace}
\newcommand{\LTI}{\textsf{LTI}\xspace} % linear time-invariant
\newcommand{\KYP}{\textsf{KYP}\xspace} % Kalman-Yakubovich-Popov
\newcommand{\ARE}{\acronym{ARE}} % algebraic Riccati equation
\newcommand{\LMI}{\acronym{LMI}} % linear matrix inequality

\newcommand{\ourmethod}{\acronym{KLAP}}
\newcommand{\ourmethodfull}{KYP lemma based low-rank approximation for $\mathcal{H}_2$-optimal passivation (\acronym{KLAP})}

\renewcommand{\i}{\mathrm{i}}
\DeclareMathOperator{\Real}{Re}

%% file: plots/lspace_D=0.0.tikz
% Recommended preamble:
% \usetikzlibrary{arrows.meta}
% \usetikzlibrary{backgrounds}
% \usepgfplotslibrary{patchplots}
% \usepgfplotslibrary{fillbetween}
% \pgfplotsset{%
%     layers/standard/.define layer set={%
%         background,axis background,axis grid,axis ticks,axis lines,axis tick labels,pre main,main,axis descriptions,axis foreground%
%     }{
%         grid style={/pgfplots/on layer=axis grid},%
%         tick style={/pgfplots/on layer=axis ticks},%
%         axis line style={/pgfplots/on layer=axis lines},%
%         label style={/pgfplots/on layer=axis descriptions},%
%         legend style={/pgfplots/on layer=axis descriptions},%
%         title style={/pgfplots/on layer=axis descriptions},%
%         colorbar style={/pgfplots/on layer=axis descriptions},%
%         ticklabel style={/pgfplots/on layer=axis tick labels},%
%         axis background@ style={/pgfplots/on layer=axis background},%
%         3d box foreground style={/pgfplots/on layer=axis foreground},%
%     },
% }

\begin{tikzpicture}[/tikz/background rectangle/.style={fill={rgb,1:red,1.0;green,1.0;blue,1.0}, fill opacity={1.0}, draw opacity={1.0}}, show background rectangle]
\begin{axis}[point meta max={nan}, point meta min={nan}, legend cell align={left}, legend columns={2}, title={}, title style={at={{(0.5,1)}}, anchor={south}, font={{\fontsize{14 pt}{18.2 pt}\selectfont}}, color={rgb,1:red,0.0;green,0.0;blue,0.0}, draw opacity={1.0}, rotate={0.0}, align={center}}, legend style={/tikz/every even column/.append style={column sep=0.5cm}, color={rgb,1:red,0.0;green,0.0;blue,0.0}, draw opacity={1.0}, line width={1}, solid, fill={rgb,1:red,1.0;green,1.0;blue,1.0}, fill opacity={1.0}, text opacity={1.0}, font={{\fontsize{8 pt}{10.4 pt}\selectfont}}, text={rgb,1:red,0.0;green,0.0;blue,0.0}, cells={anchor={center}}, at={(0.98, 0.98)}, anchor={north east}}, axis background/.style={fill={rgb,1:red,1.0;green,1.0;blue,1.0}, opacity={1.0}}, anchor={north west}, xshift={1.0mm}, yshift={-1.0mm}, width={\textwidth}, height={\textwidth}, scaled x ticks={false}, xlabel={$l_1^{\phantom{2}}$}, x tick style={color={rgb,1:red,0.0;green,0.0;blue,0.0}, opacity={1.0}}, x tick label style={color={rgb,1:red,0.0;green,0.0;blue,0.0}, opacity={1.0}, rotate={0}}, xlabel style={at={(ticklabel cs:0.5)}, anchor=near ticklabel, at={{(ticklabel cs:0.5)}}, anchor={near ticklabel}, font={{\fontsize{11 pt}{14.3 pt}\selectfont}}, color={rgb,1:red,0.0;green,0.0;blue,0.0}, draw opacity={1.0}, rotate={0.0}}, xmajorgrids={true}, xmin={-2}, xmax={2}, xticklabels={{$-2$,$-1$,$0$,$1$,$2$}}, xtick={{-2.0,-1.0,0.0,1.0,2.0}}, xtick align={inside}, xticklabel style={font={{\fontsize{8 pt}{10.4 pt}\selectfont}}, color={rgb,1:red,0.0;green,0.0;blue,0.0}, draw opacity={1.0}, rotate={0.0}}, x grid style={color={rgb,1:red,0.0;green,0.0;blue,0.0}, draw opacity={0.1}, line width={0.5}, solid}, axis x line*={left}, x axis line style={color={rgb,1:red,0.0;green,0.0;blue,0.0}, draw opacity={1.0}, line width={1}, solid}, scaled y ticks={false}, ylabel={$l_2^{\phantom{2}}$}, y tick style={color={rgb,1:red,0.0;green,0.0;blue,0.0}, opacity={1.0}}, y tick label style={color={rgb,1:red,0.0;green,0.0;blue,0.0}, opacity={1.0}, rotate={0}}, ylabel style={at={(ticklabel cs:0.5)}, anchor=near ticklabel, at={{(ticklabel cs:0.5)}}, anchor={near ticklabel}, font={{\fontsize{11 pt}{14.3 pt}\selectfont}}, color={rgb,1:red,0.0;green,0.0;blue,0.0}, draw opacity={1.0}, rotate={0.0}}, ymajorgrids={true}, ymin={-2}, ymax={2}, yticklabels={{$-2$,$-1$,$0$,$1$,$2$}}, ytick={{-2.0,-1.0,0.0,1.0,2.0}}, ytick align={inside}, yticklabel style={font={{\fontsize{8 pt}{10.4 pt}\selectfont}}, color={rgb,1:red,0.0;green,0.0;blue,0.0}, draw opacity={1.0}, rotate={0.0}}, y grid style={color={rgb,1:red,0.0;green,0.0;blue,0.0}, draw opacity={0.1}, line width={0.5}, solid}, axis y line*={left}, y axis line style={color={rgb,1:red,0.0;green,0.0;blue,0.0}, draw opacity={1.0}, line width={1}, solid}, colorbar={false}, legend to name={leg:toy0}]
    \addplot[forget plot]
        graphics[xmin={-3}, xmax={3}, ymin={-2}, ymax={2}] {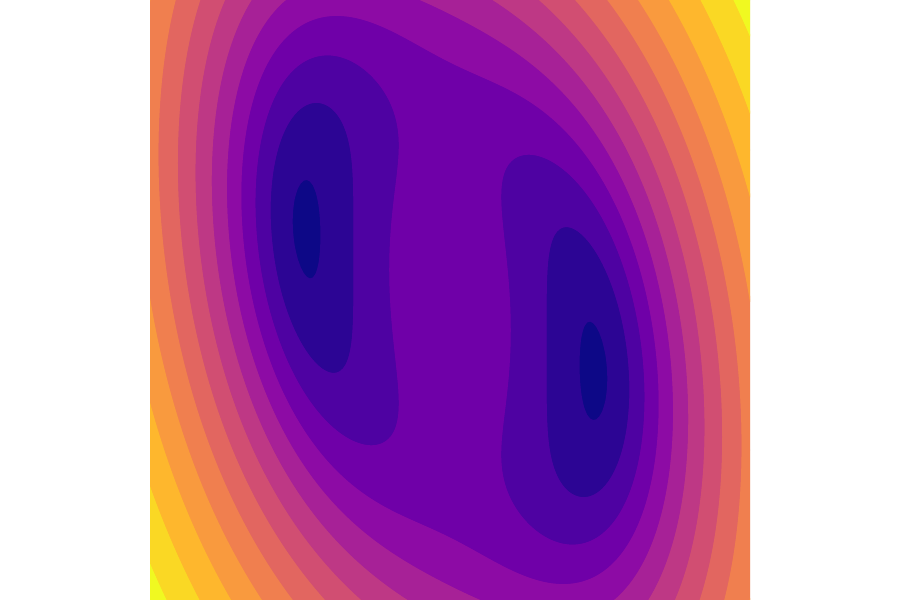}
        ;
    \addplot[color={rgb,1:red,0.6314;green,0.7882;blue,0.9569}, name path={79}, only marks, draw opacity={1.0}, line width={0}, solid, mark={*}, mark size={2.75 pt}, mark repeat={1}, mark options={color={rgb,1:red,0.0;green,0.0;blue,0.0}, draw opacity={0.0}, fill={rgb,1:red,0.6314;green,0.7882;blue,0.9569}, fill opacity={1.0}, line width={0.75}, rotate={0}, solid}]
        table[row sep={\\}]
        {
            \\
            -0.9607689228307221  0.48038446141514685  \\
        }
        ;
    \addlegendentry {$L_1^\star/\widehat{C}_1^\star$}
    \addplot[color={rgb,1:red,1.0;green,0.7059;blue,0.5098}, name path={80}, only marks, draw opacity={1.0}, line width={0}, solid, mark={*}, mark size={2.75 pt}, mark repeat={1}, mark options={color={rgb,1:red,0.0;green,0.0;blue,0.0}, draw opacity={0.0}, fill={rgb,1:red,1.0;green,0.7059;blue,0.5098}, fill opacity={1.0}, line width={0.75}, rotate={0}, solid}]
        table[row sep={\\}]
        {
            \\
            0.9607689228305435  -0.48038446141524266  \\
        }
        ;
    \addlegendentry {$L_2^\star/\widehat{C}_2^\star$}
\end{axis}
\end{tikzpicture}

%% file: plots/cspace_D=0.0.tikz
% Recommended preamble:
% \usetikzlibrary{arrows.meta}
% \usetikzlibrary{backgrounds}
% \usepgfplotslibrary{patchplots}
% \usepgfplotslibrary{fillbetween}
% \pgfplotsset{%
%     layers/standard/.define layer set={%
%         background,axis background,axis grid,axis ticks,axis lines,axis tick labels,pre main,main,axis descriptions,axis foreground%
%     }{
%         grid style={/pgfplots/on layer=axis grid},%
%         tick style={/pgfplots/on layer=axis ticks},%
%         axis line style={/pgfplots/on layer=axis lines},%
%         label style={/pgfplots/on layer=axis descriptions},%
%         legend style={/pgfplots/on layer=axis descriptions},%
%         title style={/pgfplots/on layer=axis descriptions},%
%         colorbar style={/pgfplots/on layer=axis descriptions},%
%         ticklabel style={/pgfplots/on layer=axis tick labels},%
%         axis background@ style={/pgfplots/on layer=axis background},%
%         3d box foreground style={/pgfplots/on layer=axis foreground},%
%     },
% }

\begin{tikzpicture}[/tikz/background rectangle/.style={fill={rgb,1:red,1.0;green,1.0;blue,1.0}, fill opacity={1.0}, draw opacity={1.0}}, show background rectangle]
\begin{axis}[point meta max={nan}, point meta min={nan}, legend cell align={left}, legend columns={1}, title={}, title style={at={{(0.5,1)}}, anchor={south}, font={{\fontsize{14 pt}{18.2 pt}\selectfont}}, color={rgb,1:red,0.0;green,0.0;blue,0.0}, draw opacity={1.0}, rotate={0.0}, align={center}}, legend style={color={rgb,1:red,0.0;green,0.0;blue,0.0}, draw opacity={1.0}, line width={1}, solid, fill={rgb,1:red,1.0;green,1.0;blue,1.0}, fill opacity={1.0}, text opacity={1.0}, font={{\fontsize{8 pt}{10.4 pt}\selectfont}}, text={rgb,1:red,0.0;green,0.0;blue,0.0}, cells={anchor={center}}, at={(1.02, 1)}, anchor={north west}}, axis background/.style={fill={rgb,1:red,1.0;green,1.0;blue,1.0}, opacity={1.0}}, anchor={north west}, xshift={1.0mm}, yshift={-1.0mm}, width={\textwidth}, height={\textwidth}, scaled x ticks={false}, xlabel={$\reduce{c}_1$}, x tick style={color={rgb,1:red,0.0;green,0.0;blue,0.0}, opacity={1.0}}, x tick label style={color={rgb,1:red,0.0;green,0.0;blue,0.0}, opacity={1.0}, rotate={0}}, xlabel style={at={(ticklabel cs:0.5)}, anchor=near ticklabel, at={{(ticklabel cs:0.5)}}, anchor={near ticklabel}, font={{\fontsize{11 pt}{14.3 pt}\selectfont}}, color={rgb,1:red,0.0;green,0.0;blue,0.0}, draw opacity={1.0}, rotate={0.0}}, xmajorgrids={true}, xmin={-0.5}, xmax={1.5}, xticklabels={{$-0.5$,$0.0$,$0.5$,$1.0$,$1.5$}}, xtick={{-0.5,0.0,0.5,1.0,1.5}}, xtick align={inside}, xticklabel style={font={{\fontsize{8 pt}{10.4 pt}\selectfont}}, color={rgb,1:red,0.0;green,0.0;blue,0.0}, draw opacity={1.0}, rotate={0.0}}, x grid style={color={rgb,1:red,0.0;green,0.0;blue,0.0}, draw opacity={0.1}, line width={0.5}, solid}, axis x line*={left}, x axis line style={color={rgb,1:red,0.0;green,0.0;blue,0.0}, draw opacity={1.0}, line width={1}, solid}, scaled y ticks={false}, ylabel={$\reduce{c}_2$}, y tick style={color={rgb,1:red,0.0;green,0.0;blue,0.0}, opacity={1.0}}, y tick label style={color={rgb,1:red,0.0;green,0.0;blue,0.0}, opacity={1.0}, rotate={0}}, ylabel style={at={(ticklabel cs:0.5)}, anchor=near ticklabel, at={{(ticklabel cs:0.5)}}, anchor={near ticklabel}, font={{\fontsize{11 pt}{14.3 pt}\selectfont}}, color={rgb,1:red,0.0;green,0.0;blue,0.0}, draw opacity={1.0}, rotate={0.0}}, ymajorgrids={true}, ymin={-0.5}, ymax={1.5}, yticklabels={{$-0.5$,$0.0$,$0.5$,$1.0$,$1.5$}}, ytick={{-0.5,0.0,0.5,1.0,1.5}}, ytick align={inside}, yticklabel style={font={{\fontsize{8 pt}{10.4 pt}\selectfont}}, color={rgb,1:red,0.0;green,0.0;blue,0.0}, draw opacity={1.0}, rotate={0.0}}, y grid style={color={rgb,1:red,0.0;green,0.0;blue,0.0}, draw opacity={0.1}, line width={0.5}, solid}, axis y line*={left}, y axis line style={color={rgb,1:red,0.0;green,0.0;blue,0.0}, draw opacity={1.0}, line width={1}, solid}, colorbar={false}]
    \addplot[forget plot]
        graphics[xmin={-1.0}, xmax={2.0}, ymin={-0.5}, ymax={1.5}] {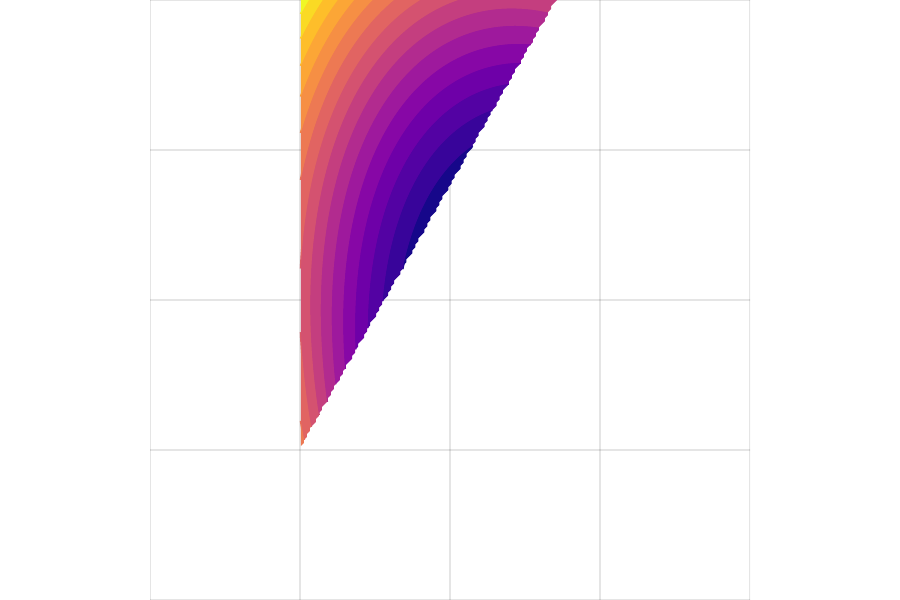}
        ;
    \addplot[color={rgb,1:red,0.6314;green,0.7882;blue,0.9569}, name path={82}, only marks, draw opacity={1.0}, line width={0}, solid, mark={*}, mark size={2.75 pt}, mark repeat={1}, mark options={color={rgb,1:red,0.0;green,0.0;blue,0.0}, draw opacity={0.0}, fill={rgb,1:red,0.6314;green,0.7882;blue,0.9569}, fill opacity={1.0}, line width={0.75}, rotate={0}, solid}]
        table[row sep={\\}]
        {
            \\
            0.4615384615386533  0.8076923076926434  \\
        }
        ;
    \addplot[color={rgb,1:red,1.0;green,0.7059;blue,0.5098}, name path={83}, only marks, draw opacity={1.0}, line width={0}, solid, mark={*}, mark size={2.75 pt}, mark repeat={1}, mark options={color={rgb,1:red,0.0;green,0.0;blue,0.0}, draw opacity={0.0}, fill={rgb,1:red,1.0;green,0.7059;blue,0.5098}, fill opacity={1.0}, line width={0.75}, rotate={0}, solid}]
        table[row sep={\\}]
        {
            \\
            0.46153846153848166  0.8076923076923428  \\
        }
        ;
\end{axis}
\end{tikzpicture}

%% file: plots/lspace_D=0.125.tikz
% Recommended preamble:
% \usetikzlibrary{arrows.meta}
% \usetikzlibrary{backgrounds}
% \usepgfplotslibrary{patchplots}
% \usepgfplotslibrary{fillbetween}
% \pgfplotsset{%
%     layers/standard/.define layer set={%
%         background,axis background,axis grid,axis ticks,axis lines,axis tick labels,pre main,main,axis descriptions,axis foreground%
%     }{
%         grid style={/pgfplots/on layer=axis grid},%
%         tick style={/pgfplots/on layer=axis ticks},%
%         axis line style={/pgfplots/on layer=axis lines},%
%         label style={/pgfplots/on layer=axis descriptions},%
%         legend style={/pgfplots/on layer=axis descriptions},%
%         title style={/pgfplots/on layer=axis descriptions},%
%         colorbar style={/pgfplots/on layer=axis descriptions},%
%         ticklabel style={/pgfplots/on layer=axis tick labels},%
%         axis background@ style={/pgfplots/on layer=axis background},%
%         3d box foreground style={/pgfplots/on layer=axis foreground},%
%     },
% }

\begin{tikzpicture}[/tikz/background rectangle/.style={fill={rgb,1:red,1.0;green,1.0;blue,1.0}, fill opacity={1.0}, draw opacity={1.0}}, show background rectangle]
\begin{axis}[point meta max={nan}, point meta min={nan}, legend cell align={left}, legend columns={3}, title={}, title style={at={{(0.5,1)}}, anchor={south}, font={{\fontsize{14 pt}{18.2 pt}\selectfont}}, color={rgb,1:red,0.0;green,0.0;blue,0.0}, draw opacity={1.0}, rotate={0.0}, align={center}}, 
legend style={/tikz/every even column/.append style={column sep=0.5cm}, color={rgb,1:red,0.0;green,0.0;blue,0.0}, draw opacity={1.0}, line width={1}, solid, fill={rgb,1:red,1.0;green,1.0;blue,1.0}, fill opacity={1.0}, text opacity={1.0}, font={{\fontsize{8 pt}{10.4 pt}\selectfont}}, text={rgb,1:red,0.0;green,0.0;blue,0.0}, cells={anchor={center}}, at={(0.98, 0.98)}, 
anchor={north east}}, axis background/.style={fill={rgb,1:red,1.0;green,1.0;blue,1.0}, opacity={1.0}}, anchor={north west}, xshift={1.0mm}, yshift={-1.0mm}, width={\textwidth}, height={\textwidth}, scaled x ticks={false}, xlabel={$l_1^{\phantom{2}}$}, x tick style={color={rgb,1:red,0.0;green,0.0;blue,0.0}, opacity={1.0}}, x tick label style={color={rgb,1:red,0.0;green,0.0;blue,0.0}, opacity={1.0}, rotate={0}}, xlabel style={at={(ticklabel cs:0.5)}, anchor=near ticklabel, at={{(ticklabel cs:0.5)}}, anchor={near ticklabel}, font={{\fontsize{11 pt}{14.3 pt}\selectfont}}, color={rgb,1:red,0.0;green,0.0;blue,0.0}, draw opacity={1.0}, rotate={0.0}}, xmajorgrids={true}, xmin={-2}, xmax={2}, xticklabels={{$-2$,$-1$,$0$,$1$,$2$}}, xtick={{-2.0,-1.0,0.0,1.0,2.0}}, xtick align={inside}, xticklabel style={font={{\fontsize{8 pt}{10.4 pt}\selectfont}}, color={rgb,1:red,0.0;green,0.0;blue,0.0}, draw opacity={1.0}, rotate={0.0}}, x grid style={color={rgb,1:red,0.0;green,0.0;blue,0.0}, draw opacity={0.1}, line width={0.5}, solid}, axis x line*={left}, x axis line style={color={rgb,1:red,0.0;green,0.0;blue,0.0}, draw opacity={1.0}, line width={1}, solid}, scaled y ticks={false}, ylabel={$l_2^{\phantom{2}}$}, y tick style={color={rgb,1:red,0.0;green,0.0;blue,0.0}, opacity={1.0}}, y tick label style={color={rgb,1:red,0.0;green,0.0;blue,0.0}, opacity={1.0}, rotate={0}}, ylabel style={at={(ticklabel cs:0.5)}, anchor=near ticklabel, at={{(ticklabel cs:0.5)}}, anchor={near ticklabel}, font={{\fontsize{11 pt}{14.3 pt}\selectfont}}, color={rgb,1:red,0.0;green,0.0;blue,0.0}, draw opacity={1.0}, rotate={0.0}}, ymajorgrids={true}, ymin={-2}, ymax={2}, yticklabels={{$-2$,$-1$,$0$,$1$,$2$}}, ytick={{-2.0,-1.0,0.0,1.0,2.0}}, ytick align={inside}, yticklabel style={font={{\fontsize{8 pt}{10.4 pt}\selectfont}}, color={rgb,1:red,0.0;green,0.0;blue,0.0}, draw opacity={1.0}, rotate={0.0}}, y grid style={color={rgb,1:red,0.0;green,0.0;blue,0.0}, draw opacity={0.1}, line width={0.5}, solid}, axis y line*={left}, y axis line style={color={rgb,1:red,0.0;green,0.0;blue,0.0}, draw opacity={1.0}, line width={1}, solid}, colorbar={false}, legend to name={leg:toy}]
    \addplot[forget plot]
        graphics[xmin={-3}, xmax={3}, ymin={-2}, ymax={2}] {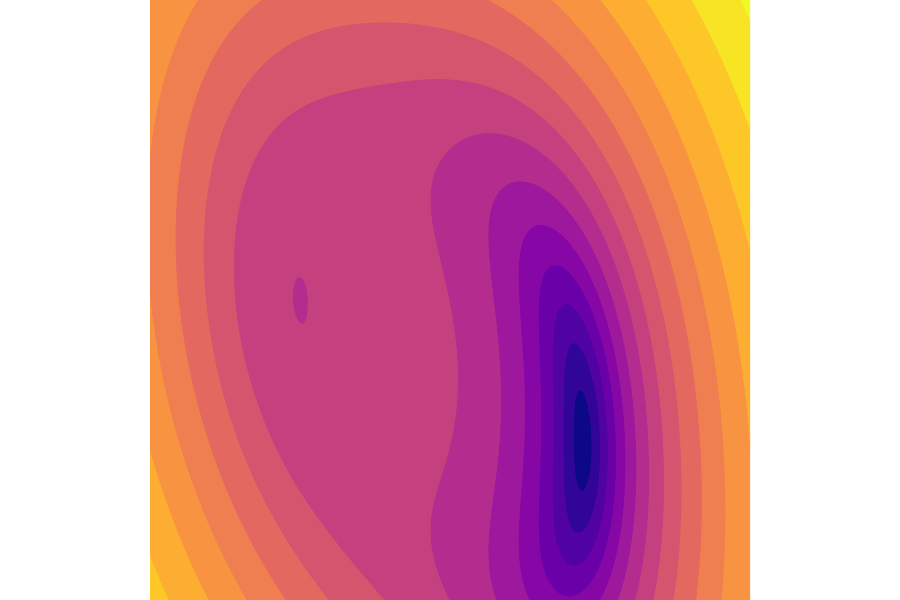}
        ;
    \addplot[color={rgb,1:red,0.6314;green,0.7882;blue,0.9569}, name path={27}, only marks, draw opacity={1.0}, line width={0}, solid, mark={*}, mark size={2.75 pt}, mark repeat={1}, mark options={color={rgb,1:red,0.0;green,0.0;blue,0.0}, draw opacity={0.0}, fill={rgb,1:red,0.6314;green,0.7882;blue,0.9569}, fill opacity={1.0}, line width={0.75}, rotate={0}, solid}]
        table[row sep={\\}]
        {
            \\
            -1.000000000002581  2.529703845881027e-11  \\
        }
        ;
    \addlegendentry {$L_{\mathrm{loc}}^\star/\widehat{C}_{\mathrm{loc}}^\star$}
    \addplot[color={rgb,1:red,1.0;green,0.7059;blue,0.5098}, name path={28}, only marks, draw opacity={1.0}, line width={0}, solid, mark={*}, mark size={2.75 pt}, mark repeat={1}, mark options={color={rgb,1:red,0.0;green,0.0;blue,0.0}, draw opacity={0.0}, fill={rgb,1:red,1.0;green,0.7059;blue,0.5098}, fill opacity={1.0}, line width={0.75}, rotate={0}, solid}]
        table[row sep={\\}]
        {
            \\
            0.885816136226263  -0.9429080681130648  \\
        }
        ;
    \addlegendentry {$L^\star/\widehat{C}^\star$}
\end{axis}
\end{tikzpicture}

%% file: plots/cspace_D=0.125.tikz
% Recommended preamble:
% \usetikzlibrary{arrows.meta}
% \usetikzlibrary{backgrounds}
% \usepgfplotslibrary{patchplots}
% \usepgfplotslibrary{fillbetween}
% \pgfplotsset{%
%     layers/standard/.define layer set={%
%         background,axis background,axis grid,axis ticks,axis lines,axis tick labels,pre main,main,axis descriptions,axis foreground%
%     }{
%         grid style={/pgfplots/on layer=axis grid},%
%         tick style={/pgfplots/on layer=axis ticks},%
%         axis line style={/pgfplots/on layer=axis lines},%
%         label style={/pgfplots/on layer=axis descriptions},%
%         legend style={/pgfplots/on layer=axis descriptions},%
%         title style={/pgfplots/on layer=axis descriptions},%
%         colorbar style={/pgfplots/on layer=axis descriptions},%
%         ticklabel style={/pgfplots/on layer=axis tick labels},%
%         axis background@ style={/pgfplots/on layer=axis background},%
%         3d box foreground style={/pgfplots/on layer=axis foreground},%
%     },
% }

\begin{tikzpicture}[/tikz/background rectangle/.style={fill={rgb,1:red,1.0;green,1.0;blue,1.0}, fill opacity={1.0}, draw opacity={1.0}}, show background rectangle]
\begin{axis}[point meta max={nan}, point meta min={nan}, legend cell align={left}, legend columns={1}, title={}, title style={at={{(0.5,1)}}, anchor={south}, font={{\fontsize{14 pt}{18.2 pt}\selectfont}}, color={rgb,1:red,0.0;green,0.0;blue,0.0}, draw opacity={1.0}, rotate={0.0}, align={center}}, legend style={color={rgb,1:red,0.0;green,0.0;blue,0.0}, draw opacity={1.0}, line width={1}, solid, fill={rgb,1:red,1.0;green,1.0;blue,1.0}, fill opacity={1.0}, text opacity={1.0}, font={{\fontsize{8 pt}{10.4 pt}\selectfont}}, text={rgb,1:red,0.0;green,0.0;blue,0.0}, cells={anchor={center}}, at={(1.02, 1)}, anchor={north west}}, axis background/.style={fill={rgb,1:red,1.0;green,1.0;blue,1.0}, opacity={1.0}}, anchor={north west}, xshift={1.0mm}, yshift={-1.0mm}, width={\textwidth}, height={\textwidth}, scaled x ticks={false}, xlabel={$\reduce{c}_1$}, x tick style={color={rgb,1:red,0.0;green,0.0;blue,0.0}, opacity={1.0}}, x tick label style={color={rgb,1:red,0.0;green,0.0;blue,0.0}, opacity={1.0}, rotate={0}}, xlabel style={at={(ticklabel cs:0.5)}, anchor=near ticklabel, at={{(ticklabel cs:0.5)}}, anchor={near ticklabel}, font={{\fontsize{11 pt}{14.3 pt}\selectfont}}, color={rgb,1:red,0.0;green,0.0;blue,0.0}, draw opacity={1.0}, rotate={0.0}}, xmajorgrids={true}, xmin={-0.5}, xmax={1.5}, xticklabels={{$-0.5$,$0.0$,$0.5$,$1.0$,$1.5$}}, xtick={{-0.5,0.0,0.5,1.0,1.5}}, xtick align={inside}, xticklabel style={font={{\fontsize{8 pt}{10.4 pt}\selectfont}}, color={rgb,1:red,0.0;green,0.0;blue,0.0}, draw opacity={1.0}, rotate={0.0}}, x grid style={color={rgb,1:red,0.0;green,0.0;blue,0.0}, draw opacity={0.1}, line width={0.5}, solid}, axis x line*={left}, x axis line style={color={rgb,1:red,0.0;green,0.0;blue,0.0}, draw opacity={1.0}, line width={1}, solid}, scaled y ticks={false}, ylabel={$\reduce{c}_2$}, y tick style={color={rgb,1:red,0.0;green,0.0;blue,0.0}, opacity={1.0}}, y tick label style={color={rgb,1:red,0.0;green,0.0;blue,0.0}, opacity={1.0}, rotate={0}}, ylabel style={at={(ticklabel cs:0.5)}, anchor=near ticklabel, at={{(ticklabel cs:0.5)}}, anchor={near ticklabel}, font={{\fontsize{11 pt}{14.3 pt}\selectfont}}, color={rgb,1:red,0.0;green,0.0;blue,0.0}, draw opacity={1.0}, rotate={0.0}}, ymajorgrids={true}, ymin={-0.5}, ymax={1.5}, yticklabels={{$-0.5$,$0.0$,$0.5$,$1.0$,$1.5$}}, ytick={{-0.5,0.0,0.5,1.0,1.5}}, ytick align={inside}, yticklabel style={font={{\fontsize{8 pt}{10.4 pt}\selectfont}}, color={rgb,1:red,0.0;green,0.0;blue,0.0}, draw opacity={1.0}, rotate={0.0}}, y grid style={color={rgb,1:red,0.0;green,0.0;blue,0.0}, draw opacity={0.1}, line width={0.5}, solid}, axis y line*={left}, y axis line style={color={rgb,1:red,0.0;green,0.0;blue,0.0}, draw opacity={1.0}, line width={1}, solid}, colorbar={false}]
    \addplot[forget plot]
        graphics[xmin={-1.0}, xmax={2.0}, ymin={-0.5}, ymax={1.5}] {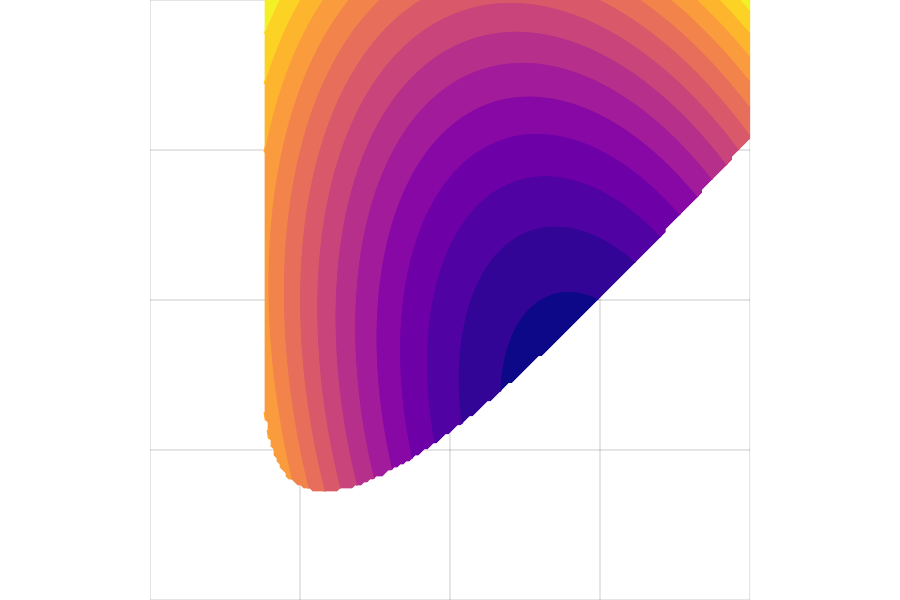}
        ;
    \addplot[color={rgb,1:red,0.6314;green,0.7882;blue,0.9569}, name path={30}, only marks, draw opacity={1.0}, line width={0}, solid, mark={*}, mark size={2.75 pt}, mark repeat={1}, mark options={color={rgb,1:red,0.0;green,0.0;blue,0.0}, draw opacity={0.0}, fill={rgb,1:red,0.6314;green,0.7882;blue,0.9569}, fill opacity={1.0}, line width={0.75}, rotate={0}, solid}]
        table[row sep={\\}]
        {
            \\
            1.290745288429207e-12  1.0000000000051628  \\
        }
        ;
    \addplot[color={rgb,1:red,1.0;green,0.7059;blue,0.5098}, name path={31}, only marks, draw opacity={1.0}, line width={0}, solid, mark={*}, mark size={2.75 pt}, mark repeat={1}, mark options={color={rgb,1:red,0.0;green,0.0;blue,0.0}, draw opacity={0.0}, fill={rgb,1:red,1.0;green,0.7059;blue,0.5098}, fill opacity={1.0}, line width={0.75}, rotate={0}, solid}]
        table[row sep={\\}]
        {
            \\
            0.8352431817125443  0.3401324147424071  \\
        }
        ;
\end{axis}
\end{tikzpicture}

%% file: plots/lspace_restart_D=0.125.tikz
% Recommended preamble:
% \usetikzlibrary{arrows.meta}
% \usetikzlibrary{backgrounds}
% \usepgfplotslibrary{patchplots}
% \usepgfplotslibrary{fillbetween}
% \pgfplotsset{%
%     layers/standard/.define layer set={%
%         background,axis background,axis grid,axis ticks,axis lines,axis tick labels,pre main,main,axis descriptions,axis foreground%
%     }{
%         grid style={/pgfplots/on layer=axis grid},%
%         tick style={/pgfplots/on layer=axis ticks},%
%         axis line style={/pgfplots/on layer=axis lines},%
%         label style={/pgfplots/on layer=axis descriptions},%
%         legend style={/pgfplots/on layer=axis descriptions},%
%         title style={/pgfplots/on layer=axis descriptions},%
%         colorbar style={/pgfplots/on layer=axis descriptions},%
%         ticklabel style={/pgfplots/on layer=axis tick labels},%
%         axis background@ style={/pgfplots/on layer=axis background},%
%         3d box foreground style={/pgfplots/on layer=axis foreground},%
%     },
% }

\begin{tikzpicture}[/tikz/background rectangle/.style={fill={rgb,1:red,1.0;green,1.0;blue,1.0}, fill opacity={1.0}, draw opacity={1.0}}, show background rectangle]
\begin{axis}[point meta max={nan}, point meta min={nan}, legend cell align={left}, legend columns={3}, title={}, title style={at={{(0.5,1)}}, anchor={south}, font={{\fontsize{14 pt}{18.2 pt}\selectfont}}, color={rgb,1:red,0.0;green,0.0;blue,0.0}, draw opacity={1.0}, rotate={0.0}, align={center}}, 
legend style={/tikz/every even column/.append style={column sep=0.5cm}, color={rgb,1:red,0.0;green,0.0;blue,0.0}, draw opacity={1.0}, line width={1}, solid, fill={rgb,1:red,1.0;green,1.0;blue,1.0}, fill opacity={1.0}, text opacity={1.0}, font={{\fontsize{8 pt}{10.4 pt}\selectfont}}, text={rgb,1:red,0.0;green,0.0;blue,0.0}, cells={anchor={center}}, at={(0.98, 0.98)}, anchor={north east}}, axis background/.style={fill={rgb,1:red,1.0;green,1.0;blue,1.0}, opacity={1.0}}, anchor={north west}, xshift={1.0mm}, yshift={-1.0mm}, width={\textwidth}, height={\textwidth}, scaled x ticks={false}, xlabel={$l_1^{\phantom{2}}$}, x tick style={color={rgb,1:red,0.0;green,0.0;blue,0.0}, opacity={1.0}}, x tick label style={color={rgb,1:red,0.0;green,0.0;blue,0.0}, opacity={1.0}, rotate={0}}, xlabel style={at={(ticklabel cs:0.5)}, anchor=near ticklabel, at={{(ticklabel cs:0.5)}}, anchor={near ticklabel}, font={{\fontsize{11 pt}{14.3 pt}\selectfont}}, color={rgb,1:red,0.0;green,0.0;blue,0.0}, draw opacity={1.0}, rotate={0.0}}, xmajorgrids={true}, xmin={-2}, xmax={2}, xticklabels={{$-2$,$-1$,$0$,$1$,$2$}}, xtick={{-2.0,-1.0,0.0,1.0,2.0}}, xtick align={inside}, xticklabel style={font={{\fontsize{8 pt}{10.4 pt}\selectfont}}, color={rgb,1:red,0.0;green,0.0;blue,0.0}, draw opacity={1.0}, rotate={0.0}}, x grid style={color={rgb,1:red,0.0;green,0.0;blue,0.0}, draw opacity={0.1}, line width={0.5}, solid}, axis x line*={left}, x axis line style={color={rgb,1:red,0.0;green,0.0;blue,0.0}, draw opacity={1.0}, line width={1}, solid}, scaled y ticks={false}, ylabel={$l_2^{\phantom{2}}$}, y tick style={color={rgb,1:red,0.0;green,0.0;blue,0.0}, opacity={1.0}}, y tick label style={color={rgb,1:red,0.0;green,0.0;blue,0.0}, opacity={1.0}, rotate={0}}, ylabel style={at={(ticklabel cs:0.5)}, anchor=near ticklabel, at={{(ticklabel cs:0.5)}}, anchor={near ticklabel}, font={{\fontsize{11 pt}{14.3 pt}\selectfont}}, color={rgb,1:red,0.0;green,0.0;blue,0.0}, draw opacity={1.0}, rotate={0.0}}, ymajorgrids={true}, ymin={-2}, ymax={2}, yticklabels={{$-2$,$-1$,$0$,$1$,$2$}}, ytick={{-2.0,-1.0,0.0,1.0,2.0}}, ytick align={inside}, yticklabel style={font={{\fontsize{8 pt}{10.4 pt}\selectfont}}, color={rgb,1:red,0.0;green,0.0;blue,0.0}, draw opacity={1.0}, rotate={0.0}}, y grid style={color={rgb,1:red,0.0;green,0.0;blue,0.0}, draw opacity={0.1}, line width={0.5}, solid}, axis y line*={left}, y axis line style={color={rgb,1:red,0.0;green,0.0;blue,0.0}, draw opacity={1.0}, line width={1}, solid}, colorbar={false}, legend to name={leg:toy:restart}]
    \addplot[forget plot]
        graphics[xmin={-3}, xmax={3}, ymin={-2}, ymax={2}] {plots/lspace_bg_D=0.125.pdf}
        ;
    \addplot[color={rgb,1:red,0.6314;green,0.7882;blue,0.9569}, name path={42}, only marks, draw opacity={1.0}, line width={0}, solid, mark={*}, mark size={2.75 pt}, mark repeat={1}, mark options={color={rgb,1:red,0.0;green,0.0;blue,0.0}, draw opacity={0.0}, fill={rgb,1:red,0.6314;green,0.7882;blue,0.9569}, fill opacity={1.0}, line width={0.75}, rotate={0}, solid}]
        table[row sep={\\}]
        {
            \\
            -1.000000000002581  2.529703845881027e-11  \\
        }
        ;
    \addlegendentry {$L_{\mathrm{loc}}^\star/C_{\mathrm{loc}}^\star$}
    \addplot[color={rgb,1:red,1.0;green,0.7059;blue,0.5098}, name path={43}, only marks, draw opacity={1.0}, line width={0}, solid, mark={*}, mark size={2.75 pt}, mark repeat={1}, mark options={color={rgb,1:red,0.0;green,0.0;blue,0.0}, draw opacity={0.0}, fill={rgb,1:red,1.0;green,0.7059;blue,0.5098}, fill opacity={1.0}, line width={0.75}, rotate={0}, solid}]
        table[row sep={\\}]
        {
            \\
            0.885816136226263  -0.9429080681130648  \\
        }
        ;
    \addlegendentry {$L^\star/C^\star$}
    \addplot[color={rgb,1:red,0.7255;green,0.949;blue,0.9412}, name path={44}, draw opacity={1.0}, line width={1}, solid]
        table[row sep={\\}]
        {
            \\
            -2.0  0.0  \\
            -1.5011880876035293  0.05250651709436533  \\
            -0.6277494071643429  0.10436235787350281  \\
            -1.0012808468015608  0.13445722708457053  \\
            -1.0205595379046832  0.06186054008120073  \\
            -0.9968775373447396  0.0012091795273136002  \\
            -0.9999280119255713  -0.000618065698069401  \\
            -0.9999996855027661  1.4732422923919783e-7  \\
            -1.000000000002581  2.529703845881027e-11  \\
        }
        ;
    \addlegendentry {path}
    \addplot[color={rgb,1:red,0.7255;green,0.949;blue,0.9412}, name path={45}, draw opacity={1.0}, line width={1}, dashed, forget plot]
        table[row sep={\\}]
        {
            \\
            -1.000000000002581  2.529703845881027e-11  \\
            0.5246950765960862  0.5906358031916974  \\
        }
        ;
    \addplot[color={rgb,1:red,0.7255;green,0.949;blue,0.9412}, name path={46}, draw opacity={1.0}, line width={1}, solid, forget plot]
        table[row sep={\\}]
        {
            \\
            0.5246950765960862  0.5906358031916974  \\
            0.8424224805748928  -0.3678299747961704  \\
            0.8056796103928872  -0.5261296098161088  \\
            0.8902435309369684  -0.838524201972416  \\
            0.8928275651450612  -0.9360721282827451  \\
            0.8857150426730148  -0.9397159192854204  \\
            0.8858213865782921  -0.9429043080377618  \\
            0.8858161343678571  -0.9429080164271081  \\
            0.885816136226263  -0.9429080681130648  \\
        }
        ;
\end{axis}
\end{tikzpicture}

%% file: plots/cspace_restart_D=0.125.tikz
% Recommended preamble:
% \usetikzlibrary{arrows.meta}
% \usetikzlibrary{backgrounds}
% \usepgfplotslibrary{patchplots}
% \usepgfplotslibrary{fillbetween}
% \pgfplotsset{%
%     layers/standard/.define layer set={%
%         background,axis background,axis grid,axis ticks,axis lines,axis tick labels,pre main,main,axis descriptions,axis foreground%
%     }{
%         grid style={/pgfplots/on layer=axis grid},%
%         tick style={/pgfplots/on layer=axis ticks},%
%         axis line style={/pgfplots/on layer=axis lines},%
%         label style={/pgfplots/on layer=axis descriptions},%
%         legend style={/pgfplots/on layer=axis descriptions},%
%         title style={/pgfplots/on layer=axis descriptions},%
%         colorbar style={/pgfplots/on layer=axis descriptions},%
%         ticklabel style={/pgfplots/on layer=axis tick labels},%
%         axis background@ style={/pgfplots/on layer=axis background},%
%         3d box foreground style={/pgfplots/on layer=axis foreground},%
%     },
% }

\begin{tikzpicture}[/tikz/background rectangle/.style={fill={rgb,1:red,1.0;green,1.0;blue,1.0}, fill opacity={1.0}, draw opacity={1.0}}, show background rectangle]
\begin{axis}[point meta max={nan}, point meta min={nan}, legend cell align={left}, legend columns={1}, title={}, title style={at={{(0.5,1)}}, anchor={south}, font={{\fontsize{14 pt}{18.2 pt}\selectfont}}, color={rgb,1:red,0.0;green,0.0;blue,0.0}, draw opacity={1.0}, rotate={0.0}, align={center}}, legend style={color={rgb,1:red,0.0;green,0.0;blue,0.0}, draw opacity={1.0}, line width={1}, solid, fill={rgb,1:red,1.0;green,1.0;blue,1.0}, fill opacity={1.0}, text opacity={1.0}, font={{\fontsize{8 pt}{10.4 pt}\selectfont}}, text={rgb,1:red,0.0;green,0.0;blue,0.0}, cells={anchor={center}}, at={(1.02, 1)}, anchor={north west}}, axis background/.style={fill={rgb,1:red,1.0;green,1.0;blue,1.0}, opacity={1.0}}, anchor={north west}, xshift={1.0mm}, yshift={-1.0mm}, width={\textwidth}, height={\textwidth}, scaled x ticks={false}, xlabel={$\reduce{c}_1$}, x tick style={color={rgb,1:red,0.0;green,0.0;blue,0.0}, opacity={1.0}}, x tick label style={color={rgb,1:red,0.0;green,0.0;blue,0.0}, opacity={1.0}, rotate={0}}, xlabel style={at={(ticklabel cs:0.5)}, anchor=near ticklabel, at={{(ticklabel cs:0.5)}}, anchor={near ticklabel}, font={{\fontsize{11 pt}{14.3 pt}\selectfont}}, color={rgb,1:red,0.0;green,0.0;blue,0.0}, draw opacity={1.0}, rotate={0.0}}, xmajorgrids={true}, xmin={-0.5}, xmax={1.5}, xticklabels={{$-0.5$,$0.0$,$0.5$,$1.0$,$1.5$}}, xtick={{-0.5,0.0,0.5,1.0,1.5}}, xtick align={inside}, xticklabel style={font={{\fontsize{8 pt}{10.4 pt}\selectfont}}, color={rgb,1:red,0.0;green,0.0;blue,0.0}, draw opacity={1.0}, rotate={0.0}}, x grid style={color={rgb,1:red,0.0;green,0.0;blue,0.0}, draw opacity={0.1}, line width={0.5}, solid}, axis x line*={left}, x axis line style={color={rgb,1:red,0.0;green,0.0;blue,0.0}, draw opacity={1.0}, line width={1}, solid}, scaled y ticks={false}, ylabel={$\reduce{c}_2$}, y tick style={color={rgb,1:red,0.0;green,0.0;blue,0.0}, opacity={1.0}}, y tick label style={color={rgb,1:red,0.0;green,0.0;blue,0.0}, opacity={1.0}, rotate={0}}, ylabel style={at={(ticklabel cs:0.5)}, anchor=near ticklabel, at={{(ticklabel cs:0.5)}}, anchor={near ticklabel}, font={{\fontsize{11 pt}{14.3 pt}\selectfont}}, color={rgb,1:red,0.0;green,0.0;blue,0.0}, draw opacity={1.0}, rotate={0.0}}, ymajorgrids={true}, ymin={-0.5}, ymax={1.5}, yticklabels={{$-0.5$,$0.0$,$0.5$,$1.0$,$1.5$}}, ytick={{-0.5,0.0,0.5,1.0,1.5}}, ytick align={inside}, yticklabel style={font={{\fontsize{8 pt}{10.4 pt}\selectfont}}, color={rgb,1:red,0.0;green,0.0;blue,0.0}, draw opacity={1.0}, rotate={0.0}}, y grid style={color={rgb,1:red,0.0;green,0.0;blue,0.0}, draw opacity={0.1}, line width={0.5}, solid}, axis y line*={left}, y axis line style={color={rgb,1:red,0.0;green,0.0;blue,0.0}, draw opacity={1.0}, line width={1}, solid}, colorbar={false}]
    \addplot[forget plot]
        graphics[xmin={-1.0}, xmax={2.0}, ymin={-0.5}, ymax={1.5}] {plots/cspace_bg_D=0.125.pdf}
        ;
    \addplot[color={rgb,1:red,0.6314;green,0.7882;blue,0.9569}, name path={57}, only marks, draw opacity={1.0}, line width={0}, solid, mark={*}, mark size={2.75 pt}, mark repeat={1}, mark options={color={rgb,1:red,0.0;green,0.0;blue,0.0}, draw opacity={0.0}, fill={rgb,1:red,0.6314;green,0.7882;blue,0.9569}, fill opacity={1.0}, line width={0.75}, rotate={0}, solid}]
        table[row sep={\\}]
        {
            \\
            1.290745288429207e-12  1.0000000000051628  \\
        }
        ;
    \addplot[color={rgb,1:red,1.0;green,0.7059;blue,0.5098}, name path={58}, only marks, draw opacity={1.0}, line width={0}, solid, mark={*}, mark size={2.75 pt}, mark repeat={1}, mark options={color={rgb,1:red,0.0;green,0.0;blue,0.0}, draw opacity={0.0}, fill={rgb,1:red,1.0;green,0.7059;blue,0.5098}, fill opacity={1.0}, line width={0.75}, rotate={0}, solid}]
        table[row sep={\\}]
        {
            \\
            0.8352431817125443  0.3401324147424071  \\
        }
        ;
    \addplot[color={rgb,1:red,0.7255;green,0.949;blue,0.9412}, name path={59}, draw opacity={1.0}, line width={1}, solid]
        table[row sep={\\}]
        {
            \\
            1.0000000000000009  4.000000000000003  \\
            0.37618879337960665  2.2417863210868094  \\
            -0.11684004448457935  0.41893954385971194  \\
            0.0006412436850451453  1.0115165975746583  \\
            0.010491116251769128  1.0428192215584886  \\
            -0.0015563564411131825  0.9937674433290365  \\
            -3.599144607258076e-5  0.9998561977893499  \\
            -1.5724856733401182e-7  0.9999993710056654  \\
            1.290745288429207e-12  1.0000000000051628  \\
        }
        ;
    \addplot[color={rgb,1:red,0.7255;green,0.949;blue,0.9412}, name path={60}, draw opacity={1.0}, line width={1}, dashed]
        table[row sep={\\}]
        {
            \\
            1.290745288429207e-12  1.0000000000051628  \\
            0.40000000000012936  0.9000000000040012  \\
        }
        ;
    \addplot[color={rgb,1:red,0.7255;green,0.949;blue,0.9412}, name path={61}, draw opacity={1.0}, line width={1}, solid]
        table[row sep={\\}]
        {
            \\
            0.4000000000001295  0.9000000000040019  \\
            0.7760490581764244  0.4384759736603514  \\
            0.727399622497861  0.31251506333100365  \\
            0.8413885376560447  0.3515874888656384  \\
            0.84498431311396  0.349345012067961  \\
            0.8351030897451378  0.3400059188785736  \\
            0.8352504577488902  0.3401392411803649  \\
            0.8352431791371355  0.34013241232614927  \\
            0.8352431817125443  0.3401324147424071  \\
        }
        ;
\end{axis}
\end{tikzpicture}

%% file: plots/popov_D=0.125.tikz
% Recommended preamble:
% \usetikzlibrary{arrows.meta}
% \usetikzlibrary{backgrounds}
% \usepgfplotslibrary{patchplots}
% \usepgfplotslibrary{fillbetween}
% \pgfplotsset{%
%     layers/standard/.define layer set={%
%         background,axis background,axis grid,axis ticks,axis lines,axis tick labels,pre main,main,axis descriptions,axis foreground%
%     }{
%         grid style={/pgfplots/on layer=axis grid},%
%         tick style={/pgfplots/on layer=axis ticks},%
%         axis line style={/pgfplots/on layer=axis lines},%
%         label style={/pgfplots/on layer=axis descriptions},%
%         legend style={/pgfplots/on layer=axis descriptions},%
%         title style={/pgfplots/on layer=axis descriptions},%
%         colorbar style={/pgfplots/on layer=axis descriptions},%
%         ticklabel style={/pgfplots/on layer=axis tick labels},%
%         axis background@ style={/pgfplots/on layer=axis background},%
%         3d box foreground style={/pgfplots/on layer=axis foreground},%
%     },
% }

\begin{tikzpicture}[/tikz/background rectangle/.style={fill={rgb,1:red,1.0;green,1.0;blue,1.0}, fill opacity={1.0}, draw opacity={1.0}}, show background rectangle]
\begin{axis}[point meta max={nan}, point meta min={nan}, legend cell align={left}, legend columns={1}, title={}, title style={at={{(0.5,1)}}, anchor={south}, font={{\fontsize{14 pt}{18.2 pt}\selectfont}}, color={rgb,1:red,0.0;green,0.0;blue,0.0}, draw opacity={1.0}, rotate={0.0}, align={center}}, legend style={color={rgb,1:red,0.0;green,0.0;blue,0.0}, draw opacity={1.0}, line width={1}, solid, fill={rgb,1:red,1.0;green,1.0;blue,1.0}, fill opacity={1.0}, text opacity={1.0}, font={{\fontsize{8 pt}{10.4 pt}\selectfont}}, text={rgb,1:red,0.0;green,0.0;blue,0.0}, cells={anchor={center}}, at={(0.98, 0.98)}, anchor={north east}}, axis background/.style={fill={rgb,1:red,1.0;green,1.0;blue,1.0}, opacity={1.0}}, anchor={north west}, xshift={1.0mm}, yshift={-1.0mm}, width={0.7\textwidth}, height={0.39375\textwidth}, scaled x ticks={false}, xlabel={Frequency $\omega$}, x tick style={color={rgb,1:red,0.0;green,0.0;blue,0.0}, opacity={1.0}}, x tick label style={color={rgb,1:red,0.0;green,0.0;blue,0.0}, opacity={1.0}, rotate={0}}, xlabel style={at={(ticklabel cs:0.5)}, anchor=near ticklabel, at={{(ticklabel cs:0.5)}}, anchor={near ticklabel}, font={{\fontsize{11 pt}{14.3 pt}\selectfont}}, color={rgb,1:red,0.0;green,0.0;blue,0.0}, draw opacity={1.0}, rotate={0.0}}, xmode={log}, log basis x={10}, xmajorgrids={true}, xmin={0.1}, xmax={100.0}, xticklabels={{$10^{-1}$,$10^{0}$,$10^{1}$,$10^{2}$}}, xtick={{0.1,1.0,10.0,100.0}}, xtick align={inside}, xticklabel style={font={{\fontsize{8 pt}{10.4 pt}\selectfont}}, color={rgb,1:red,0.0;green,0.0;blue,0.0}, draw opacity={1.0}, rotate={0.0}}, x grid style={color={rgb,1:red,0.0;green,0.0;blue,0.0}, draw opacity={0.1}, line width={0.5}, solid}, axis x line*={left}, x axis line style={color={rgb,1:red,0.0;green,0.0;blue,0.0}, draw opacity={1.0}, line width={1}, solid}, scaled y ticks={false}, ylabel={$\Phi(\imath \omega)$}, y tick style={color={rgb,1:red,0.0;green,0.0;blue,0.0}, opacity={1.0}}, y tick label style={color={rgb,1:red,0.0;green,0.0;blue,0.0}, opacity={1.0}, rotate={0}}, ylabel style={at={(ticklabel cs:0.5)}, anchor=near ticklabel, at={{(ticklabel cs:0.5)}}, anchor={near ticklabel}, font={{\fontsize{11 pt}{14.3 pt}\selectfont}}, color={rgb,1:red,0.0;green,0.0;blue,0.0}, draw opacity={1.0}, rotate={0.0}}, ymajorgrids={true}, ymin={-0.44650748978106436}, ymax={3.283159159826205}, yticklabels={{$0$,$1$,$2$,$3$}}, ytick={{0.0,1.0,2.0,3.0}}, ytick align={inside}, yticklabel style={font={{\fontsize{8 pt}{10.4 pt}\selectfont}}, color={rgb,1:red,0.0;green,0.0;blue,0.0}, draw opacity={1.0}, rotate={0.0}}, y grid style={color={rgb,1:red,0.0;green,0.0;blue,0.0}, draw opacity={0.1}, line width={0.5}, solid}, axis y line*={left}, y axis line style={color={rgb,1:red,0.0;green,0.0;blue,0.0}, draw opacity={1.0}, line width={1}, solid}, colorbar={false}]
    \addplot[color={rgb,1:red,0.6314;green,0.7882;blue,0.9569}, name path={22}, draw opacity={1.0}, line width={1}, solid]
        table[row sep={\\}]
        {
            \\
            0.1  2.2517289120344888  \\
            0.10722672220103231  2.251987917474298  \\
            0.11497569953977356  2.2522857392772657  \\
            0.12328467394420663  2.252628199588652  \\
            0.1321941148466029  2.253021995372951  \\
            0.1417474162926805  2.253474830348722  \\
            0.15199110829529336  2.2539955669723692  \\
            0.16297508346206444  2.2545944015633537  \\
            0.17475284000076838  2.2552830661534866  \\
            0.1873817422860384  2.25607506121423  \\
            0.20092330025650468  2.2569859240823575  \\
            0.2154434690031884  2.258033538682386  \\
            0.23101297000831597  2.259238493053096  \\
            0.24770763559917108  2.2606244922472696  \\
            0.2656087782946686  2.2622188354136252  \\
            0.2848035868435802  2.2640529673153877  \\
            0.30538555088334157  2.266163116220226  \\
            0.32745491628777285  2.2685910320395832  \\
            0.35111917342151316  2.27138484082459  \\
            0.37649358067924676  2.274600034249096  \\
            0.40370172585965547  2.2783006155068977  \\
            0.43287612810830584  2.282560426041859  \\
            0.4641588833612779  2.287464680534524  \\
            0.4977023564332111  2.293111740220087  \\
            0.533669923120631  2.2996151562185654  \\
            0.5722367659350217  2.307106013861507  \\
            0.6135907273413173  2.315735603761327  \\
            0.657933224657568  2.3256784316440577  \\
            0.7054802310718643  2.337135549834368  \\
            0.7564633275546289  2.3503381366112523  \\
            0.8111308307896871  2.3655511441130295  \\
            0.8697490026177833  2.383076643361404  \\
            0.9326033468832199  2.4032561492073596  \\
            1.0  2.4264705882352944  \\
            1.0722672220103233  2.453135461229535  \\
            1.1497569953977358  2.483686751955884  \\
            1.2328467394420661  2.518549518450064  \\
            1.321941148466029  2.5580745370279594  \\
            1.4174741629268053  2.6024164226868445  \\
            1.5199110829529336  2.6513050119244372  \\
            1.6297508346206442  2.7036233119411253  \\
            1.7475284000076838  2.7566400748189084  \\
            1.8738174228603839  2.8046470566098574  \\
            2.0092330025650473  2.8366516700451405  \\
            2.154434690031884  2.8328627256582086  \\
            2.3101297000831598  2.760636285095548  \\
            2.4770763559917106  2.573818884412904  \\
            2.6560877829466865  2.225838550854255  \\
            2.8480358684358014  1.7075500140794668  \\
            3.0538555088334154  1.0902878558047384  \\
            3.2745491628777286  0.5091154035186576  \\
            3.511191734215131  0.07450582183989884  \\
            3.7649358067924683  -0.18660432038594443  \\
            4.0370172585965545  -0.30862100207077936  \\
            4.328761281083058  -0.3409508864902926  \\
            4.641588833612778  -0.323302077149055  \\
            4.977023564332112  -0.2813130244775288  \\
            5.3366992312063095  -0.22977753233321752  \\
            5.722367659350217  -0.17663203673721678  \\
            6.1359072734131725  -0.12584655751957285  \\
            6.579332246575681  -0.07920343542408215  \\
            7.054802310718643  -0.037318161425207164  \\
            7.564633275546289  -0.00020632869432501977  \\
            8.11130830789687  0.03240709806072925  \\
            8.697490026177835  0.06092151298633519  \\
            9.326033468832199  0.0857730054518128  \\
            10.0  0.1073897016472756  \\
            10.722672220103231  0.126170512229895  \\
            11.497569953977356  0.1424764910937245  \\
            12.32846739442066  0.15662885144752517  \\
            13.219411484660288  0.16891035682950042  \\
            14.174741629268055  0.17956828602853134  \\
            15.199110829529339  0.1888179957797574  \\
            16.297508346206442  0.1968465654133994  \\
            17.47528400007684  0.2038162640844697  \\
            18.73817422860384  0.20986772303487305  \\
            20.09233002565047  0.21512277249390696  \\
            21.544346900318832  0.21968694348363818  \\
            23.101297000831593  0.22365165495004274  \\
            24.770763559917114  0.227096115398487  \\
            26.560877829466868  0.23008897077023416  \\
            28.48035868435802  0.23268972961339623  \\
            30.538555088334157  0.23494999434838332  \\
            32.74549162877728  0.2369145245344711  \\
            35.11191734215131  0.23862215501419742  \\
            37.64935806792467  0.2401065889053998  \\
            40.37017258596554  0.2413970827481307  \\
            43.287612810830595  0.242519038740066  \\
            46.4158883361278  0.24349451691287574  \\
            49.770235643321115  0.2443426782961883  \\
            53.3669923120631  0.2450801685589173  \\
            57.22367659350217  0.2457214502807079  \\
            61.35907273413173  0.24627909086051392  \\
            65.79332246575679  0.24676401208843135  \\
            70.54802310718642  0.24718570656732092  \\
            75.64633275546291  0.2475524254518425  \\
            81.11308307896873  0.2478713413565272  \\
            86.97490026177834  0.24814868975621576  \\
            93.26033468832199  0.24838989174864756  \\
            100.0  0.24859966065919775  \\
        }
        ;
    \addlegendentry {$\Sigma$}
    \addplot[color={rgb,1:red,0.0;green,0.0;blue,0.0}, name path={23}, draw opacity={1.0}, line width={1}, dotted, forget plot]
        table[row sep={\\}]
        {
            \\
            0.0001  0.0  \\
            100000.0  0.0  \\
        }
        ;
    \addplot[color={rgb,1:red,0.5529;green,0.898;blue,0.6314}, name path={24}, only marks, draw opacity={1.0}, line width={0}, solid, mark={*}, mark size={3.75 pt}, mark repeat={1}, mark options={color={rgb,1:red,0.0;green,0.0;blue,0.0}, draw opacity={0.0}, fill={rgb,1:red,0.5529;green,0.898;blue,0.6314}, fill opacity={1.0}, line width={0.75}, rotate={0}, solid}]
        table[row sep={\\}]
        {
            \\
            4.328761281083058  -0.3409508864902926  \\
        }
        ;
    \addlegendentry {$\lambda_{\mathrm{min}}$}
    \addplot[color={rgb,1:red,1.0;green,0.6235;blue,0.6078}, name path={25}, draw opacity={1.0}, line width={1}, dashed]
        table[row sep={\\}]
        {
            \\
            0.1  2.5926797985247814  \\
            0.10722672220103231  2.5929388039645906  \\
            0.11497569953977356  2.5932366257675583  \\
            0.12328467394420663  2.5935790860789445  \\
            0.1321941148466029  2.5939728818632437  \\
            0.1417474162926805  2.5944257168390146  \\
            0.15199110829529336  2.594946453462662  \\
            0.16297508346206444  2.5955452880536463  \\
            0.17475284000076838  2.596233952643779  \\
            0.1873817422860384  2.5970259477045228  \\
            0.20092330025650468  2.59793681057265  \\
            0.2154434690031884  2.5989844251726786  \\
            0.23101297000831597  2.6001893795433886  \\
            0.24770763559917108  2.601575378737562  \\
            0.2656087782946686  2.603169721903918  \\
            0.2848035868435802  2.6050038538056803  \\
            0.30538555088334157  2.6071140027105186  \\
            0.32745491628777285  2.609541918529876  \\
            0.35111917342151316  2.6123357273148824  \\
            0.37649358067924676  2.6155509207393886  \\
            0.40370172585965547  2.6192515019971903  \\
            0.43287612810830584  2.6235113125321514  \\
            0.4641588833612779  2.6284155670248164  \\
            0.4977023564332111  2.6340626267103797  \\
            0.533669923120631  2.640566042708858  \\
            0.5722367659350217  2.6480569003517997  \\
            0.6135907273413173  2.6566864902516194  \\
            0.657933224657568  2.6666293181343503  \\
            0.7054802310718643  2.678086436324661  \\
            0.7564633275546289  2.691289023101545  \\
            0.8111308307896871  2.706502030603322  \\
            0.8697490026177833  2.7240275298516967  \\
            0.9326033468832199  2.7442070356976522  \\
            1.0  2.767421474725587  \\
            1.0722672220103233  2.7940863477198277  \\
            1.1497569953977358  2.8246376384461764  \\
            1.2328467394420661  2.8595004049403565  \\
            1.321941148466029  2.899025423518252  \\
            1.4174741629268053  2.943367309177137  \\
            1.5199110829529336  2.99225589841473  \\
            1.6297508346206442  3.044574198431418  \\
            1.7475284000076838  3.097590961309201  \\
            1.8738174228603839  3.14559794310015  \\
            2.0092330025650473  3.177602556535433  \\
            2.154434690031884  3.173813612148501  \\
            2.3101297000831598  3.1015871715858405  \\
            2.4770763559917106  2.9147697709031966  \\
            2.6560877829466865  2.566789437344548  \\
            2.8480358684358014  2.0485009005697594  \\
            3.0538555088334154  1.431238742295031  \\
            3.2745491628777286  0.8500662900089502  \\
            3.511191734215131  0.41545670833019144  \\
            3.7649358067924683  0.15434656610434816  \\
            4.0370172585965545  0.03232988441951323  \\
            4.328761281083058  0.0  \\
            4.641588833612778  0.01764880934123758  \\
            4.977023564332112  0.059637862012763776  \\
            5.3366992312063095  0.11117335415707508  \\
            5.722367659350217  0.1643188497530758  \\
            6.1359072734131725  0.21510432897071974  \\
            6.579332246575681  0.26174745106621045  \\
            7.054802310718643  0.30363272506508543  \\
            7.564633275546289  0.3407445577959676  \\
            8.11130830789687  0.37335798455102187  \\
            8.697490026177835  0.4018723994766278  \\
            9.326033468832199  0.42672389194210536  \\
            10.0  0.44834058813756816  \\
            10.722672220103231  0.4671213987201876  \\
            11.497569953977356  0.4834273775840171  \\
            12.32846739442066  0.49757973793781773  \\
            13.219411484660288  0.509861243319793  \\
            14.174741629268055  0.520519172518824  \\
            15.199110829529339  0.52976888227005  \\
            16.297508346206442  0.537797451903692  \\
            17.47528400007684  0.5447671505747623  \\
            18.73817422860384  0.5508186095251657  \\
            20.09233002565047  0.5560736589841996  \\
            21.544346900318832  0.5606378299739307  \\
            23.101297000831593  0.5646025414403353  \\
            24.770763559917114  0.5680470018887795  \\
            26.560877829466868  0.5710398572605268  \\
            28.48035868435802  0.5736406161036889  \\
            30.538555088334157  0.5759008808386759  \\
            32.74549162877728  0.5778654110247636  \\
            35.11191734215131  0.57957304150449  \\
            37.64935806792467  0.5810574753956924  \\
            40.37017258596554  0.5823479692384232  \\
            43.287612810830595  0.5834699252303586  \\
            46.4158883361278  0.5844454034031683  \\
            49.770235643321115  0.5852935647864809  \\
            53.3669923120631  0.5860310550492098  \\
            57.22367659350217  0.5866723367710005  \\
            61.35907273413173  0.5872299773508065  \\
            65.79332246575679  0.5877148985787239  \\
            70.54802310718642  0.5881365930576136  \\
            75.64633275546291  0.5885033119421351  \\
            81.11308307896873  0.5888222278468198  \\
            86.97490026177834  0.5890995762465083  \\
            93.26033468832199  0.5893407782389402  \\
            100.0  0.5895505471494903  \\
        }
        ;
    \addlegendentry {$\Sigma_{\mathrm{pert}}$}
\end{axis}
\end{tikzpicture}

%% file: plots/lspace_init_D=0.125.tikz
% Recommended preamble:
% \usetikzlibrary{arrows.meta}
% \usetikzlibrary{backgrounds}
% \usepgfplotslibrary{patchplots}
% \usepgfplotslibrary{fillbetween}
% \pgfplotsset{%
%     layers/standard/.define layer set={%
%         background,axis background,axis grid,axis ticks,axis lines,axis tick labels,pre main,main,axis descriptions,axis foreground%
%     }{
%         grid style={/pgfplots/on layer=axis grid},%
%         tick style={/pgfplots/on layer=axis ticks},%
%         axis line style={/pgfplots/on layer=axis lines},%
%         label style={/pgfplots/on layer=axis descriptions},%
%         legend style={/pgfplots/on layer=axis descriptions},%
%         title style={/pgfplots/on layer=axis descriptions},%
%         colorbar style={/pgfplots/on layer=axis descriptions},%
%         ticklabel style={/pgfplots/on layer=axis tick labels},%
%         axis background@ style={/pgfplots/on layer=axis background},%
%         3d box foreground style={/pgfplots/on layer=axis foreground},%
%     },
% }

\begin{tikzpicture}[/tikz/background rectangle/.style={fill={rgb,1:red,1.0;green,1.0;blue,1.0}, fill opacity={1.0}, draw opacity={1.0}}, show background rectangle]
\begin{axis}[point meta max={nan}, point meta min={nan}, legend cell align={left}, legend columns={1}, title={}, title style={at={{(0.5,1)}}, anchor={south}, font={{\fontsize{14 pt}{18.2 pt}\selectfont}}, color={rgb,1:red,0.0;green,0.0;blue,0.0}, draw opacity={1.0}, rotate={0.0}, align={center}}, legend style={color={rgb,1:red,0.0;green,0.0;blue,0.0}, draw opacity={1.0}, line width={1}, solid, fill={rgb,1:red,1.0;green,1.0;blue,1.0}, fill opacity={1.0}, text opacity={1.0}, font={{\fontsize{8 pt}{10.4 pt}\selectfont}}, text={rgb,1:red,0.0;green,0.0;blue,0.0}, cells={anchor={center}}, at={(1.02, 1)}, anchor={north west}}, axis background/.style={fill={rgb,1:red,1.0;green,1.0;blue,1.0}, opacity={1.0}}, anchor={north west}, xshift={1.0mm}, yshift={-1.0mm}, width={\textwidth}, height={\textwidth}, scaled x ticks={false}, xlabel={$l_1^{\phantom{2}}$}, x tick style={color={rgb,1:red,0.0;green,0.0;blue,0.0}, opacity={1.0}}, x tick label style={color={rgb,1:red,0.0;green,0.0;blue,0.0}, opacity={1.0}, rotate={0}}, xlabel style={at={(ticklabel cs:0.5)}, anchor=near ticklabel, at={{(ticklabel cs:0.5)}}, anchor={near ticklabel}, font={{\fontsize{11 pt}{14.3 pt}\selectfont}}, color={rgb,1:red,0.0;green,0.0;blue,0.0}, draw opacity={1.0}, rotate={0.0}}, xmajorgrids={true}, xmin={-2}, xmax={2}, xticklabels={{$-2$,$-1$,$0$,$1$,$2$}}, xtick={{-2.0,-1.0,0.0,1.0,2.0}}, xtick align={inside}, xticklabel style={font={{\fontsize{8 pt}{10.4 pt}\selectfont}}, color={rgb,1:red,0.0;green,0.0;blue,0.0}, draw opacity={1.0}, rotate={0.0}}, x grid style={color={rgb,1:red,0.0;green,0.0;blue,0.0}, draw opacity={0.1}, line width={0.5}, solid}, axis x line*={left}, x axis line style={color={rgb,1:red,0.0;green,0.0;blue,0.0}, draw opacity={1.0}, line width={1}, solid}, scaled y ticks={false}, ylabel={$l_2^{\phantom{2}}$}, y tick style={color={rgb,1:red,0.0;green,0.0;blue,0.0}, opacity={1.0}}, y tick label style={color={rgb,1:red,0.0;green,0.0;blue,0.0}, opacity={1.0}, rotate={0}}, ylabel style={at={(ticklabel cs:0.5)}, anchor=near ticklabel, at={{(ticklabel cs:0.5)}}, anchor={near ticklabel}, font={{\fontsize{11 pt}{14.3 pt}\selectfont}}, color={rgb,1:red,0.0;green,0.0;blue,0.0}, draw opacity={1.0}, rotate={0.0}}, ymajorgrids={true}, ymin={-2}, ymax={2}, yticklabels={{$-2$,$-1$,$0$,$1$,$2$}}, ytick={{-2.0,-1.0,0.0,1.0,2.0}}, ytick align={inside}, yticklabel style={font={{\fontsize{8 pt}{10.4 pt}\selectfont}}, color={rgb,1:red,0.0;green,0.0;blue,0.0}, draw opacity={1.0}, rotate={0.0}}, y grid style={color={rgb,1:red,0.0;green,0.0;blue,0.0}, draw opacity={0.1}, line width={0.5}, solid}, axis y line*={left}, y axis line style={color={rgb,1:red,0.0;green,0.0;blue,0.0}, draw opacity={1.0}, line width={1}, solid}, colorbar={false}]
    \addplot[forget plot]
        graphics[xmin={-3}, xmax={3}, ymin={-2}, ymax={2}] {plots/lspace_bg_D=0.125.pdf}
        ;
    \addplot[color={rgb,1:red,0.6314;green,0.7882;blue,0.9569}, name path={75}, only marks, draw opacity={1.0}, line width={0}, solid, mark={*}, mark size={2.75 pt}, mark repeat={1}, mark options={color={rgb,1:red,0.0;green,0.0;blue,0.0}, draw opacity={0.0}, fill={rgb,1:red,0.6314;green,0.7882;blue,0.9569}, fill opacity={1.0}, line width={0.75}, rotate={0}, solid}]
        table[row sep={\\}]
        {
            \\
            0.8409358077487007  -1.1892012038491426  \\
        }
        ;
    \addplot[color={rgb,1:red,1.0;green,0.7059;blue,0.5098}, name path={76}, only marks, draw opacity={1.0}, line width={0}, solid, mark={*}, mark size={2.75 pt}, mark repeat={1}, mark options={color={rgb,1:red,0.0;green,0.0;blue,0.0}, draw opacity={0.0}, fill={rgb,1:red,1.0;green,0.7059;blue,0.5098}, fill opacity={1.0}, line width={0.75}, rotate={0}, solid}]
        table[row sep={\\}]
        {
            \\
            0.885816136226263  -0.9429080681130648  \\
        }
        ;
    \addplot[color={rgb,1:red,0.7255;green,0.949;blue,0.9412}, name path={77}, draw opacity={1.0}, line width={1}, solid]
        table[row sep={\\}]
        {
            \\
            0.8409358077487007  -1.1892012038491426  \\
            0.8611070665254024  -1.185682713962255  \\
            0.8876598234324653  -1.1709511090987426  \\
            0.8956962181991118  -0.9599698358421274  \\
            0.8857124625287769  -0.9514614603989673  \\
            0.8858305846972045  -0.9429271524282457  \\
            0.8858161269436671  -0.9429084484387512  \\
            0.8858161362270195  -0.9429080681141557  \\
        }
        ;
\end{axis}
\end{tikzpicture}

%% file: plots/cspace_init_D=0.125.tikz
% Recommended preamble:
% \usetikzlibrary{arrows.meta}
% \usetikzlibrary{backgrounds}
% \usepgfplotslibrary{patchplots}
% \usepgfplotslibrary{fillbetween}
% \pgfplotsset{%
%     layers/standard/.define layer set={%
%         background,axis background,axis grid,axis ticks,axis lines,axis tick labels,pre main,main,axis descriptions,axis foreground%
%     }{
%         grid style={/pgfplots/on layer=axis grid},%
%         tick style={/pgfplots/on layer=axis ticks},%
%         axis line style={/pgfplots/on layer=axis lines},%
%         label style={/pgfplots/on layer=axis descriptions},%
%         legend style={/pgfplots/on layer=axis descriptions},%
%         title style={/pgfplots/on layer=axis descriptions},%
%         colorbar style={/pgfplots/on layer=axis descriptions},%
%         ticklabel style={/pgfplots/on layer=axis tick labels},%
%         axis background@ style={/pgfplots/on layer=axis background},%
%         3d box foreground style={/pgfplots/on layer=axis foreground},%
%     },
% }

\begin{tikzpicture}[/tikz/background rectangle/.style={fill={rgb,1:red,1.0;green,1.0;blue,1.0}, fill opacity={1.0}, draw opacity={1.0}}, show background rectangle]
\begin{axis}[point meta max={nan}, point meta min={nan}, legend cell align={left}, legend columns={4}, title={}, title style={at={{(0.5,1)}}, anchor={south}, font={{\fontsize{14 pt}{18.2 pt}\selectfont}}, color={rgb,1:red,0.0;green,0.0;blue,0.0}, draw opacity={1.0}, rotate={0.0}, align={center}}, 
legend style={/tikz/every even column/.append style={column sep=0.5cm}, color={rgb,1:red,0.0;green,0.0;blue,0.0}, draw opacity={1.0}, line width={1}, solid, fill={rgb,1:red,1.0;green,1.0;blue,1.0}, fill opacity={1.0}, text opacity={1.0}, font={{\fontsize{8 pt}{10.4 pt}\selectfont}}, text={rgb,1:red,0.0;green,0.0;blue,0.0}, cells={anchor={center}}, at={(0.98, 0.98)}, anchor={north east}}, axis background/.style={fill={rgb,1:red,1.0;green,1.0;blue,1.0}, opacity={1.0}}, anchor={north west}, xshift={1.0mm}, yshift={-1.0mm}, width={\textwidth}, height={\textwidth}, scaled x ticks={false}, xlabel={$\reduce{c}_1$}, x tick style={color={rgb,1:red,0.0;green,0.0;blue,0.0}, opacity={1.0}}, x tick label style={color={rgb,1:red,0.0;green,0.0;blue,0.0}, opacity={1.0}, rotate={0}}, xlabel style={at={(ticklabel cs:0.5)}, anchor=near ticklabel, at={{(ticklabel cs:0.5)}}, anchor={near ticklabel}, font={{\fontsize{11 pt}{14.3 pt}\selectfont}}, color={rgb,1:red,0.0;green,0.0;blue,0.0}, draw opacity={1.0}, rotate={0.0}}, xmajorgrids={true}, xmin={-0.5}, xmax={1.5}, xticklabels={{$-0.5$,$0.0$,$0.5$,$1.0$,$1.5$}}, xtick={{-0.5,0.0,0.5,1.0,1.5}}, xtick align={inside}, xticklabel style={font={{\fontsize{8 pt}{10.4 pt}\selectfont}}, color={rgb,1:red,0.0;green,0.0;blue,0.0}, draw opacity={1.0}, rotate={0.0}}, x grid style={color={rgb,1:red,0.0;green,0.0;blue,0.0}, draw opacity={0.1}, line width={0.5}, solid}, axis x line*={left}, x axis line style={color={rgb,1:red,0.0;green,0.0;blue,0.0}, draw opacity={1.0}, line width={1}, solid}, scaled y ticks={false}, ylabel={$\reduce{c}_2$}, y tick style={color={rgb,1:red,0.0;green,0.0;blue,0.0}, opacity={1.0}}, y tick label style={color={rgb,1:red,0.0;green,0.0;blue,0.0}, opacity={1.0}, rotate={0}}, ylabel style={at={(ticklabel cs:0.5)}, anchor=near ticklabel, at={{(ticklabel cs:0.5)}}, anchor={near ticklabel}, font={{\fontsize{11 pt}{14.3 pt}\selectfont}}, color={rgb,1:red,0.0;green,0.0;blue,0.0}, draw opacity={1.0}, rotate={0.0}}, ymajorgrids={true}, ymin={-0.5}, ymax={1.5}, yticklabels={{$-0.5$,$0.0$,$0.5$,$1.0$,$1.5$}}, ytick={{-0.5,0.0,0.5,1.0,1.5}}, ytick align={inside}, yticklabel style={font={{\fontsize{8 pt}{10.4 pt}\selectfont}}, color={rgb,1:red,0.0;green,0.0;blue,0.0}, draw opacity={1.0}, rotate={0.0}}, y grid style={color={rgb,1:red,0.0;green,0.0;blue,0.0}, draw opacity={0.1}, line width={0.5}, solid}, axis y line*={left}, y axis line style={color={rgb,1:red,0.0;green,0.0;blue,0.0}, draw opacity={1.0}, line width={1}, solid}, colorbar={false}, legend to name={leg:toy:init}]
    \addplot[forget plot]
        graphics[xmin={-1.0}, xmax={2.0}, ymin={-0.5}, ymax={1.5}] {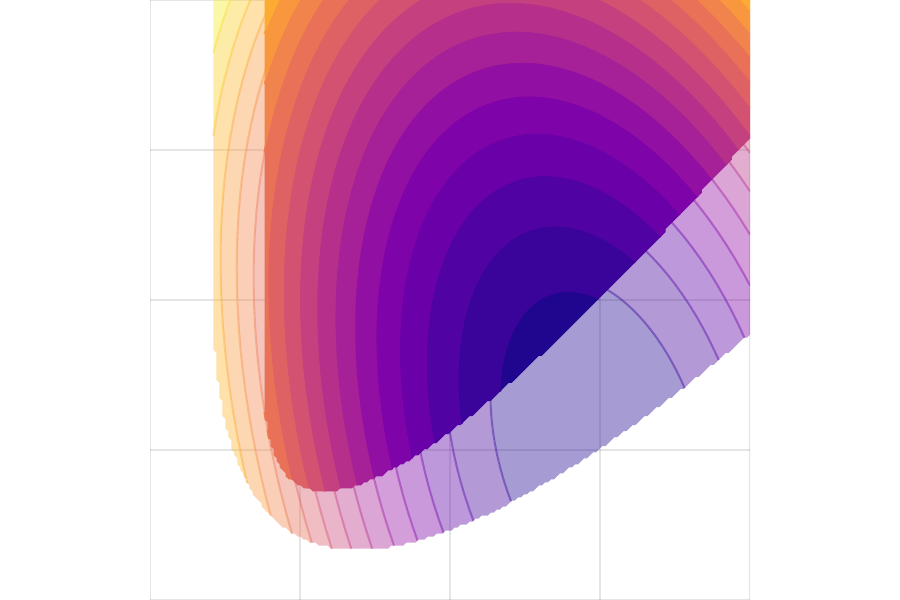}
        ;
    \addplot[color={rgb,1:red,0.6314;green,0.7882;blue,0.9569}, name path={68}, only marks, draw opacity={1.0}, line width={0}, solid, mark={*}, mark size={2.75 pt}, mark repeat={1}, mark options={color={rgb,1:red,0.0;green,0.0;blue,0.0}, draw opacity={0.0}, fill={rgb,1:red,0.6314;green,0.7882;blue,0.9569}, fill opacity={1.0}, line width={0.75}, rotate={0}, solid}]
        table[row sep={\\}]
        {
            \\
            0.7740544202513302  0.3196512449802108  \\
        }
        ;
    \addlegendentry {$L_0/\widehat{C}_0$}
    \addplot[color={rgb,1:red,1.0;green,0.7059;blue,0.5098}, name path={69}, only marks, draw opacity={1.0}, line width={0}, solid, mark={*}, mark size={2.75 pt}, mark repeat={1}, mark options={color={rgb,1:red,0.0;green,0.0;blue,0.0}, draw opacity={0.0}, fill={rgb,1:red,1.0;green,0.7059;blue,0.5098}, fill opacity={1.0}, line width={0.75}, rotate={0}, solid}]
        table[row sep={\\}]
        {
            \\
            0.8352431817125443  0.3401324147424071  \\
        }
        ;
    \addlegendentry {$L^\star/C^\star$}
    \addplot[color={rgb,1:red,0.5529;green,0.898;blue,0.6314}, name path={70}, only marks, draw opacity={1.0}, line width={0}, solid, mark={*}, mark size={2.75 pt}, mark repeat={1}, mark options={color={rgb,1:red,0.0;green,0.0;blue,0.0}, draw opacity={0.0}, fill={rgb,1:red,0.5529;green,0.898;blue,0.6314}, fill opacity={1.0}, line width={0.75}, rotate={0}, solid}]
        table[row sep={\\}]
        {
            \\
            1.0  0.0  \\
        }
        ;
    \addlegendentry {$C^{\phantom{\star}}$}
    \addplot[color={rgb,1:red,0.7255;green,0.949;blue,0.9412}, name path={71}, draw opacity={1.0}, line width={1}, solid]
        table[row sep={\\}]
        {
            \\
            0.7740544202513302  0.3196512449802108  \\
            0.8013062232726933  0.34108589030834824  \\
            0.8377998927843104  0.36832453016125233  \\
            0.8489839667476515  0.35313716446400023  \\
            0.8350995144037836  0.3400346548459766  \\
            0.8352632047410984  0.3401512005828975  \\
            0.8352431688485733  0.34013240267345185  \\
            0.8352431817135928  0.34013241474339007  \\
        }
        ;
    \addlegendentry {path}
\end{axis}
\end{tikzpicture}

%% file: plots/acc_optim_restart.tikz
% Recommended preamble:
% \usetikzlibrary{arrows.meta}
% \usetikzlibrary{backgrounds}
% \usepgfplotslibrary{patchplots}
% \usepgfplotslibrary{fillbetween}
% \pgfplotsset{%
%     layers/standard/.define layer set={%
%         background,axis background,axis grid,axis ticks,axis lines,axis tick labels,pre main,main,axis descriptions,axis foreground%
%     }{
%         grid style={/pgfplots/on layer=axis grid},%
%         tick style={/pgfplots/on layer=axis ticks},%
%         axis line style={/pgfplots/on layer=axis lines},%
%         label style={/pgfplots/on layer=axis descriptions},%
%         legend style={/pgfplots/on layer=axis descriptions},%
%         title style={/pgfplots/on layer=axis descriptions},%
%         colorbar style={/pgfplots/on layer=axis descriptions},%
%         ticklabel style={/pgfplots/on layer=axis tick labels},%
%         axis background@ style={/pgfplots/on layer=axis background},%
%         3d box foreground style={/pgfplots/on layer=axis foreground},%
%     },
% }

\begin{tikzpicture}[/tikz/background rectangle/.style={fill={rgb,1:red,1.0;green,1.0;blue,1.0}, fill opacity={1.0}, draw opacity={1.0}}, show background rectangle]
\begin{axis}[point meta max={nan}, point meta min={nan}, legend cell align={left}, legend columns={1}, legend style={color={rgb,1:red,0.0;green,0.0;blue,0.0}, draw opacity={1.0}, line width={1}, solid, fill={rgb,1:red,1.0;green,1.0;blue,1.0}, fill opacity={1.0}, text opacity={1.0}, font={{\fontsize{8 pt}{10.4 pt}\selectfont}}, text={rgb,1:red,0.0;green,0.0;blue,0.0}, cells={anchor={center}}, at={(0.98, 0.98)}, anchor={north east}}, axis background/.style={fill={rgb,1:red,1.0;green,1.0;blue,1.0}, opacity={1.0}}, anchor={north west}, xshift={1.0mm}, yshift={-1.0mm}, width={0.8\textwidth}, height={0.5\textwidth}, scaled x ticks={false}, xlabel={iteration number $k$}, x tick style={color={rgb,1:red,0.0;green,0.0;blue,0.0}, opacity={1.0}}, x tick label style={color={rgb,1:red,0.0;green,0.0;blue,0.0}, opacity={1.0}, rotate={0}}, xlabel style={at={(ticklabel cs:0.5)}, anchor=near ticklabel, at={{(ticklabel cs:0.5)}}, anchor={near ticklabel}, font={{\fontsize{11 pt}{14.3 pt}\selectfont}}, color={rgb,1:red,0.0;green,0.0;blue,0.0}, draw opacity={1.0}, rotate={0.0}}, xmajorgrids={true}, xmin={-0.8399999999999999}, xmax={28.84}, xticklabels={{$0$,$5$,$10$,$15$,$20$,$25$}}, xtick={{0.0,5.0,10.0,15.0,20.0,25.0}}, xtick align={inside}, xticklabel style={font={{\fontsize{8 pt}{10.4 pt}\selectfont}}, color={rgb,1:red,0.0;green,0.0;blue,0.0}, draw opacity={1.0}, rotate={0.0}}, x grid style={color={rgb,1:red,0.0;green,0.0;blue,0.0}, draw opacity={0.1}, line width={0.5}, solid}, axis x line*={left}, x axis line style={color={rgb,1:red,0.0;green,0.0;blue,0.0}, draw opacity={1.0}, line width={1}, solid}, scaled y ticks={false}, ylabel={$\|\transferFunction-\transferFunctionPas(\cdot;L_k)\|_{\calH_2}^2$}, y tick style={color={rgb,1:red,0.0;green,0.0;blue,0.0}, opacity={1.0}}, y tick label style={color={rgb,1:red,0.0;green,0.0;blue,0.0}, opacity={1.0}, rotate={0}}, ylabel style={at={(ticklabel cs:0.5)}, anchor=near ticklabel, at={{(ticklabel cs:0.5)}}, anchor={near ticklabel}, font={{\fontsize{11 pt}{14.3 pt}\selectfont}}, color={rgb,1:red,0.0;green,0.0;blue,0.0}, draw opacity={1.0}, rotate={0.0}}, ymode={log}, log basis y={10}, ymajorgrids={true}, ymin={0.6381192375482343}, ymax={296.8701815151773}, yticklabels={{$10^{0}$,$10^{1}$,$10^{2}$}}, ytick={{1.0,10.0,100.0}}, ytick align={inside}, yticklabel style={font={{\fontsize{8 pt}{10.4 pt}\selectfont}}, color={rgb,1:red,0.0;green,0.0;blue,0.0}, draw opacity={1.0}, rotate={0.0}}, y grid style={color={rgb,1:red,0.0;green,0.0;blue,0.0}, draw opacity={0.1}, line width={0.5}, solid}, axis y line*={left}, y axis line style={color={rgb,1:red,0.0;green,0.0;blue,0.0}, draw opacity={1.0}, line width={1}, solid}, colorbar={false}]
    \addplot[color={rgb,1:red,0.0;green,0.6056;blue,0.9787}, name path={17}, draw opacity={1.0}, line width={2}, solid, forget plot]
        table[row sep={\\}]
        {
            \\
            0.0  4.777116726784881  \\
            1.0  3.121644219071394  \\
            2.0  2.6853290250913115  \\
            3.0  1.6634351491953963  \\
            4.0  0.7704496205420134  \\
            5.0  0.7664382837391059  \\
            6.0  0.7652768551549213  \\
            7.0  0.7639460457218968  \\
            8.0  0.7619498972623014  \\
            9.0  0.7594315109878667  \\
            10.0  0.7592819100762584  \\
            11.0  0.7592796913276828  \\
            12.0  0.7592796900326457  \\
            13.0  0.7592796900319775  \\
        }
        ;
    \addplot[color={rgb,1:red,0.8889;green,0.4356;blue,0.2781}, name path={18}, draw opacity={1.0}, line width={2}, solid, forget plot]
        table[row sep={\\}]
        {
            \\
            0.0  10.054310391409116  \\
            1.0  3.460383241748076  \\
            2.0  2.6038372220260704  \\
            3.0  2.2242443142196233  \\
            4.0  1.7820357069385202  \\
            5.0  1.3593267414044898  \\
            6.0  1.2601395667936073  \\
            7.0  1.235665322350628  \\
            8.0  1.1886430009185354  \\
            9.0  1.183512151876317  \\
            10.0  1.1688322318308777  \\
            11.0  1.1535317263243174  \\
            12.0  1.153340886681183  \\
            13.0  1.153340495828687  \\
            14.0  1.1533404943980234  \\
        }
        ;
    \addplot[color={rgb,1:red,0.8889;green,0.4356;blue,0.2781}, name path={19}, draw opacity={1.0}, line width={2}, dashed, forget plot]
        table[row sep={\\}]
        {
            \\
            14.0  1.1533404884395688  \\
            15.0  1.1504192837399292  \\
            16.0  1.1376566767282619  \\
            17.0  0.898580448474807  \\
            18.0  0.7808964375347003  \\
            19.0  0.7648389314564203  \\
            20.0  0.7605794130068442  \\
            21.0  0.7594461480661213  \\
            22.0  0.7592874801496736  \\
            23.0  0.7592865880075489  \\
            24.0  0.7592797064888444  \\
            25.0  0.7592796900321318  \\
            26.0  0.7592796900319773  \\
        }
        ;
    \addplot[color={rgb,1:red,0.2422;green,0.6433;blue,0.3044}, name path={20}, draw opacity={1.0}, line width={2}, solid, forget plot]
        table[row sep={\\}]
        {
            \\
            0.0  249.49774946738347  \\
            1.0  44.46154862905379  \\
            2.0  3.8418041289008835  \\
            3.0  1.2035172026183731  \\
            4.0  1.1369056135386637  \\
            5.0  0.9850897334992258  \\
            6.0  0.7695664521245106  \\
            7.0  0.7657334995213699  \\
            8.0  0.7635397348524549  \\
            9.0  0.7631578164825632  \\
            10.0  0.7610358321698777  \\
            11.0  0.7593371962596064  \\
            12.0  0.7592800185011991  \\
            13.0  0.7592796906673496  \\
            14.0  0.7592796900325453  \\
        }
        ;
    \addplot[color={rgb,1:red,0.7644;green,0.4441;blue,0.8243}, name path={21}, draw opacity={1.0}, line width={2}, solid, forget plot]
        table[row sep={\\}]
        {
            \\
            0.0  6.44295700452157  \\
            1.0  4.9495478500002  \\
            2.0  1.5077275768774288  \\
            3.0  1.2497730498171087  \\
            4.0  1.1657410541626383  \\
            5.0  1.1637599350381438  \\
            6.0  1.1587774569098255  \\
            7.0  1.156922559256413  \\
            8.0  1.1561700028554691  \\
            9.0  1.153348612586663  \\
            10.0  1.1533405086148951  \\
            11.0  1.1533404943946948  \\
            12.0  1.153340494392641  \\
        }
        ;
    \addplot[color={rgb,1:red,0.7644;green,0.4441;blue,0.8243}, name path={22}, draw opacity={1.0}, line width={2}, dashed, forget plot]
        table[row sep={\\}]
        {
            \\
            12.0  1.1533404884346374  \\
            13.0  1.1504193576730748  \\
            14.0  1.1376571446613124  \\
            15.0  0.8985828708506802  \\
            16.0  0.7809038260157954  \\
            17.0  0.7648421432475169  \\
            18.0  0.7605803391246003  \\
            19.0  0.759446291638909  \\
            20.0  0.7592874868955148  \\
            21.0  0.7592865931823013  \\
            22.0  0.7592797065665033  \\
            23.0  0.7592796900321341  \\
            24.0  0.7592796900319774  \\
        }
        ;
    \addplot[color={rgb,1:red,0.6755;green,0.5557;blue,0.0942}, name path={23}, draw opacity={1.0}, line width={2}, solid, forget plot]
        table[row sep={\\}]
        {
            \\
            0.0  7.195803187396993  \\
            1.0  2.8975045786007807  \\
            2.0  2.7661464902944597  \\
            3.0  1.8816236404720683  \\
            4.0  1.308730626714211  \\
            5.0  1.2630630543240386  \\
            6.0  1.2134957764897087  \\
            7.0  1.1730546345033646  \\
            8.0  1.1552359115379187  \\
            9.0  1.1533731630154032  \\
            10.0  1.1533417175618694  \\
            11.0  1.1533405005741133  \\
            12.0  1.1533404944006325  \\
            13.0  1.1533404943926413  \\
        }
        ;
    \addplot[color={rgb,1:red,0.6755;green,0.5557;blue,0.0942}, name path={24}, draw opacity={1.0}, line width={2}, dashed, forget plot]
        table[row sep={\\}]
        {
            \\
            13.0  1.1533404884346372  \\
            14.0  1.1504193577508182  \\
            15.0  1.1376571451537516  \\
            16.0  0.8985828723313116  \\
            17.0  0.7809038340742146  \\
            18.0  0.764842146865541  \\
            19.0  0.7605803401141262  \\
            20.0  0.7594462917961897  \\
            21.0  0.7592874869053747  \\
            22.0  0.7592865931903207  \\
            23.0  0.7592797065665915  \\
            24.0  0.7592796900321341  \\
            25.0  0.7592796900319775  \\
        }
        ;
    \addplot[color={rgb,1:red,0.0;green,0.6658;blue,0.681}, name path={25}, draw opacity={1.0}, line width={2}, solid, forget plot]
        table[row sep={\\}]
        {
            \\
            0.0  17.242741175269813  \\
            1.0  9.75207261558936  \\
            2.0  2.415329347535924  \\
            3.0  1.0500468038917101  \\
            4.0  1.017444837869461  \\
            5.0  0.8813437627171351  \\
            6.0  0.8285130361492694  \\
            7.0  0.809768956345796  \\
            8.0  0.8087804460863647  \\
            9.0  0.780745810771043  \\
            10.0  0.7622982446557679  \\
            11.0  0.7598563895757743  \\
            12.0  0.7593023694177036  \\
            13.0  0.7592797744291223  \\
            14.0  0.7592796902991008  \\
            15.0  0.7592796900322453  \\
        }
        ;
    \addplot[color={rgb,1:red,0.9308;green,0.3675;blue,0.5758}, name path={26}, draw opacity={1.0}, line width={2}, solid, forget plot]
        table[row sep={\\}]
        {
            \\
            0.0  18.743490509276203  \\
            1.0  4.349808117322811  \\
            2.0  2.1053303188789507  \\
            3.0  1.1655507073029714  \\
            4.0  0.9333158758797452  \\
            5.0  0.866159072403977  \\
            6.0  0.8497794577735217  \\
            7.0  0.8303198152027775  \\
            8.0  0.7852052136559223  \\
            9.0  0.784269823733536  \\
            10.0  0.7615412299306975  \\
            11.0  0.7596433144981323  \\
            12.0  0.7592866778303549  \\
            13.0  0.7592797558469802  \\
            14.0  0.7592796900506811  \\
            15.0  0.7592796900321072  \\
        }
        ;
    \addplot[color={rgb,1:red,0.777;green,0.5097;blue,0.1464}, name path={27}, draw opacity={1.0}, line width={2}, solid, forget plot]
        table[row sep={\\}]
        {
            \\
            0.0  22.123839770589935  \\
            1.0  5.59814668136164  \\
            2.0  3.36209510710185  \\
            3.0  2.733333163771709  \\
            4.0  2.566202425555422  \\
            5.0  1.7611285605842713  \\
            6.0  1.3498752849428794  \\
            7.0  1.339015299178219  \\
            8.0  1.2669242951013973  \\
            9.0  1.2176432384281648  \\
            10.0  1.2104322640224823  \\
            11.0  1.1868360617136888  \\
            12.0  1.1597189411758255  \\
            13.0  1.1534013998065102  \\
            14.0  1.1533405585829692  \\
            15.0  1.1533404943951429  \\
            16.0  1.1533404943926437  \\
        }
        ;
    \addplot[color={rgb,1:red,0.777;green,0.5097;blue,0.1464}, name path={28}, draw opacity={1.0}, line width={2}, dashed, forget plot]
        table[row sep={\\}]
        {
            \\
            16.0  1.1533404884346337  \\
            17.0  1.1504193621638168  \\
            18.0  1.1376571731156486  \\
            19.0  0.8985829313854796  \\
            20.0  0.7809042983172404  \\
            21.0  0.7648423578783916  \\
            22.0  0.760580396653123  \\
            23.0  0.7594463008716283  \\
            24.0  0.7592874875295349  \\
            25.0  0.7592865937056121  \\
            26.0  0.7592797065717629  \\
            27.0  0.7592796900321338  \\
            28.0  0.7592796900319773  \\
        }
        ;
    \addplot[color={rgb,1:red,0.0;green,0.6643;blue,0.553}, name path={29}, draw opacity={1.0}, line width={2}, solid, forget plot]
        table[row sep={\\}]
        {
            \\
            0.0  3.2246562046021436  \\
            1.0  2.157199676067088  \\
            2.0  1.9916905935854265  \\
            3.0  1.8791448667314437  \\
            4.0  1.3371683461903838  \\
            5.0  1.1590322766343155  \\
            6.0  1.1587226378271307  \\
            7.0  1.157313133451624  \\
            8.0  1.1565312005805517  \\
            9.0  1.153388846196224  \\
            10.0  1.1533405997295019  \\
            11.0  1.1533404944432453  \\
            12.0  1.153340494392667  \\
        }
        ;
    \addplot[color={rgb,1:red,0.0;green,0.6643;blue,0.553}, name path={30}, draw opacity={1.0}, line width={2}, dashed, forget plot]
        table[row sep={\\}]
        {
            \\
            12.0  1.1533404884346592  \\
            13.0  1.1504193593114929  \\
            14.0  1.1376571550126362  \\
            15.0  0.8985829702997277  \\
            16.0  0.7809039772610769  \\
            17.0  0.7648422041270443  \\
            18.0  0.7605803589723508  \\
            19.0  0.7594462945503114  \\
            20.0  0.759287486926585  \\
            21.0  0.7592865931865511  \\
            22.0  0.7592797065679457  \\
            23.0  0.7592796900321339  \\
            24.0  0.7592796900319775  \\
        }
        ;
        \addplot[color={rgb,1:red,0.0;green,0.0;blue,0.0}, name path={16}, draw opacity={1.0}, line width={3}, solid]
        table[row sep={\\}]
        {
            \\
            0.0  1.5532544107284632  \\
            1.0  1.4443838940924507  \\
            2.0  1.3605431752297672  \\
            3.0  0.8505647533475232  \\
            4.0  0.7643660237920762  \\
            5.0  0.76276832847222  \\
            6.0  0.7619561359138816  \\
            7.0  0.7613018170955391  \\
            8.0  0.761280699093826  \\
            9.0  0.7594066797445216  \\
            10.0  0.7592798901514056  \\
            11.0  0.7592796914532011  \\
            12.0  0.7592796900324005  \\
        }
        ;
    \addlegendentry {initialization strategy}
\end{axis}
\end{tikzpicture}